\documentclass[final,onefignum,onetabnum]{class}
\usepackage{amssymb}
\usepackage{bbm}
\usepackage{mathtools}
\usepackage{braket,amsfonts,amsopn}
\usepackage[caption=false]{subfig}
\usepackage{pgfplots}
\usepackage{algorithmic}
\usepackage{graphicx,epstopdf}

\usepackage{booktabs}
\usepackage{multirow}
\usepackage{multicol} %side-by-side equations
\usepackage{comment}
\usepackage{color}
\usepackage[mathscr]{eucal}
\usepackage{changes}

\usepackage[numbers]{natbib}
\usepackage{bibentry}
\usepackage{adjustbox}
\usepackage{enumitem}
\usepackage{url}

\allowdisplaybreaks[4]

\newsiamremark{example}{Example}
\newsiamremark{remark}{Remark}

\DeclarePairedDelimiterX\sset[2]{\lbrace}{\rbrace}%
{ #1 \,\delimsize| \,\mathopen{} #2 }

\DeclareMathOperator*{\argmax}{arg\,max}
\DeclareMathOperator*{\esssup}{ess\,sup}

\makeatletter
\renewcommand\paragraph{\@startsection{paragraph}{4}{\z@}%
  {3.25ex \@plus1ex \@minus.2ex}%
  {-1em}%
  {\normalfont\normalsize\bfseries}}
\renewcommand\subparagraph{\@startsection{subparagraph}{5}{\parindent}%
  {3.25ex \@plus1ex \@minus .2ex}%
  {-1em}%
  {\normalfont\normalsize\bfseries}}
\def\toclevel@paragraph{4}
\def\toclevel@paragraph{5}
\def\l@paragraph{\@dottedtocline{5}{10em}{5em}}
\def\l@subparagraph{\@dottedtocline{6}{14em}{6em}}
\makeatother

\makeatother

\newcommand{\e}[1]{\mathbb{E} \left[ #1 \right]}
\newcommand{\ee}[2]{\mathbb{E}_{#1} \left[ #2 \right]}

\newcommand{\vvar}[2]{\mathrm{VaR}_{#1} \left[ #2 \right]}
\newcommand{\vvvar}[3]{\mathrm{VaR}^{#1}_{#2} \left[ #3 \right]}

\newcommand{\cccvar}[3]{\mathrm{CVaR}^{#1}_{#2} \left[ #3 \right]}
\newcommand{\ro}[1]{\rho \left[ #1 \right]}
\newcommand{\rro}[2]{\rho_{#1} \left[ #2 \right]}
\newcommand{\Var}[1]{\mathrm{Var} \left[ #1 \right]}
\newcommand{\VVar}[2]{\mathrm{Var}_{#1} \left[ #2 \right]}

\newcommand{\rrisk}[2]{\Cs{R}_{#1} \left[ #2 \right]}
\newcommand{\rregret}[2]{\Cs{V}_{#1} \left[ #2 \right]}

\newcommand{\bs}[1]{\boldsymbol{#1}} %for lower case
\newcommand{\Bs}[1]{\mathbb{#1}} %for upper case mathbb
\newcommand{\Ts}[1]{\mathbbmtt{#1}} %for upper case mathbbtt
\newcommand{\Cs}[1]{\mathcal{#1}} %for upper case mathcal
\newcommand{\Fs}[1]{\mathfrak{#1}} %for upper case mathfrak
\newcommand{\txi}{\tilde{\bs{\xi}}}
\newcommand{\tbs}[1]{\tilde{\bs{#1}}}
\newcommand{\ul}[1]{\underline{#1}}
\newcommand{\ol}[1]{\overline{#1}}
\newcommand{\one}[1]{\mathbbm{1}_{\{#1\}}}

\newcommand{\nomP}{\Ts{P}_{0}}
\newcommand{\trueP}{\Ts{P}^{\text{true}}}
\newcommand{\nomp}{p_{0}}

\newcommand{\pprobset}[2]{\Fs{P}\left( \Bs{R}^{#1} \times \Bs{R}^{#2}, \Fs{B}(\Bs{R}^{#1}) \times \Fs{B}(\Bs{R}^{#2}) \right)}
\newcommand{\measurespace}{\left( \Xi, \Cs{F} \right)}
\newcommand{\probspace}[1]{\left( \Bs{R}^{#1}, \Fs{B}(\Bs{R}^{#1}) \right)}
\newcommand{\promeasurespace}{\left( \Xi \times \Xi, \Cs{F} \times \Cs{F} \right)}

\newcommand{\Pspace}[1]{\left( \Xi, \Cs{F}, #1 \right)}
\newcommand{\M}{\Fs{M}_{+}(\Xi,\Cs{F})}
\renewcommand{\P}{\Fs{P}(\Bs{R}^{d},\Fs{B}(\Bs{R}^{d}))}

\newcommand{\conv}[1]{\text{conv}(#1)}  
\newcommand{\inte}[1]{\text{int}\left(#1\right)}  
\newcommand{\polar}[1]{\Cs{#1}^{\mathrm{o}}}
\newcommand{\dual}[1]{\Cs{#1}^{\prime}}

\newcommand{\dro}{DRO}   
\newcommand{\sip}{SIP}

\newcommand{\st}{\text{s.t.}}

\title{Distributionally Robust Optimization: A Review %of the Bridge Between Data and Decisions
}

\headers{Distributionally Robust Optimization: A Review}{H. Rahimian and S. Mehrotra}

\author{Hamed Rahimian\thanks{Department of Industrial Engineering and Management Sciences, Northwestern University, Evanston, IL 60208 
  (\email{hamed.rahimian@northwestern.edu}).}
\and Sanjay Mehrotra\thanks{Department of Industrial Engineering and Management Sciences, Northwestern University, Evanston, IL 60208 (\email{mehrotra@northwestern.edu}).}}

\ifpdf
\hypersetup{
  pdftitle={Distributionally Robust Optimization: A Review},
  pdfauthor={H. Rahimian and S. Mehrotra}
}
\fi

\begin{document}

\maketitle

\begin{abstract}
\noindent The concepts of risk-aversion, chance-constrained optimization, and robust optimization have developed significantly over the last decade. % in the operations research community. On the other hand, by relying on these concepts, among others, the 
Statistical learning community has also witnessed a rapid theoretical and applied growth by relying on these concepts.  A  modeling framework, called {\it distributionally robust optimization} (DRO),  has recently received significant  attention in both the operations research and statistical learning communities. This paper surveys main concepts and contributions to DRO, and its relationships with robust optimization, risk-aversion, chance-constrained optimization, and function regularization. %This paper also gives an overview of different applications that have used the DRO concept.
\end{abstract}



%
%
%
%\begin{center} August 2018 \end{center}

%\BCOLReport{17.02}%{Mathematical Programming}

%\pagebreak

\tableofcontents

%\pagebreak

% % % % % % % % % % % % % % % % % % % % % % % % % % % % % % % % % % % % % % % % % % % %	

%%%%%%%%%%%%%%%%%%%%%%%%%%%%%%%%%%%%%%%%%%%%%%%%%%%%%%%%%%%%%
\section{Introduction}
\label{sec: rev.intro}
%%%%%%%%%%%%%%%%%%%%%%%%%%%%%%%%%%%%%%%%%%%%%%%%%%%%%%%%%%%%%

Many real-world decision problems arising in engineering  and management  have uncertain parameters. This parameter uncertainty may be due to limited observability of data, noisy measurements, implementations and  prediction errors. %, or our inability to evaluate the decisions precisely.  
{\it Stochastic optimization} (SO) and (2) {\it robust optimization} frameworks have classically allowed  to model this uncertainty within a decision-making framework. Stochastic optimization assumes that the decision maker has {\it complete} knowledge about the underlying uncertainty through a {\it known} probability distribution and minimizes a functional of the cost, see, e.g., \citet{shapiro2014SP,birge2011SP}. The probability distribution of the random parameters is inferred from prior beliefs, expert opinions, errors in predictions based on the historical data (e.g., \citet{kim2015scheduling}), or a mixture of these. In robust optimization, on the other hand, it is assumed that the decision maker has no distributional knowledge  about the underlying uncertainty, except for its support, and the model minimizes the worst-case cost over an uncertainty set, see, e.g., \citet{elghaoui1997lsq,elghaoui1998SDP,ben1998robust,bertsimas2004price,ben2000robust,ben2009RO}. The concept of robust optimization has a relationship with chance-constrained optimization, where in certain cases there is a direct relationship between a robust optimization model and a chance-constrained optimization model, see, e.g., \citet[~pp157--158]{boyd2004CVX}.

% while SO does not  consider any  preference with respect to risk. %An  alternative modeling  approach  to  handle  uncertainty  is to use information on the probability distribution of the uncertain parameters. Alternatively, we may be interested in performing sensitivity analysis with respective to the underlying uncertain parameters.
We often have partial knowledge on the statistical properties of the model parameters.  Specifically, the probability distribution quantifying the model parameter uncertainty is known ambiguously. 
A typical approach  to handle this ambiguity, from a statistical point of view, is to  estimate the probability distribution using statistical tools, such as the maximal likelihood estimator, minimum Hellinger distance estimator \cite{vidyashankar2015}, or maximum entropy principle \cite{grunwald2004game}. The decision-making process can then be performed with respect to the estimated distribution. 
Because such an estimation may be imprecise, the impact of inaccuracy in estimation---and the subsequent ambiguity in the underlying distribution---is widely studied in the literature through (1) the  perturbation analysis of optimization problems, see, e.g., \citet{bonnans2013Perturbation}, (2) stability analysis of a SO model with respect to a change in the probability distribution, see, e.g, \citet{rachev1991,romisch2003}, or (3) input uncertainty analysis in stochastic simulation models, see, e.g., \citet{lam2016input} and references therein. The typical goals of  these approaches are to quantify the sensitivity of the optimal value/solution(s) to the probability distribution and provide continuity and/or large-deviation-type   results, see, e.g., \citet{dupavcova1990stability,schultz2000,heitsch2006stability,rachev2002quantitative,pflug2012distance}. While these approaches quantify the input uncertainty, they do not provide a systematic modeling framework to hedge against the ambiguity in the underlying probability distribution. % of uncertain parameters. 

{\it Ambiguous stochastic optimization} is a systematic modeling approach that bridges the gap between data and decision-making---statistics
and optimization frameworks---to protect the decision-maker from the  ambiguity in the underlying probability distribution. % of uncertain parameters. 
The ambiguous stochastic optimization approach  assumes that the underlying probability distribution is unknown and lies in an {\it ambiguity set} of probability distributions. As in robust optimization, this approach hedges against the ambiguity in probability distribution by taking a worst-case approach.  \citet{scarf1958} is arguably the first to consider such  an approach to obtain an  order quantity for a newsvendor problem to maximize the worst-case expected profit, where the worst-case is taken with respect to all product demand probability distributions with a known mean and variance. Since the seminal work of Scarf, and particularly in the past few years, significant research has been done on ambiguous stochastic optimization problems. This paper provides a review of the theoretical, modeling, and computational developments in this area. Moreover, we review the applications of the ambiguous stochastic optimization model that have been developed in the recent years. 
This paper also puts \dro\ in the context of risk-averse optimization, chance-constrained optimization, and robust optimization. 

\subsection{A General \dro\ Model}
\label{sec: rev.generic_model}

We  now formally introduce the model formulation that we discuss in this paper. Let $\bs{x} \in \Cs{X} \subseteq \Bs{R}^{n}$ be the decision vector.  On a measurable space $\measurespace$, let us define a random vector $\txi: \Xi \mapsto \Omega \subseteq \Bs{R}^{d}$% with a realization $\bs{\xi}:=[\xi_{1}, \ldots, \xi_{d}]$
, a random cost function $h(\bs{x},\txi): \Cs{X} \times \Xi \mapsto \Bs{R}$, and a vector of random functions $\bs{g}(\bs{x},\txi): \Cs{X} \times \Xi \mapsto \Bs{R}^{m}$, i.e., $\bs{g}(\bs{x}, \cdot):=[g_1(\bs{x}, \cdot),\ldots,g_m(\bs{x}, \cdot)]^{\top}$. %Also, assume $\M$ denotes the cone of nonnegative measures defined  on $\measurespace$, whereas $\P \subset \M$ denotes the set of probability measures defined on $\measurespace$.
%Also, assume $\probset{d}$ denotes the set of all probability measures on the measurable space $\probspace{d}$. %the cone  of nonnegative Borel measures on $\Bs{R}^{d}$. 
Given this setup, a general stochastic optimization problem has the form 
\begin{equation}
\label{eq: SO}
\inf_{\bs{x} \in \Cs{X} } \ \sset*{\rrisk{P} {h(\bs{x},\txi)}}{\rrisk{P}{\bs{g}(\bs{x},\txi)} \le \bs{0}}, \tag{\text{SO}}
\end{equation}
where $P$ denotes  the (known)  probability measure on $\measurespace$ %of the random vector $\txi$
and $\Cs{R}_{P}: \Cs{Z} \mapsto \Bs{R}$ denotes a (componentwise)  real-valued functional  under $P$, where $\Cs{Z}$ is a linear space of measurable functions on $\measurespace$. The functional $\Cs{R}_{P}$ accounts for quantifying the uncertainty in the outcomes of the decision, for a given fixed probability measure $P$. 
%Throughout the paper, we refer to $\rrisk{P}{\cdot}$ as a {\it risk measure}. %, implying \eqref{eq: SO} is a risk-averse SO).
This setup represents a broad range of problems in statistics, optimization, and control, such as regression and classification models \cite{friedman2016SL,james2013SL}, simulation-optimization \cite{fu2016SO,pasupathy2013SO}, stochastic optimal control \cite{bertsekas1995DP}, Markov decision processes \cite{puterman2005MDP}, and stochastic programming \cite{birge2011SP,shapiro2014SP}.  

As a special case of \eqref{eq: SO}, we have the classical stochastic programming problems:
\begin{equation}
	\label{eq: SO_Obj}
	\inf_{\bs{x} \in \Cs{X} } \ \ee{P}{h(\bs{x},\txi)}, %\tag{\text{R.-N.- SO  Obj.}}
\end{equation}
and
\begin{equation}
	\label{eq: SO_Cons}
	\inf_{\bs{x} \in \Cs{X} } \ \sset*{h(\bs{x})}{\ee{P}{\bs{g}(\bs{x},\txi)} \le \bs{0}}, %\tag{\text{R.- N.- SO Const.}}
\end{equation} 
where $\rrisk{P}{\cdot}$ is taken as the expected-value functional $\ee{P}{\cdot}$.  
%By taking $m=0$ in \eqref{eq: SO}, one can recover \eqref{eq: SO_Obj}. 
Note that by taking $h(\bs{x}, \cdot):=\mathbbm{1}_{A(\bs{x})} (\cdot)$ in \eqref{eq: SO_Obj}, where $\mathbbm{1}_{A(\bs{x})} (\cdot)$ denotes an indicator function for an arbitrary set $A(\bs{x}) \subseteq \Cs{B}(\Bs{R}^{d})$ (we define the indicator function and $\Cs{B}(\Bs{R}^{d})$ precisely in Section \ref{sec: rev.notation}), we obtain  the class of problems with a probabilistic objective function of the form  $P\{ \txi \in A(\bs{x})\}$, see, e.g., \citet{prekopa2003probabilistic}. The set  $A(\bs{x})$ is called a {\it safe region} and may be of the form $\bs{a}(\bs{x})^{\top} \txi \le \bs{b}(\bs{x})$ or $\bs{a}(\txi)^{\top} \bs{x} \le \bs{b}(\txi)$\footnote{We  say  a safe region of the form $\bs{a}(\bs{x})^{\top} \txi \le \bs{b}(\bs{x})$ is bi-affine in $\bs{x}$ and $\bs{\xi}$ if $\bs{a}(\bs{x}) $ and $\bs{b}(\bs{x})$ are both affine in $\bs{x}$. Similarly, we say a safe region of the  form  $\bs{a}(\txi)^{\top} \bs{x} \le \bs{b}(\txi)$ is bi-affine in $\bs{x}$ and $\bs{\xi}$ if $\bs{a}(\txi)$ and $\bs{b}(\txi)$ are both affine in $\txi$. Observe that a bi-affine safe region of the form  $\bs{a}(\bs{x})^{\top} \txi \le \bs{b}(\bs{x})$ can be equivalently written as a bi-affine safe region of the form  $\bs{a}(\txi)^{\top} \bs{x} \le \bs{b}(\txi)$, and vice versa.}. Similarly, by taking $h(\bs{x},\txi):=h(\bs{x})$ and $\bs{g}(\bs{x}, \cdot):=[\mathbbm{1}_{A_{1}(\bs{x})} (\cdot), \ldots, \mathbbm{1}_{A_{m}(\bs{x})} (\cdot)]^{\top}$,  for  suitable indicator functions $\mathbbm{1}_{A_{j}(\bs{x})} (\cdot)$, $j=1, \ldots, m$, \eqref{eq: SO_Cons}  is in the form of probabilistic (i.e., chance) constraints $P\{ \txi \in A_{j}(\bs{x})\}  \le 0, \; j=1, \ldots, m$, see, e.g., \citet{charnes1958chance,charnes1959chance,prekopa1970probabilistic,prekopa1974,dentcheva2006probabilistic}. Note that the case where the event $\{\txi \in A_{j}(\bs{x})\}$ is formed via  several constraints is called {\it joint chance constraint} as compared to {\it individual chance constraint}, where the event  $\{\txi \in A_{j}(\bs{x})\}$ is formed via  one constraint. 

A robust optimization model is defined as 
\begin{equation}
\label{eq: RO}
\inf_{\bs{x} \in \Cs{X} } \ \sup_{\bs{\xi} \in \Cs{U}} \ \sset*{h(\bs{x}, \bs{\xi})}{\sup_{\bs{\xi} \in \Cs{U}} \ \bs{g}(\bs{x},\bs{\xi}) \le \bs{0} }, \tag{\text{RO}}
\end{equation}
where $\Cs{U} \subseteq \Bs{R}^{d}$ denotes  an {\it uncertainty set} for the  parameters $\txi$. Similar to \eqref{eq: SO}, 
\begin{equation}
    \label{eq: RO_Obj}
    \inf_{\bs{x} \in \Cs{X} } \ \sup_{\bs{\xi} \in \Cs{U}} \ h(\bs{x}, \bs{\xi})
\end{equation} 
and 
\begin{equation}
    \label{eq: RO_Cons}
    \inf_{\bs{x} \in \Cs{X} } \ \sset*{h(\bs{x})}{\sup_{\bs{\xi} \in \Cs{U}} \ \bs{g}(\bs{x},\bs{\xi}) \le \bs{0} }
\end{equation}
are two special cases of \eqref{eq: RO}. 

% % % % % % % % % % % % % % % % % % % % % % % % % % % % % % % % % % % % % % % % % % % %
Problem \eqref{eq: SO}, as well as \eqref{eq: SO_Obj} and \eqref{eq: SO_Cons}, require the knowledge of the underlying measure $P$, whereas  \eqref{eq: RO}, as well as \eqref{eq: RO_Obj} and \eqref{eq: RO_Cons}, ignore all distributional knowledge of $\txi$, except for its support. 
An ambiguous version of \eqref{eq: SO} is formulated as 
\begin{equation}
\label{eq: DRO}
\inf_{\bs{x} \in \Cs{X} } \ \sup_{P \in \Cs{P}} \ 	\sset*{\rrisk{P}{h(\bs{x},\txi)}}{\sup_{P \in \Cs{P} } \ \rrisk{P }{\bs{g}(\bs{x},\txi)} \le \bs{0}} \tag{\text{DRO}}.
\end{equation}
Here, $\Cs{P}$ denotes the {\it ambiguity set of probability measures}, i.e., a family of measures consistent with the prior knowledge about uncertainty. %, and is a subset of all probability distributions on $\Xi$.
Note that if we consider the measurable space $(\Omega, \Cs{B})$, where $\Cs{B}$ denotes the Borel $\sigma$-field on $\Omega$, i.e., $\Cs{B}=\Omega \cap \Cs{B}(\Bs{R}^{d})$, %with $\Cs{B}(\Bs{R}^{d})$ is the Borel $\sigma$-field on $\Bs{R}^{d}$, 
then $\Cs{P}$ can be viewed as an ambiguity set of probability distributions $\Ts{P}$ defined on $(\Omega, \Cs{B})$ and induced by $\txi$\footnote{In this paper, we use $\Cs{P}$ to denote both an ambiguity set of probability measures  and an ambiguity set of  distributions induced by $\txi$. Whether  $\Cs{P}$  denotes  an ambiguity set of probability measures  or an ambiguity set of distributions induced by $\txi$ should be understood from the context and the distinction we make between the notation of a probability measure and a probability distribution.}. 

As discussed before, \eqref{eq: DRO} finds a decision that minimizes the worst-case of the functional  $\Cs{R}$ of the cost  $h$ among all probability measures in the ambiguity set provided that the (componentwise) worst-case of the functional $\Cs{R}$ of the function $\bs{g}$ is non-positive.
The  ambiguous versions of  \eqref{eq: SO_Obj} and \eqref{eq: SO_Cons} are  formulated as follows:
\begin{equation}
\label{eq: DRO_Obj}
\inf_{\bs{x} \in \Cs{X} } \ \sup_{P \in \Cs{P}} \ 	\ee{P}{h(\bs{x},\txi)},
\end{equation}
and
\begin{equation}
\label{eq: DRO_Cons}
\inf_{\bs{x} \in \Cs{X} } \ 	\sset*{h(\bs{x})}{\sup_{P \in \Cs{P} } \ \ee{P }{\bs{g}(\bs{x},\txi)} \le \bs{0}}.
\end{equation}
Models \eqref{eq: DRO_Obj} and \eqref{eq: DRO_Cons}  are  discussed in the
context of minimax stochastic optimization models, in which optimal solutions are evaluated under the worst-case expectation with respect  to a family of probability distributions of the uncertain
parameters, see, e.g.,  \citet{scarf1958}; \citet{zackova1966} (a.k.a. Dupa{\v{c}}ov{\'a}); \citet{dupacova1987,breton1995,shapiro2002minimax,shapiro2004minmax}. \citet{delage2010} refer to this approach as {\it distributionally robust optimization}, in short \dro, and since then, this terminology has become widely dominant in the  research community. We adopt this terminology, and for the rest of the paper, we refer to the ambiguous stochastic optimization of the form  \eqref{eq: DRO} as \dro. 

As mentioned before, \eqref{eq: DRO} is a modeling approach that assumes only partial distributional information, whereas \eqref{eq: SO} assumes complete distributional information. In fact, if $\Cs{P}$ contains only the  true distribution of the random vector $\txi$, \eqref{eq: DRO} reduces to \eqref{eq: SO}. On the other hand, if $\Cs{P}$ contains all probability distributions on the support of the random vector $\txi$, supported on $\Cs{U}$, then, \eqref{eq: DRO} reduces to \eqref{eq: RO}. Thus, a judicial choice of  $\Cs{P}$ can put \eqref{eq: DRO} between \eqref{eq: SO} and \eqref{eq: RO}. Consequently, \eqref{eq: DRO} may not be as conservative as \eqref{eq: RO}, which ignores all distributional information, except for the support $\Cs{U}$ of the uncertain parameters. 
\eqref{eq: DRO} can be viewed as a unifying framework for \eqref{eq: SO} and \eqref{eq: RO} (see also \citet{qian2018}). 

\subsection{Motivation and Contributions}
In this paper, we provide an overview of the main contributions to \dro\ within both operations research and machine learning communities. While there are separate review papers on RO, see, e.g., \cite{bertsimas2011REV,gabrel2014,gorissen2015}, to the best of our knowledge, there are a few tutorials and survey papers on \dro\ within the operations research community. A tutorial on \dro, its connection to risk-averse optimization, and the use of $\phi$-divergence to construct the ambiguity set is presented in \citet{bayraksan2015}. \citet{shapiro2018tutorial} provides a general tutorial on \dro\ and its  connection to risk-averse optimization. \citet{postek2016} surveys different papers that address distributionally robust risk constraints, with a variety of risk functional and ambiguity sets. Similar to \cite{bayraksan2015,shapiro2018tutorial,postek2016}, in this paper, we show the connection between \dro\ and risk aversion. However, the current review is different from those in the literature from a number of perspectives. 
We outline our contributions as follows:
\begin{itemize}
    \item We bring together the research done on \dro\ within the operations research and machine learning communities. This motivation is materialized throughout the  paper as we take a holistic view of \dro, from modeling, to solution techniques and to applications. %The framework \eqref{eq: DRO} is general enough to allow  the study of a wide range of problems arising within operations research and machine learning communities. 
    
    \item  We provide a detailed discussion on how  \dro\ models are connected to different concepts such as game theory, risk-averse optimization, chance-constrained optimization, robust optimization, and function regularization in statistical learning. 
    
    \item From the algorithmic perspective, we review techniques to solve a \dro\ model. %Moreover, we discuss the computational developments. 
    
    \item From the modeling and theoretical perspectives, we categorize different approaches to model the distributional ambiguity and discuss results for each of these ambiguity sets. % into four categories. This is unlike most papers that introduce two categories, namely moment-based and discrepancy-based  models. 
    Moreover, we discuss the calibration of different parameters used in these ambiguity sets of distributions. 
    
    %\item From the theoretical perspective, we provide a detailed discussion of the main contributions within each of these categories of the ambiguity sets, reviewed in this paper. 
    
   % \item From the applied perspective, we review papers in a wide variety of  application domains, from operations management and finance, to statistical learning, to energy and natural resources, to healthcare. 
\end{itemize}

\subsection{Organization of this Paper}
This paper is organized as follows. 
In Section \ref{sec: rev.notation}, we introduce the notation and the basic definitions. % that we use to in the paper.  
Section \ref{sec: rev.game_risk_chance_reg}  reviews the connection of \dro\ to different concepts: game theory in Section \ref{sec: rev.game_theory}, robust optimization in Section \ref{sec: rev.dro_ro}, risk-aversion and chance-constrained optimization with its relationship to robust optimization in Section \ref{sec: rev.rel_risk}, and regularization in statistical learning in Section \ref{sec: rev.rel_regularization}. In Section \ref{sec: rev.solution}, we review two main solution techniques to solve a \dro\ model by introducing tools in semi-infinite programming and duality. 
In Section \ref{sec: rev.choice.ambiguity}, we discuss different models to construct the ambiguity set of distributions. This includes discrepancy-based models in Section \ref{sec: rev.distance}, moment-based models in Section \ref{sec: rev.moment}, shape-preserving-based models in Section \ref{sec: rev.shape}, and kernel-based models in Section  \ref{sec: rev.kernel}. 
In Section \ref{sec: rev.calibration}, we discuss the calibration of different parameters used in the ambiguity set of distributions. In Section \ref{sec: rev.cost_inner}, we discuss different functionals  that amount for quantifying the uncertainty in the outcomes of a fixed decision. This includes  regret functions in Section \ref{sec: rev.regret}, risk measures in Section \ref{sec: rev.risk}, and utility functions in Section \ref{sec: rev.utility}. In Section \ref{sec: rev.toolboxes}, we introduce some modeling toolboxes for a \dro\ model. %In Section \ref{sec: rev.applications}, we review  a wide range of applications. %We end the paper in Section \ref{sec: rev.conclusion} with some conclusions and directions of future research. 

%%%%%%%%%%%%%%%%%%%%%%%%%%%%%%%%%%%%%%%%%%%%%%%%%%%%%%%%%%%%%%
\section{Notation and Basic Definitions}
\label{sec: rev.notation}
%%%%%%%%%%%%%%%%%%%%%%%%%%%%%%%%%%%%%%%%%%%%%%%%%%%%%%%%%%%%%%
%In the previous section, we introduced several notation. 
In this section, we introduce  additional  notation  used throughout the paper. In order to keep the paper self-contained, we also introduce all definitions used in this paper in this section.

For a given space $\Xi$ and a  $\sigma$-field $\Cs{F}$ of that space, we define an underlying measurable space $\measurespace$. In particular, let us define  $(\Bs{R}^{d}, \Cs{B}(\Bs{R}^{d}))$, where $\Cs{B}(\Bs{R}^{d})$ is the Borel $\sigma$-field on $\Bs{R}^{d}$. 
Let $\mathbbm{1}_{A}: \Xi \mapsto \{0,1\}$ indicate the indicator function of set  $A \in \Cs{F}$ where $\mathbbm{1}_{A}(s)=1$ if $s \in A$, and 0 otherwise.    Let $\Fs{M}_{+}(\cdot,\cdot)$ and $\Fs{M}(\cdot,\cdot)$ denote the set of all nonnegative measures and the set of all probability measures $Q: \Cs{F} \mapsto [0,1]$ defined on $\measurespace$, respectively. A measure $\nu_{2}$ is preferred over a measure $\nu_{1}$, denoted as $\nu_{2}  \succeq \nu_{1}$ if $\nu_{2}(A) \ge \nu_{1}(A)$ for all measurable sets $A \in \Cs{F}$. 
We denote by $Q\{A\}$ the probability of event $A \in \Cs{F}$, with respect to $Q \in \Fs{M}\measurespace$.
A random vector $\txi: \measurespace \mapsto (\Bs{R}^{d}, \Cs{B}(\Bs{R}^{d}))$ is always denoted with a tilde sign, while a realization of the random vector $\txi$ is denoted by the same symbol without a tilde, i.e., $\bs{\xi}$. 
For a probability measure $Q \in \Fs{M}\measurespace$, we define a probability space $\Pspace{Q}$. We denote  by $\Ts{Q}:=Q \circ \txi^{-1}$ the probability distribution induced by a random vector $\txi$ under $Q$, where $\txi^{-1}$ denotes the inverse image of $\txi$.  That is, $\Ts{Q} : \Cs{B}(\Bs{R}^{d}) \mapsto [0,1]$ is a probability distribution on $(\Bs{R}^{d}, \Cs{B}(\Bs{R}^{d}))$. 
Let  $\Fs{P}(\cdot,\cdot)$ denote the set of all such probability distributions. For example, $\Fs{P}(\Bs{R}^{d}, \Cs{B}(\Bs{R}^{d}))$ denotes the set of all probability distributions of $\txi$. %induced by a random vector $\txi$ on a space $(\cdot,\cdot)$.
Note that in our notation, we make a distinction between a probability measure $Q \in  \Fs{M}\measurespace$  and a  probability distribution $\Ts{Q} \in \P$.  Nevertheless, we  have always an appropriate transformation, so we might use the terminology of probability measure and probability distribution interchangeably.  Given this, for a function $f: \Bs{R}^{d} \mapsto \Bs{R}$, we may write  $\int_{\Xi} f(\txi(s)) Q(ds)$ equivalently as $\int_{\Bs{R}^{d}} f(s) \Ts{Q}(ds)$ with a change of measure. As we shall see later, we may denote $f(\txi(s))$ with $f(s)$ in this transformation. 
For two random variables $Z, Z^{\prime}: \Xi \mapsto \Bs{R}$, we use $Z \ge Z^{\prime}$ to denote $Z(s) \ge Z^{\prime}(s) $ almost everywhere (a.e.) on $\Xi$.  
A random variable $Z$ is $Q$-integrable if $ \|Z\|_{1}:=\int_{\Xi} |Z| d Q$ is finite. 
Two random variables $Z, Z^{\prime}$ are distributionally equivalent, denoted by $Z \overset{\text{d}}{\sim} Z^{\prime}$, if  they induce the same distribution, i.e., $Q\{Z \le z\}=Q\{Z^{\prime} \le z\}$. 
We also denote by $\Cs{S}(\Xi, \Cs{F})$ the collection of all $\Cs{F}$-measurable functions $Z: \measurespace \mapsto (\ol{\Bs{R}}, \Cs{B}(\overline{\Bs{R}}))$, where $\ol{\Bs{R}}$ denotes the extended real line $\Bs{R} \cup \{-\infty, +\infty\}$. 

For a finite space  $\Xi$ with $M$ atoms $\Xi=\{s_{1}, \ldots, s_{M}\}$ and $\Cs{F}=2^{\Xi}$, \linebreak let $\{q(s_{1}), \ldots, q(s_{M})\}$ be the probabilities of the  corresponding elementary events under probability measure $Q \in \Fs{M}\measurespace$. As a shorthand notation, we use $\bs{q}=[q_{1}, \ldots, q_{M}]^{T} \in \Bs{R}^{M}$, where $q_{i}:=q(s_{i})$, $i \in \{1, \ldots, M\}$. A $\Cs{F}$-measurable function $Z: \Xi \mapsto \Bs{R}$ has $M$   outcomes $\{Z(s_{1}), \ldots, Z(s_{M})\}$ with probabilities $\{q_{1}, \ldots, q_{M}\}$. For short, we identify $Z$ as  a vector in $\Bs{R}^{M}$, i.e., $\bs{z}=[z_{1}, \ldots, z_{M}]^{T}$ with $z_{i}:=Z(s_{i})$, $i \in \{1, \ldots, M\}$.

%For two probability measures $P, Q \in \Fs{M}\measurespace$, we say that $P$ is absolutely continuous with respect to $Q$ and denote it as $P \ll Q$ if and only if $Q(A)=0$ implies $P(A)=0$ for any $A \in \Cs{F}$. If $P \ll Q$, then $\frac{dP}{dQ}$ denotes the Radon-Nikodym derivative of $P$ with respect to $Q$. 

Consider a linear space $\Cs{V}$, paired with a dual linear space $\Cs{V}^{*}$, in the sense that  a (real-valued) bilinear form $\langle \cdot, \cdot \rangle: \Cs{V} \times \Cs{V}^{*} \mapsto \Bs{R}$  is defined. That is, for any $v \in \Cs{V}$ and $v^{*} \in \Cs{V}^{*}$,  we have that $\langle \cdot, v^{*} \rangle: \Cs{V}  \mapsto \Bs{R}$ and $\langle v, \cdot \rangle: \Cs{V}^{*}  \mapsto \Bs{R}$ are linear functionals on $\Cs{V}$ and $\Cs{V}^{*}$, respectively. Similarly, we define $\Cs{W}$ and $\Cs{W}^{*}$. For a linear mapping $A: \Cs{V} \mapsto \Cs{W}$, we define the adjoint mapping $A^{*}: \Cs{W}^{*} \mapsto \Cs{V}^{*}$ by means of the equation $\langle w^{*}, Av \rangle= \langle A^{*}w^{*}, v \rangle$, $\forall v \in \Cs{V}$. 
For two linear mappings, defined by finite dimensional matrices $A$ and $B$,  $A\bullet B=Tr(A^TB)$ denotes the Frobenius inner product between
matrices. Moreover, $\bs{A} \odot \bs{B}$ denotes the Hadamard (i.e., componentwise) product between matrices. 

For a  function $f: \Cs{V} \mapsto \ol{\Bs{R}}$, the (convex) conjugate $f^{*}: \Cs{V}^{*} \mapsto \ol{\Bs{R}}$ is defined as $f^{*}(v^{*})=\sup_{v \in \Cs{V}} \{\langle v^{*},v \rangle - f(v) \} $. Similarly, the biconjugate  $f^{**}: \Cs{V} \mapsto \ol{\Bs{R}}$ is defined as $f^{**}(v)=\sup_{v^{*} \in \Cs{V}^{*}} \{\langle v^{*},v \rangle - f^{*}(v^{*}) \} $. The characteristic function $\delta(\cdot|\Cs{A})$ of a nonempty set $\Cs{A} \in \Cs{V}$ is defined as $\delta(v|\Cs{A})=0$ if $v \in \Cs{A}$, and $+\infty$ otherwise. The support function of a   nonempty set $\Cs{A} \in \Cs{V}$   is defined as the convex conjugate of the characteristic function $\delta(\cdot|\Cs{A})$: $\delta^{*}(v^{*}|\Cs{V})=\sup_{v \in \Cs{V}} \{\langle v^{*},v \rangle -  \delta(v|\Cs{A}) \}= \sup_{v \in \Cs{V}} \langle v^{*},v \rangle  $.

%For $p \in  [1, \infty)$ and $Q \in \Fs{M}\measurespace$, let $\Cs{Z}:=\Cs{Z}_{p}(Q):=L_{p}\Pspace{Q}$
%be the linear space of all $\Cs{F}$-measurable functions  $Z$, having a finite $ \|Z\|_{p}^{p}$, where $\|Z\|_{p}$ denotes the $L_{p}$-norm and is defined as  $\Big( \int_{\Xi} |Z|^{p} d Q \Big)^{\frac{1}{p}} $. 
%For $Z, Z^{\prime} \in \Cs{Z}$, we use $Z \ge Z^{\prime}$ to denote $Z(\omega) \ge Z^{\prime}(\omega) $ almost everywhere (a.e.) on $\Xi$.  
%For the space  $\Cs{Z}_{p}(Q)$, there is an associated dual space $\Cs{Z}^{*}:=\Cs{Z}^{*}_{q}(Q):=L_{q}\Pspace{Q}$ of all $\Cs{F}$-measurable functions  $\zeta$, having a finite $\|Z\|_{q}^{q}$, such that $\frac{1}{p}+ \frac{1}{q}=1$, $\ q \in (1,\infty]$\footnote{For $p \in  [1, \infty)$, $\Cs{Z}_{p}(Q)$ and its dual space are both Banach spaces.}. For $Z \in \Cs{Z}$ and $\zeta \in \Cs{Z}^{*}$, let $\langle Z,\zeta \rangle:= \int_{\Xi} Z\zeta d Q$ denote their scalar product. 
%Let $\Cs{D}:=\sset*{\zeta \in \Cs{Z}^{*}}{\int_{\Xi} \zeta d Q=1, \; \zeta \ge 0, \ Q\text{-almost surely}} $ denote the set of all probability density functions (pdf) in $\Cs{Z}^{*}$. 

For  $Q \in \Fs{M}\measurespace$, let $\Cs{L}_{\infty}\Pspace{Q}$
be the linear space of all essentially bounded $\Cs{F}$-measurable functions  $Z$. A function $Z$ is essentially bounded if  \linebreak $ \|Z\|_{\infty}:=\esssup_{s \in \Omega} |Z(s)|$ is finite, where 
$$\esssup_{s \in \Xi} |Z(s)|:=\inf\Bigg\{\sup_{s \in \Xi} |Z^{\prime}(s)| \; \Big| \linebreak \; Z(s)=Z^{\prime}(s) \ \text{a.e.} \ s \in \Xi \Bigg\}.$$ %, or equivalently $Q\textrm{-}\esssup_{s \in \Xi} Z(s):=\inf\{a \in \Bs{R}: Q\{s \in \Xi: Z(s)>a)\}=0\}$.
%where $\|Z\|_{p}$ denotes the $L_{p}$-norm and is defined as  $\Big( \int_{\Xi} |Z|^{p} d Q \Big)^{\frac{1}{p}} $. 
%For $Z, Z^{\prime} \in \Cs{Z}$, we use $Z \ge Z^{\prime}$ to denote $Z(\omega) \ge Z^{\prime}(\omega) $ a.e. on $\Xi$.  
%For the space  $\Cs{Z}$, there is an associated dual space $\Cs{Z}^{*}:=L_{1}\Pspace{Q}$ of all $Q$-integrable $\Cs{F}$-measurable functions  $\zeta$. A function $\zeta$ is $Q$-integrable if $ \|\zeta\|_{1}:=\int_{\Xi} |\zeta| d Q$ is finite\footnote{Both $\Cs{Z}$ and $\Cs{Z}^{*}$ equppied with their corresponding norms are Banach spaces.}. %For $Z \in \Cs{Z}$ and $\zeta \in \Cs{Z}^{*}$, let $\langle Z,\zeta \rangle:= \int_{\Xi} Z\zeta d Q$ denote their scalar product. 
%Let $\Cs{D}:=\sset*{\zeta \in \Cs{Z}^{*}}{\int_{\Xi} \zeta d Q=1, \; \zeta \ge 0, \ Q\text{-almost surely}} $ denote the set of all probability density functions (pdf) in $\Cs{Z}^{*}$. 

We denote by $\|\cdot\|_{p}: \Bs{R}^{d} \mapsto \Bs{R}$ the $\ell_{p}$-norm on $\Bs{R}^{d}$. That is, for a vector $\bs{u} \in \Bs{R}^{d}$, $\|\bs{u}\|_{p}=\Big( \sum_{i=1}^{d} |u_{i}|^{p} \Big)^{\frac{1}{p}}$. %The $L_{p}$-norm on a functional space is denoted as $\|\cdot\|_{p}$ and is defined as   $\Big( \int_{\Xi} |Z|^{p} d P \Big)^{\frac{1}{p}} $. 
We use $\Delta^{d}$ to denote the simplex in $\Bs{R}^{d}$, i.e., $\Delta^{d}=\sset*{\bs{u} \in \Bs{R}^{d}}{\bs{e}^{\top}\bs{u}=1, \; \bs{u} \ge \bs{0}}$, where $\bs{e}$ is a vector of ones in $\Bs{R}^{d}$. 
Let $(\cdot)_{+}$ denote $\max\{0,\cdot\}$. 

For a proper cone $\Cs{K}$, the relation $x \preccurlyeq_{\Cs{K}} y$ indicates that $y-x \in \Cs{K}$. For simplicity, we drop $\Cs{K}$ from the notation, when $\Cs{K}$ is the positive semidefinite cone. Let $\Cs{S}_{+}^{n}$ denote the cone of symmetric positive semidefinite matrices in the $n \times n$ matrix spaces $\Bs{R}^{n \times n}$. For a cone $\Cs{K} \subset \Cs{V}$, we define its dual cone as $\dual{K}:=\sset*{v^{*} \in \Cs{V}^{*}}{\langle v^{*},v \rangle \ge 0, \; \forall v \in \Cs{K} }$. The negative of the dual cone is called polar cone and is denoted by $\polar{K}$. 
The $\Cs{K}$-epigraph of a function $\bs{f}: \Bs{R}^{N} \mapsto \Bs{R}^{M}$ and a proper
cone $\Cs{K}$ is conic-representable if the set  $\sset*{(\bs{x},\bs{y}) \in \Bs{R}^{N}\times \Bs{R}^{M}}{\bs{f}(\bs{x}) \preccurlyeq_{\Cs{K}} \bs{y}}$ can be expressed via conic inequalities,
possibly involving a cone different from $\Cs{K}$ and additional
auxiliary variables.

For a set $\Cs{K}$, we use  $\conv{\Cs{K}}$ and $\inte{\Cs{K}}$ to denote the convex hull and the interior of $\Cs{K}$, respectively. 

Because we  also review \dro\ papers in the context of statistical learning in this paper, we introduce some terminologies in statistical learning. For every approach that uses a set of (training) data to prescribe a solution or to  predict an outcome, it is important to assess the {\it out-of-sample} quality of the prescriber/predictor under a new set of (test) data, independent from the training set. Consider a given set of (training) data $\{\bs{\xi}^{i}\}_{i=1}^{N}$. Suppose that $\Ts{P}_{N}$ is the empirical probability distribution on $\{\bs{\xi}^{i}\}_{i=1}^{N}$. Data-driven approaches are interested in the performance of a data-driven solution (or, in-sample solution) $\hat{\bs{x}}_{N}$  that is constructed using $\{\bs{\xi}^{i}\}_{i=1}^{N}$. %The solution $\hat{\bs{x}}_{N}$ is also called an in-sample solution. 
A primitive data-driven solution for a problem of the form  \eqref{eq: SO_Obj} can be obtained by solving a {\it sample average approximation} (SAA)  of that problem, where the underlying distribution is chosen to be $\Ts{P}_{N}$ \cite{shapiro2014SP}.
Assessing the quality of this solution is well-studied in the context of SO, see, e.g., \citet{bayraksan2006quality,bayraksan2009quality,homem2014montecarlo}. Here, we introduce the analogous of such performance measure that are used to  assess the quality of a solution in the context of a \dro\ model. 
Let us focus on a \dro\ problem of the form  \eqref{eq: DRO_Obj} for the ease of exposition. Consider a data-driven solution $\bs{x}_{N} \in \Cs{X}$. Such a solution may be   obtained by solving a data-driven version of the \dro\ model \eqref{eq: DRO_Obj}, where the ambiguity set $\Cs{P}$ is constructed using data, namely $\Cs{P}_{N}$. 
The out-of-sample performance of $\bs{x}_{N}$ is defined as $\ee{\trueP}{h(\bs{x},\txi}$, which is the expected cost of $\bs{x}_{N}$ given  a new (test) sample that is independent of $\{\bs{\xi}^{i}\}_{i=1}^{N}$, drawn from an unknown true distribution $\trueP:=P^{\text{true}} \circ \txi^{-1}$. However, as $\trueP$ is unknown, one need to establish performance guarantees.  
One  such guarantee, referred to as {\it finite-sample performance guarantee} or {\it generalization bound} is defined as 
$$ \Ts{P}_{N} \left \lbrace \ee{\trueP}{h(\bs{x}_{N},\txi} \le \hat{V}_{N}   \right \rbrace \ge 1-\alpha,$$
which guarantees that an (in-sample) {\it certificate} $\hat{V}_{N}$ provides a $(1-\alpha)$ confidence (with respect to the training sample) on the  out-of-sample performance of $\bs{x}_{N}$. %In the context of SO, $\hat{V}_{N}$ can be chosen as the objective function of SAA problem evaluated at $\bs{x}_{N}$ (e.g., the optimal value of the SAA problem when $\bs{x}_{N}$ is optimal to SAA). 
The certificate $\hat{V}_{N}$ may  be chosen as the optimal value of the inner problem in \dro, where the worst-case is taken within $\Cs{P}_{N}$,  evaluated at $\bs{x}_{N}$, see, e.g.,  \cite{mohajerin2018}. % (e.g., the optimal value of the data-driven \dro\ model when $\bs{x}_{N}$ is optimal to the data-driven \dro\ model). 
The other guarantee, referred to as {\it asymptotic consistency}, guarantees that as $N$ increases, the certificate $\hat{V}_{N} $ and the data-driven solution $\bs{x}_{N}$ converges---in some sense---to the optimal value and an optimal solution of the true (unambiguous)  problem of the form  \eqref{eq: SO_Obj}, see, e.g.,  \cite{mohajerin2018}.

%%%%%%%%%%%%%%%%%%%%%%%%%%%%%%%%%%%%%%%%%%%%%%%%%%%%%%%%%%%%%%%%%%%%%%%%%%%%%%
\section[Relationship with Game Theory, Risk-Aversion, Chance-Constrained Optimization, and Regularization]{\texorpdfstring{Relationship with Game Theory, Risk-Aversion, \linebreak Chance-Constrained Optimization, and Regularization}{Relationship with Game Theory, Risk-Aversion, Chance-Constrained Optimization, and Regularization}}

\label{sec: rev.game_risk_chance_reg}
%%%%%%%%%%%%%%%%%%%%%%%%%%%%%%%%%%%%%%%%%%%%%%%%%%%%%%%%%%%%%%%%%%%%%%%%%%%%%%

%\alertHR{Discuss the motivation for DRO by connecting it the conservatism of RO. Disucss about reducing the conservatism of RO by Chance constraints.}
 
%%%%%%%%%%%%%%%%%%%%%%%%%%%%%%%%%%%%%
\subsection{Relationship with Game Theory}
\label{sec: rev.game_theory}
%%%%%%%%%%%%%%%%%%%%%%%%%%%%%%%%%%%%%

In this section, we present a game-theoretic interpretation of \dro. Indeed, a worst-case approach to SO may be viewed to have  its roots in John von Neumann's game theory. For ease of exposition, let us consider a problem of the form  \eqref{eq: DRO_Obj}. 

The decision maker, the first player in this setup, makes a decision $\bs{x} \in \Cs{X}$ whose consequences (i.e., cost $h$) depends on the outcome of the random vector $\txi$. The decision maker assumes  that $\txi$ follows some distribution $\Ts{P} \in \Cs{P}$. However, he/she does not know which   distribution the nature, the second player in this setup, will choose to represent the uncertainty in $\txi$. Thus, in one hand, the decision maker is looking for a decision that minimizes the maximum expected cost with respect to $\Cs{P}$, on the other hand, the nature is seeking a distribution that maximizes the minimum expected cost with respect to $\Cs{X}$.  Under suitable conditions, it can be shown that these two problems are the dual of each other and the solution to one problem provides the solution to the other problem. Such a solution $(\bs{x}^{*}, \Ts{P}^{*})$ is called an {\it equilibrium} or {\it saddle} point. In other words, at this point,  the decision maker would not change its decision $\bs{x}^{*}$, knowing that the nature chose $\Ts{P}^{*}$. Similarly, the nature would not change its distribution $\Ts{P}^{*}$, knowing that the decision maker chose $\bs{x}^{*}$. We state this result in the following theorem, which generalizes John von Neumann's minmax theorem. 

\begin{theorem}{(\citet[Theorem~3.4]{sion1958})}
    Suppose that 
    \begin{enumerate}[label=(\roman*)]
        \item $\Cs{X}$ and $\Cs{P}$ are convex and compact spaces,
        \item $\bs{x} \mapsto \rrisk{\Ts{P}}{h(\bs{x},\txi)}$ is upper semicontinuous and quasiconcave on $\Cs{P}$ for all $\bs{x} \in \Cs{X}$, and 
        \item $\Ts{P} \mapsto \rrisk{\Ts{P}}{h(\bs{x},\txi)}$ is lower semicontinuous and quasiconvex on $\Cs{X}$ for all $\Ts{P} \in \Cs{P}$. 
    \end{enumerate}
        Then, 
        $$ \inf_{\bs{x} \in \Cs{X} } \ \sup_{\Ts{P} \in \Cs{P}} \ 	\rrisk{\Ts{P}}{h(\bs{x},\txi)}= \sup_{\Ts{P} \in \Cs{P}} \  \inf_{\bs{x} \in \Cs{X} } \	\rrisk{\Ts{P}}{h(\bs{x},\txi)}.$$
\end{theorem}

According to the above theorem, under appropriate conditions, the exchange of the order between $\inf$ and $\sup$ will not change the optimal value to \linebreak $\inf_{\bs{x} \in \Cs{X} } \ \sup_{\Ts{P} \in \Cs{P}} \ 	\ee{\Ts{P}}{h(\bs{x},\txi)}$. We refer to \citet{grunwald2004game} for a variety of alternative regularity conditions for this to hold. The exchange of the order between $\inf$ and $\sup$ can be interpreted as follows \cite{grunwald2004game}: a probability distribution $\Ts{P}^{*}$ that maximizes the {\it generalized entropy} $\inf_{\bs{x} \in \Cs{X}} \ \rrisk{\Ts{P}}{h(\bs{x},\txi)}$ over $\Cs{P}$  has an associated decision $\bs{x}^{*}$, achieving  $\inf_{\bs{x} \in \Cs{X}} \ \rrisk{\Ts{P}^{*}}{h(\bs{x},\txi)}$,  and it achieves $\inf_{\bs{x} \in \Cs{X}} \ \sup_{\Ts{P} \in \Cs{P}} \ \rrisk{\Ts{P}}{h(\bs{x},\txi)}$.

%%%%%%%%%%%%%%%%%%%%%%%%%%%%%%%%%%%%%
\subsection{Relationship between DRO and  RO}
\label{sec: rev.dro_ro}
%%%%%%%%%%%%%%%%%%%%%%%%%%%%%%%%%%%%%
In Section \ref{sec: rev.intro}, we mentioned that when  the ambiguity set of probability distributions  contains all probability distributions on the support of the uncertain parameters, \dro\ and RO are equivalent. In this section, we present a different perspective on the relationship between \dro\ and RO under the assumption that the sample space $\Xi$ is finite. 
For ease of exposition, we focus on \eqref{eq: DRO_Obj}. A similar argument follows for \eqref{eq: DRO}. 

Suppose that $\Xi$ is a finite sample space with $M$ atoms,  $\Xi=\{s_{1}, \ldots, s_{M}\}$. Then, for a fixed $x \in \Cs{X}$, $h(\bs{x}, \txi)$ has $M$ possible outcomes $\{h(\bs{x}, \txi(s_{1})), \ldots, h(\bs{x}, \txi(s_{M})) \}$. For short, let us write these outcomes as a vector $\bs{h}(\bs{x}) \in \Bs{R}^{M}$, where $h_{m}(\bs{x}):=h(\bs{x}, \txi(s_{m}))$. In \eqref{eq: DRO_Obj}, $\Cs{P}$ is a subset of all probability measures on $\txi$. So, one can think of $\Cs{P}$ as a subset of all discrete probability distributions $\Ts{P}$ on $\Bs{R}^{d}$ induced by $\txi$. That is, $\Ts{P}$  can be identified with a vector $\bs{p} \in \Bs{R}^{M}$. Consequently, $\Cs{P}$ may be interpreted as a subset of $\Bs{R}^{M}$. 
With this interpretation, \eqref{eq: DRO_Obj} is written as 
\begin{equation}
\label{eq: rev.DRO_RO}
\inf_{\bs{x} \in \Cs{X} } \ \sup_{\bs{p} \in \Cs{P}} \ 	  \bs{p}^{\top}\bs{h}(\bs{x}).
\end{equation}
By defining $f(\bs{x}, \bs{p}):= \bs{p}^{\top}\bs{h}(\bs{x})$, we can rewrite the above problem as  \linebreak $\inf_{\bs{x} \in \Cs{X} } \ \sup_{\bs{p} \in \Cs{P}} \ f(\bs{x}, \bs{p}) $. This problem  has the form of \eqref{eq: RO_Obj}, where  the probability vector $\bs{p}$ takes values in an ``uncertainty set"  $\Cs{P}$. 
Techniques that are applicable for specifying the uncertainty set in a RO model may now be used to specify $\Cs{P}$ in \eqref{eq: rev.DRO_RO}, see, e.g., \citet{ben2001Convex,ben2000robust,bertsimas2004norm,chen2007robust}. We also refer to \citet{bertsimas2018RO} and Section \ref{sec: rev.rel_chance}.  For a through treatment of different nonlinear functions $f(\bs{x},\bs{p})$ and different uncertainty sets $\Cs{P}$, we refer to \citet{bental2015nonlinear}. However, as we shall see below, DRO has the richness that allows the use of techniques developed in the statistical literature to model the problem. Moreover, its framework allows $\Xi$ to be continuous. 
We also refer to \citet{xu2012optimization} for a distributional interpretation of RO. 

%%%%%%%%%%%%%%%%%%%%%%%%%%%%%%%%%%%%%
\subsection{Relationship with Risk-Aversion}
\label{sec: rev.rel_risk}
%%%%%%%%%%%%%%%%%%%%%%%%%%%%%%%%%%%%%

%%%%%%%%%%%%%%%%%%%%%%%%%%%%%%%%%%%%%
\subsubsection{Relationship between \dro\ and Coherent and Law Invariant Risk Measures}
\label{sec: rev.dro_coherent}
%%%%%%%%%%%%%%%%%%%%%%%%%%%%%%%%%%%%%

Under mild conditions (e.g., real-valued cost functions,  a convex and compact  ambiguity set), the worst-case expectations given in \eqref{eq: DRO_Obj} or \eqref{eq: DRO_Cons} are equivalent to  a {\it coherent} risk measure \citep{artzner1999,rockafellar2007,ruszczynski2006optimization}. 
Furthermore, under mild conditions, the worst-case expectations given in \eqref{eq: DRO_Obj} or \eqref{eq: DRO_Cons} are equivalent to a {\it law invariant} risk measure \cite{shapiro2017DRSP}.  These results imply that \dro\ models have an equivalent risk-averse optimization problem. In order to explain the relationship between \eqref{eq: DRO_Obj} and  \eqref{eq: DRO_Cons}  and risk-averse optimization more precisely, we  present some definitions and fundamental results. 

%Consider a probability space $\Pspace{P}$ and its associated space $\Cs{Z}=\Cs{Z}_{p}(P)$, for $p \in [1,\infty)$. 
%:=L_{p}\Pspace{\Ts{P}}$ be the space of measurable functions  $Z:\Xi \mapsto \Bs{R}$, having a finite $p$th order moment, i.e., $\int_{\Xi} |Z|^{p} d \Ts{P} < \infty$.   
%Define {\it risk measure} $\rho$ as a real-valued function on $\Cs{Z}$, i.e., . % MP
\begin{definition}{(\citet[Definition~2.4]{artzner1999}, \citet[Definition~6.4]{shapiro2014SP})}
\label{def: rev.coherent}
A (real-valued) risk measure $\rho: \Cs{Z} \mapsto \Bs{R}$ is called coherent if it  satisfies the following axioms: 
\begin{itemize}
    \item {\it Translation Equivariance:}  If $a \in \Bs{R}$ and $Z \in \Cs{Z}$, then $\rho(Z+a)=\rho(Z)+a$.
	\item {\it Positive Homogeneity:} If $t \ge 0$ and $Z \in \Cs{Z}$, then $\rho(tZ)=t\rho(Z)$. 
	\item {\it Monotonicity:} If $Z, Z^{\prime} \in \Cs{Z}$ and $Z \ge Z^{\prime}$, then $\rho(Z) \ge \rho(Z^{\prime})$.
	\item {\it Convexity:} $\rho\left( tZ+(1-t)Z^{\prime} \right) \le t\rho(Z) + (1-t) \rho(Z^{\prime})$, for all $Z, Z^{\prime} \in \Cs{Z}$ and all $t \in [0,1]$.
\end{itemize}
A risk measure $\rho$ is called convex if it satisfies all the above axioms besides the  positive homogeneity condition. 
\end{definition}

\begin{remark}
    In Definition \ref{def: rev.coherent}, the convexity axiom can be replaced with the {\it subadditivity} axiom: $\rho\left( Z+Z^{\prime} \right) \le \rho(Z) + \rho(Z^{\prime})$, for all $Z, Z^{\prime} \in \Cs{Z}$. This is true because the convexity and positive homogeneity axioms imply the subadditivity axiom, and conversely, the positive homogeneity and subadditivity axioms imply the convexity axiom. \citet[Definition~2.4]{artzner1999} defines a coherent risk measure with the subadditivity axiom, whereas \citet[Definition~6.4]{shapiro2014SP}  defines a coherent risk measure with the convexity axiom. 
\end{remark}

%For a random variable $Z \in \Cs{Z}$, some examples of coherent risk measures are:
%\begin{itemize}
%	\item {\it Expectation:} $\rho(Z)=\e{Z}$.
%	\item {\it Worst-Case:} $\rho(Z)=\esssup Z$, where $\esssup Z:=\inf\{a \in \Bs{R}: \Prob{Z > a}=0\}$ or equivalently,  $\esssup Z:=\inf\{\sup_{\omega \in \Omega} Z^{\prime}(\omega),\, : \, Z^{\prime}=Z \; $P$\textrm{-almost surely} \}$.
%	\item {\it Conditional Value-at-Risk} ($\cvart$){\it:} $\rho(Z)=\ccvar{\alpha}{Z}$, where $\ccvar{\alpha}{Z}:=\frac{1}{1-\alpha}\int_\alpha^1 \vvar{\tau}{Z}\,d\tau$ and $\vvar{\tau}{Z}:=\inf\{u\,:\,\Prob{Z \leq u}\geq \tau\}$ is the Value-at-Risk (VaR) at level $\tau$. 
%\end{itemize}

%\citet[Theorem~3.1]{shapiro2012minimax}

%\citet[Theorem~2.2]{ruszczynski2006optimization}
\begin{definition}{(\citet[Definition~2.1]{shapiro2017DRSP})}
\label{def: rev.law_invariant_measure}
A (real-valued) risk measure $\rho: \Cs{Z} \mapsto \Bs{R}$ is called law invariant if for all $Z, Z^{\prime} \in \Cs{Z}$,  $Z \overset{\text{d}}{\sim} Z^{\prime}$ implies that $\rho(Z)=\rho(Z^{\prime})$. 
\end{definition}

\begin{definition}{(\citet[Definition~2.2]{shapiro2017DRSP})}
\label{def: rev.law_invariant_set}
A set $\Cs{M}$ is called law invariant if $\zeta \in \Cs{M}$ and $\zeta \overset{\text{d}}{\sim} \zeta^{\prime}$ implies that $\zeta^{\prime} \in \Cs{M}$.
\end{definition}

To relate the worst-case expectation with respect to a set of probability distributions induced by $\txi$ to coherent risk measures, we adopt  the following result from \citet[Theorem~6.7]{shapiro2014SP}, \citet[Theorem~3.1]{shapiro2012minimax}. 

\begin{theorem}
    \label{thm: rev.duality_rho}
	Let $\Cs{Z}$ be the linear space of all essentially bounded $\Cs{F}$-measurable functions $Z: \Xi \mapsto \Bs{R}$ that are $P$-integrable for all $P \in \Fs{M}\measurespace$. Let  $\Cs{Z}^{*}$ be the space of all  signed measures $P$ on $\measurespace$ such that $\int_{\Xi} |d P| < \infty$. Suppose that $\Cs{Z}$ is paired with $\Cs{Z}^{*}$ such that the bilinear form $\ee{P}{Z}$ is well-defined. Moreover, suppose that $\Cs{Z}$ and $\Cs{Z}^{*}$ are equipped with the sup norm $\|\cdot\|_{\infty}$ and variation norm $\|\cdot\|_{1}$, respectively\footnote{Recall that  for a function $Z \in \Cs{Z}$, $\|Z\|_{\infty}=\esssup_{s \in \Omega} |Z(s)|$, where 
$\esssup_{s \in \Xi} |Z(s)|=\inf\Big\{\sup_{s \in \Xi} |Z^{\prime}(s)| \; \Big| \; Z(s)=Z^{\prime}(s) \ \text{a.e.} \ s \in \Xi \Big\}$. Also, for a measure $P \in \Cs{Z}^{*}$, $\|P\|_{1}=\int_{\Xi} |d P|  $. }. Recall $\Fs{M}\measurespace$ denotes the space of all probability measures on $\measurespace$:  $\Fs{M}\measurespace=\sset*{ P \in \Cs{Z}^{*}}{\int_{\Xi}  d P=1, \;  P \succcurlyeq 0} $. Let  $\rho: \Cs{Z} \mapsto \ol{\Bs{R}}$. Then, $\rho$  is a real-valued  
	coherent risk measure if and only if  % Then, $\rho$ is continuous (in the norm topology of $\Cs{Z}$) and 
	there exists a convex compact set $\Cs{M} \subseteq \Fs{M}\measurespace$ (in the weakly* topology of $\Cs{Z}^{*}$)  such that 
	\begin{equation}
	\label{eq: duality_rho}
	\rho(Z)= \sup_{P \in \Cs{M}} \ \ee{P}{Z}, \; \forall Z \in \Cs{Z}.
	\end{equation}
	Moreover, given a  real-valued coherent risk measure, the set $\Cs{M}$ in \eqref{eq: duality_rho} can be written in the form
	\begin{equation*}
	\Cs{M}= \sset*{P \in \Fs{M}\measurespace}{\ee{P}{Z} \le \rho(Z), \;  \forall Z \in \Cs{Z}}. 
	\end{equation*}
	%Conversely, if the representation \eqref{eq: duality_rho} holds for some convex compact set $\Cs{M} \subset  \Cs{D}$ (in the norm topology of $\Cs{Z}^{*}$), then $\rho$ is a real-valued coherent risk measure.
\end{theorem}

\begin{proof}
    First note that $\Cs{Z}$ is a Banach space, paired with the dual space  $\Cs{Z}^{*}$, which is also a Banach space. Then, by a similar proof to \citet[Theorem~6.7]{shapiro2014SP}, we can show that if $\rho$ is a proper and lower semicontinuous coherent risk measure, then \eqref{eq: duality_rho} holds when $\Cs{M}$ is equal to the subdifferential of $\rho$ at $0 \in \Cs{Z}$, i.e., $\Cs{M}=\partial \rho(0)$, 
    where 
    $$\partial \rho(Z)=\argmax_{P \in \Cs{M}} \ee{P}{Z}.$$ 
    Now, we show that $\rho$ is a proper and lower semicontinuous coherent risk measure. Consider the cone $\Cs{C} \subset \Cs{Z}$ of nonnegative functions $Z$. This cone is closed, convex, and pointed, and it defines a partial order relation on $\Cs{Z}$ that $Z \ge Z^{\prime}$ if and only if $Z(s) \ge Z^{\prime}(s) $ a.e.  on $\Xi$. We let the least upper bound of $Z, Z^{\prime}$ be $Z \vee  Z^{\prime}$, where $(Z \vee  Z^{\prime})(s)=\max\{Z(s), Z^{\prime}(s)\}$. It follows that $\Cs{Z}$ with cone $\Cs{C}$ forms a Banach lattice\footnote{It is said a partial order relation induces a {\it lattice structure} on $\Cs{Z}$ if the least upper bound exists for any $Z, Z^{\prime} \in \Cs{Z}$ \cite{shapiro2014SP}. A Banach space $\Cs{Z}$ with lattice structure is called {\it Banach lattice} if $Z, Z^{\prime} \in \Cs{Z}$ and $|Z| \ge |Z^{\prime}|$ implies $\|Z\| \ge \|Z^{\prime}\|$ \cite{shapiro2014SP}.}. Thus, by  \citet[Theorem~7.91]{shapiro2014SP}, we conclude that $\rho$ is continuous and subdifferentiable on the interior of its domain. This, in turn, implies that the lower semicontinuity of $\rho$ is automatically satisfied. Moreover, by \citet[Theorem~7.85]{shapiro2014SP}, the subdifferentials of $\rho$ at any point form a nonempty, convex, and weakly* compact subset of $\Cs{Z}^{*}$. 
    In particular, $\Cs{M}=\partial \rho(0)$ is a convex and weakly* compact set $\Cs{M} \subseteq \Fs{M}\measurespace$. 
    
    Conversely, suppose that \eqref{eq: duality_rho} holds with the set $\Cs{M}$ being  a convex and weakly* compact subset of $\Fs{M}\measurespace$. Then, $\rho$ is a real-valued coherent risk measure. 
    
    To prove the last part notice that for any $Z \in \Cs{Z}$, we have $\rho(Z) \ge \rho(0)+ \ee{P}{Z-0}$, for all $P \in \partial \rho(0)$. Now, by  the facts that $\Cs{M}=\partial \rho(0)$ and  $\rho(0)=0$, we conclude 
    $\Cs{M}= \sset*{P \in \Fs{M}\measurespace}{\ee{P}{Z} \le \rho(Z), \;  \forall Z \in \Cs{Z}}$.
\end{proof}

Before we proceed, let us characterize the set $\Cs{M}$, as described in Theorem \ref{thm: rev.duality_rho},  for three well-studied coherent risk measures, namely  {\it conditional value-at-risk} (CVaR), see, e.g., \citet{rockafellar2000,rockafellar2002,rockafellar2007}, convex combination of expectation and CVaR, see, e.g., \citet{zhang2016}, and {\it mean-upper-absolute semideviation}, see, e.g., \citet{shapiro2014SP}. CVaR at level $\beta$,  $0<\beta<1$, denoted by $\cccvar{Q}{\beta}{\cdot}$, is defined as  $\cccvar{Q}{\beta}{Z}:=\frac{1}{1-\beta}\int_\beta^1 \vvar{\alpha}{Z}\,d\alpha$, where $\vvvar{Q}{\alpha}{Z}:=\inf\sset*{u}{Q\{Z \leq u\}\geq \alpha}$ is the Value-at-Risk (VaR) at level $\alpha$. The mean-upper-absolute semideviation is  defined as $\ee{Q}{Z} + c\ee{Q}{(Z-\ee{P}{Z})_{+}}$, where $c \in [0,1]$.
\begin{example}
    \label{ex: rev.CVaR_dual}
    Consider a probability space $\Pspace{Q}$ and $\Cs{Z}=\Cs{L}_{\infty}\Pspace{Q}$. Suppose that $\Xi$ is a finite space with $M$ atoms. % as $\Xi=\{\omega_{1}, \ldots, \omega_{M}\}$. %Thus, $Z$ has  outcomes $\{Z(\omega_{1}), \ldots, Z(\omega_{M})\}$ with probabilities $\{p(\omega_{1}), \ldots, p(\omega_{M})\}$. For short, we identify $Z \in \Cs{Z}$ and $P Z^{-1}$ as  vectors in $\Bs{R}^{M}$, i.e., $Z=[Z_{1}, \ldots, Z_{m}]^{T}$ and $PZ^{-1}=[p_{1}, \ldots, p_{m}]^{T}$ with $Z_{m}=Z(\omega_{m})$ and $p_{m}=p(\omega_{m})$. 
    For a coherent risk measure $\rho$, we have $\rho(Z)= \sup_{\bs{p} \in \Cs{M}} \ \left \lbrace  \sum_{m=1}^{M} Z_{m} p_{m} \right \rbrace, \; \forall Z \in \Cs{Z}$,
    where $\Cs{M}$ is closed convex subset of $$\Cs{D}:=\sset*{\bs{p} \in \Bs{R}^{M}}{ \bs{p}^{\top}\bs{e}=1, \;  \bs{p} \ge \bs{0}},$$
    and $\bs{e}$ is a vector of ones. 
    \begin{itemize}
        \item When $\rho(Z)=\cccvar{Q}{\beta}{Z}$, we have $$\Cs{M}= \sset*{\bs{p} \in \Cs{D}}{ p_{m} \in [0, \frac{q_{m}}{1-\beta}], \; m=1, \ldots, M}.$$
        
        \item When $\rho(Z)=\ee{Q}{Z}+ \inf_{\tau \in \Bs{R}} \ \ee{Q}{(1-\gamma_{1})(\tau-Z)_{+} + (\gamma_{2}-1)(Z-\tau)_{+}}$, with $\gamma_{1} \in [0,1)$ and $\gamma_{2}>1$, we have 
        $$\Cs{M}= \sset*{\bs{p} \in \Cs{D}}{ p_{m} \in [q_{m}\gamma_{1}, q_{m}\gamma_{2} ], \; m=1, \ldots, M}.$$ 
        %where $\bs{q}^{*}$ denotes the Radon-Nikodym derivative of a reference measure $Q^{*}$ with respect to $Q$.
        The above risk measure is also equivalent to $\gamma_{1} \ee{Q}{Z}+ (1-\gamma_{1})\cccvar{Q}{\beta}{Z}$, where $\beta:=\frac{1-\gamma_{1}}{\gamma_{2}-\gamma_{1}}$. 
        
        \item When $\rho(Z)=\ee{Q}{Z} + c\ee{Q}{(Z-\ee{P}{Z})_{+}}$, we have 
        $$\Cs{M}= \sset*{\bs{p}^{\prime} \in \Cs{D}}{ \bs{p}^{\prime}= \bs{q} + \bs{\zeta} \odot \bs{q} - (\bs{\zeta}^{\top}\bs{q}) \odot \bs{q}, \; \|\bs{\zeta}\|_{\infty} \le c},$$
        where $\bs{a} \odot \bs{b}$ denotes the componentwise product of two vectors $\bs{a}$ and $\bs{b}$. 
        \qed
    \end{itemize}
    
\end{example}

Theorem \ref{thm: rev.duality_rho} relates problems \eqref{eq: DRO_Obj} and  \eqref{eq: DRO_Cons} to   risk-averse optimization problems, involving the coherent risk-measure $\rho$. 
Consider a fixed $\bs{x} \in \Cs{X}$. With an appropriate transformation of measure $\Ts{P}=P  \circ \txi^{-1}$, we can write  the inner problem $\sup_{\Ts{P} \in \Cs{P}} \ \ee{\Ts{P}}{h(\bs{x},\txi)}$ in \eqref{eq: DRO_Obj} as $\sup_{P \in \Cs{P}} \ \ee{P}{h(\bs{x},s)}$, where in the former, $\Cs{P}$ is a set of probability distributions induced by $\txi$, while in the latter, $\Cs{P}$ is a set of probability measures on $\measurespace$. Then, by applying Theorem \ref{thm: rev.duality_rho} and setting $Z=h(\bs{x},\txi)$, $\sup_{P \in \Cs{P}} \ \ee{P}{h(\bs{x},s)}$ evaluates a (real-valued) coherent risk measure $\ro{h(\bs{x},s)}$, provided that 
$\Cs{P} \subset \Fs{M}\measurespace$ is a convex compact  set. % in the associated dual space.  
It is easy to verify that  such a  function $\rho$ is coherent: 
\begin{itemize}
	\item {\it Translation Equivariance:}  Consider $\bs{x} \in \Cs{X}$ and $a \in \Bs{R}$. 
	Then, $ \ro{h(\bs{x},s)+ a}= \sup_{P \in \Cs{P}} \ \ee{P}{h(\bs{x},s)+a}= \sup_{P \in \Cs{P}} \ \ee{P}{h(\bs{x},s)} +a = \ro{h(\bs{x},s)} + a$. 
	
	\item {\it Positive Homogeneity:} Consider $\bs{x} \in \Cs{X}$ and $t \ge 0$. Then, $\ro{t  h(\bs{x},s)}= \sup_{P \in \Cs{P}} \ \ee{P}{t  h(\bs{x},s)}= t \sup_{P \in \Cs{P}} \ \ee{P}{h(\bs{x},s)}= t \ro{h(\bs{x},s)}$.

	\item {\it Monotonicity:} Consider $\bs{x}, \bs{x}^{\prime} \in \Cs{X}$ such that $h(\bs{x},s) \ge h(\bs{x}^{\prime},s)$. Thus, $\ee{P}{h(\bs{x},s)} \ge \ee{P}{h(\bs{x}^{\prime},s)}$ for any $P \in \Cs{P}$, which implies  \linebreak $\ro{h(\bs{x},s)}=\sup_{P \in \Cs{P}} \ \ee{P}{h(\bs{x},s)} \ge \sup_{P \in \Cs{P}} \ \ee{P}{h(\bs{x}^{\prime},s)}=\ro{h(\bs{x}^{\prime},s)}$.
	
	\item {\it Convexity:} Consider $\bs{x}, \bs{x}^{\prime} \in \Cs{X}$ and $t \in [0,1]$.
	Then, we have 
	\begin{align*}
	    \ro{t  h(\bs{x},s)+ (1-t)   h(\bs{x}^{\prime},s)} & = \sup_{P \in \Cs{P}} \ \ee{P}{t  h(\bs{x},s) + (1-t)  h(\bs{x}^{\prime},s)}\\
	    & \le \sup_{P \in \Cs{P}} \ \ee{P}{t   h(\bs{x},s)} +  \sup_{P \in \Cs{P}} \ \ee{P}{ (1-t)  h(\bs{x}^{\prime},s)} \\
	    & = t \sup_{P \in \Cs{P}} \ \ee{P}{ h(\bs{x},s)} +  (1-t) \sup_{P \in \Cs{P}} \ \ee{P}{h(\bs{x}^{\prime},s)}\\
	    & = t \ro{h(\bs{x},s)} + (1-t) \ro{h(\bs{x}^{\prime},s)},
	\end{align*}
	where we used the translation equivariance property. 
	
\end{itemize}

Consequently, \eqref{eq: DRO_Obj} is equivalent to minimizing a coherent risk measure. 
Similarly, \eqref{eq: DRO_Cons} is equivalent to a risk-averse optimization problem, subject to coherent risk constraints. 
Thus, a convex and compact ambiguity set of distributions gives rise to  a coherent risk measure.
Conversely, Theorem \ref{thm: rev.duality_rho} implies  that given a risk preference
that can be expressed in the form of a coherent risk
measure as a primitive, 
we can construct a corresponding convex and compact ambiguity set $\Cs{P}$ of probability distributions in a
\dro\ framework. Thus, the ambiguity set becomes a consequence of the particular
risk measure the decision maker selects.

It is worth noting that if $h$ is a convex random function in \eqref{eq: DRO_Obj}, i.e., $h(\cdot, \bs{\xi})$  is convex in $\bs{x}$ for almost every $\bs{\xi}$, then, $\ro{h(\cdot,\txi)}$ is convex in $\bs{x}$. Convexity of $\bs{g}$ in \eqref{eq: DRO_Cons} also implies the convexity of the  region induced by the risk constraints $\ro{\bs{g}(\cdot,\txi)} \le \bs{0}$. 
In our setup, neither $h(\cdot, \bs{\xi})$ nor $\bs{g}(\cdot, \bs{\xi})$ need to be convex as for example in the case where they are indicator functions. %As discussed before, these give rise to the class of probabilistic optimization problems. 

We now state the connection between the worst-case expectation with respect to a set of probability distributions induced by $\txi$ to law invariant risk measures. 
\begin{comment}
\begin{theorem}
	\label{thm: duality_rho_law}{\citet[Theorem~2.3]{shapiro2017DRSP}}
	Consider a probability space $\Pspace{Q}$, $\Cs{Z}$, $\Cs{Z}^{*}$, and $\Cs{D}$ as defined in Theorem \ref{thm: rev.duality_rho}.  %, and the dual space $\Cs{Z}^{*}:=\Cs{Z}_{q}(Q)$, $q \in (1,\infty]$.  
	Consider  $\rho: \Cs{Z} \mapsto \Bs{R}$, defined as 
	$\rho(Z)= \sup_{P \in \Cs{P}} \ee{P}{Z}= \sup_{\zeta \in \Cs{M}} \ \langle Z,\zeta \rangle, \; \forall Z \in \Cs{Z},$
	with respect to a set $\Cs{P}$ of absolutely continuous probability measures with respect to $Q$ and its associated set of probability density functions $\Cs{M}=\sset*{\zeta=\frac{dP}{dQ}}{P \in \Cs{P}} \subset \Cs{D}$. 
	If the set $\Cs{M}$ is law invariant, then the corresponding risk measure $\rho$ is law invariant. Conversely, if the  risk measure $\rho$ is law invariant, and the set  $\Cs{M}$ is convex and weakly* closed, then the set $\Cs{M}$ is law invariant. 
\end{theorem}
\end{comment}

\begin{theorem}
	\label{thm: duality_rho_law}{(\citet[Theorem~2.3]{shapiro2017DRSP})}
	Consider $\Cs{Z}$ and $\Cs{Z}^{*}$  as defined in Theorem \ref{thm: rev.duality_rho}.  %, and the dual space $\Cs{Z}^{*}:=\Cs{Z}_{q}(Q)$, $q \in (1,\infty]$.  
	Also, consider  $\rho: \Cs{Z} \mapsto \Bs{R}$, defined as 
	$\rho(Z)= \sup_{P \in \Cs{P}} \ee{P}{Z}, \; \forall Z \in \Cs{Z},$
	%with respect to a set $\Cs{P}$ of absolutely continuous probability measures with respect to $Q$ and its associated set of probability density functions $\Cs{M}=\sset*{\zeta=\frac{dP}{dQ}}{P \in \Cs{P}}$. 
	If the set $\Cs{P}$ is law invariant, then the corresponding risk measure $\rho$ is law invariant. Conversely, if the  risk measure $\rho$ is law invariant, and the set  $\Cs{P}$ is convex and weakly* closed, then the set $\Cs{P}$ is law invariant. 
\end{theorem}

For the connection between a general multistage \dro\ model, risk-averse multistage programming with conditional coherent risk mappings, and the concept of  time consistency of the problem and policies, we refer to \citet{shapiro2012minimax,shapiro2016rectangular,shapiro2018tutorial}.

\subsubsection{Relationship with Chance-Constrained Optimization}
\label{sec: rev.rel_chance}

In the previous section, we discussed how \dro\ is connected to risk-averse optimization. % through a proper choice of the ambiguity set $\Cs{P}$. 
In this section, we present another perspective that connects \dro\ to risk-averse optimization through a proper choice of the uncertainty set of the random variables $\txi$, as in RO. 

Many approaches in RO construct  the uncertainty set for the parameters $\txi$ such that the uncertainty set implies a probabilistic guarantee  with respect to the true unknown distribution. 
To explain how this construction is related to risk and \dro, consider the uncertain constraints $g(\bs{x}, \txi) \le 0$ for  a fixed $\bs{x}$. Suppose that $\txi$ belongs to a bounded uncertainty set $\Cs{U} \subseteq \Bs{R}^{d}$, i.e.,  $\Cs{U}$ is the support of $\txi$. The RO counterpart of this constraint then can be formulated as 
\begin{equation}
	\label{eq: RO_Cons1}
	g(\bs{x}, \bs{\xi}) \le 0, \; \forall \bs{\xi} \in \Cs{U}. 
\end{equation}
Two criticisms of \eqref{eq: RO_Cons1} are that: (1) it treats all uncertain parameters $\bs{\xi} \in \Cs{U}$  with equal weights and (2) all the parametrized constraints are hard, i.e., no violation is accepted. An alternative framework to reduce the conservatism caused by this approach is to use a chance constraint framework that allows a small probability of violation (with respect to the  probability distribution of $\txi$) instead of enforcing the constraint to be satisfied almost everywhere. % the flexibility that  at most a certain fraction of the constraints to be violated. 
Under the assumption that   $\txi $ is defined on a probability space $(\Xi, \Cs{F}, P^{\text{true}})$,  %distributed according to the probability distribution $\trueP$, 
the chance constraint framework can be represented as follows: 
\begin{equation}
	\label{eq: chance}
	P^{\text{true}}\{g(\bs{x}, \txi) \le 0 \} \ge 1-\epsilon,
\end{equation}  
for some $0<\epsilon <1 $. The parameter $\epsilon$ controls the risk of violating the uncertain constraint $g(\bs{x}, \txi) \le 0$. In fact,  as $\epsilon$ goes to zero, the set   $$\Cs{X}_{\epsilon}:=\sset*{\bs{x} \in \Cs{X} }{P^{\text{true}}\{g(\bs{x}, \txi) \le 0 \} \ge 1-\epsilon}$$  decreases to $$\Cs{X}(\Cs{U}):=\sset*{ \bs{x} \in \Cs{X} }{ g(\bs{x}, \bs{\xi}) \le 0, \; \forall \bs{\xi} \in \Cs{U} }.$$ 
%Moreover, \eqref{eq: chance} is equivalent to 
%\begin{equation}
%	\label{eq: chance_VaR}
%	\vvar{1-\epsilon}{\bs{g}(\bs{x}, \txi) } \le \bs{0},
%\end{equation}
%where $\vvar{1-\alpha}{\bs{g}(\bs{x}, \txi)}$ is taken with respect to $\trueP$.  %%:=\inf\sset*{\bs{u}}{\trueProb{\bs{g}(\bs{x}, \txi) \le \bs{u} } \ge \beta }$ is the Value-at-Risk (VaR) of $\bs{g}(\bs{x}, \txi) $ at level $\beta$ under $\trueP$.  

Motivated by the chance constraint framework \eqref{eq: chance}, many approaches in RO construct  an uncertainty set $\Cs{U}_{\epsilon}$ such that a feasible solution to a problem of the  form \eqref{eq: RO_Cons1} will  also  be  feasible with  probability  at least $1-\epsilon$ with respect to $P^{\text{true}}$. %solution it implies a probabilistic guarantee on the constraint $\bs{g}(\bs{x}, \txi) \le \bs{0}$ with respect to $\trueP$. 
More precisely, for any fixed $\bs{x}$, these constructions guarantee that  the following implication holds: 
\begin{equation}
\label{eq: rev.prob.guarantee}
	\text{If} \  g(\bs{x}, \bs{\xi}) \le 0, \ \forall \bs{\xi} \in \Cs{U}_{\epsilon}, \ \text{then,} \ P^{\text{true}}\{g(\bs{x}, \txi) \le 0 \} \ge 1-\epsilon. \tag{C1}
\end{equation}

However, as we argued before, the probability measure $P^{\text{true}}$ cannot be known with certainty. As far as it is relevant to the scope and interest of this paper, there are two streams of research in order to handle the ambiguity about the true probability distribution and obtain a safe (or, conservative) approximation\footnote{A set of constraints is called a safe or conservative approximation of the chance constraint if the feasible region induced by the approximation is a subset of the feasible region induced by the chance constraint.} to \eqref{eq: chance}\footnote{There is another stream of research that approximates \eqref{eq: chance} by CVaR or its approximations, see, e.g., \citet{chen2007robust,chen2009goal,chen2010joint} and references there in.}:   (1) scenario approximation scheme of \eqref{eq: RO_Cons1} based on Monte Carlo sampling, see, e.g., \citet{campi2004,calafiore2005,nemirovski2006scenario,campi2008,luedtke2008chance,bental2009LMI}, and (2) \dro\ approach to \eqref{eq: chance}, see, e.g., \citet{nemirovski2006convex,erdougan2006}.  
Research on scenario approximation of \eqref{eq: RO_Cons1} focuses on providing probabilistic guarantee (with respect to the sample probability measure) that a solution to the sampled problem of \eqref{eq: RO_Cons1} is feasible to \eqref{eq: chance} with a high probability. % or optimal to  \eqref{eq: chance} with a high probability. 

The \dro\ approach, on the other hand, forms a version of \eqref{eq: chance} as follows: 
\begin{equation}
\label{eq: chance_DRO}
P\{g(\bs{x}, \txi) \le 0 \} \ge 1-\epsilon, \; \forall P \in \Cs{P} \equiv \inf_{P \in \Cs{P}} \ P\{g(\bs{x}, \txi) \le 0 \} \ge 1-\epsilon. 
\end{equation}  
Let $\bar{\Cs{X}}_{\epsilon}$ denote the feasibility set induced by \eqref{eq: chance_DRO}: $$\bar{\Cs{X}}_{\epsilon}:=\sset*{ \bs{x} \in \Cs{X} }{\inf_{P \in \Cs{P}} \ P\{g(\bs{x}, \txi) \le 0 \} \ge 1-\epsilon}.$$ If $P^{\text{true}} \in \Cs{P}$, then, $\bs{x} \in \bar{\Cs{X}}_{\epsilon}  $ implies $\bs{x} \in \Cs{X}_{\epsilon}$. That is, $\bar{\Cs{X}}_{\epsilon} $ provides a conservative approximation to $\Cs{X}_{\epsilon}$\footnote{One can in turn seek a safe approximation to \eqref{eq: chance_DRO}. For example, one stream of such approximations includes using Chebyshev's inequality, see, e.g., \citet{popescu2005semidefinite,bertsimas2005optimal}, Bernstein's inequality, see, e.g., \citet{nemirovski2006convex},
or Hoeffding's inequality. We review such safe approximations to \eqref{eq: chance_DRO} in Section \ref{sec: rev.choice.ambiguity} in details.}. 
By leveraging a goodness-of-fit test, \citet{bertsimas2018RO} construct a $(1-\alpha)$-confidence region $\Cs{P}(\alpha)$ for $P^{\text{true}}$. Such a construction leads to an uncertainty set $\Cs{U}_{\epsilon}(\alpha)$ that guarantees the implication \eqref{eq: rev.prob.guarantee} \cite{bertsimas2018RO}.

Let us now assume that the sample space $\Xi$ is finite. By the relationship between RO and \dro, discussed in Section \ref{sec: rev.dro_ro}, one may think the parameter $\bs{\xi}$  in \eqref{eq: RO_Cons1} represents a probability distribution $\bs{p}$ on $\Bs{R}^{d}$, which is random. That said, we may define $f(\bs{x}, \bs{p}):=\rrisk{\bs{p}}{g(\bs{x}, \txi)}$. By leveraging the results in \citet{bertsimas2018RO}, we aim  to construct a data-driven ambiguity set $\Cs{P}_{\epsilon}$ that  guarantees the following implication: 
\begin{equation}
\label{eq: rev.prob.guarantee_dro}
	\text{If} \  \rrisk{\bs{p}}{g(\bs{x}, \txi)} \le 0, \ \forall \bs{p} \in \Cs{P}_{\epsilon}, \ \text{then,} \ P^{\text{true}}\{\rrisk{\tbs{p}}{g(\bs{x}, \txi)} \le 0 \} \ge 1-\epsilon. \tag{C2}
\end{equation}

\begin{theorem}{(\citet[Theorem~2]{bertsimas2018RO})}
    \label{thm: rev.chanceDRO}
    Suppose that for any fixed $\bs{x}$, $\rrisk{\bs{p}}{g(\bs{x}, \txi)}$ is concave in $\bs{p}$. Consider a set of   data $\{\bs{\xi}^{i}\}_{i=1}^{N}$, drawn independently and identically distributed (i.i.d.) according to $P^{\text{true}}$. Let $\Cs{P}_{\epsilon}(\alpha)$ be a $(1-\alpha)$-confidence region for $P^{\text{true}}$, constructed from a goodness-of-fit test on data. Moreover, for any $\bs{y} \in \Bs{R}^{d}$, let $l_{\epsilon}(\bs{y}; \alpha)$ be a closed, convex, finite-valued, and positively homogeneous (in $\bs{y}$) upper bound to the worst-case VaR of $\bs{y}^{\top} \tbs{p}$ at level $1-\epsilon$ over $\Cs{P}_{\epsilon}(\alpha)$, i.e., $\sup_{P \in \Cs{P}_{\epsilon}(\alpha)} \ \vvvar{P}{1-\epsilon}{\bs{y}^{\top} \tbs{p}} \le l_{\epsilon}(\bs{y}; \alpha), \; \bs{y} \in \Bs{R}^{d}$. 
    Then, the closed, convex set $\Cs{P}_{\epsilon}(\alpha)$ for which $\delta^{*}\big(\bs{y} | \Cs{P}_{\epsilon}(\alpha)\big)= l_{\epsilon}(\bs{y}; \alpha)$ guarantees the implication \eqref{eq: rev.prob.guarantee_dro} with probability at least $(1-\alpha)$ (with respect to the sample probability measure). 
\end{theorem}
As a byproduct of Theorem \ref{thm: rev.chanceDRO}, $\delta^{*}\big(\bs{y} | \Cs{P}_{\epsilon}(\alpha)\big) \le \bs{b}$ provides a safe approximation to \linebreak $\sup_{P \in \Cs{P}_{\epsilon}(\alpha)} \ P\{\bs{y}^{\top} \tbs{p} \le \bs{b}\} \ge 1-\epsilon$. That is, there is a one-to-one correspondence between the ambiguity set $\Cs{P}_{\epsilon}(\alpha)$ that satisfies the probabilistic guarantee \eqref{eq: rev.prob.guarantee_dro} and safe approximations to $\sup_{P \in \Cs{P}_{\epsilon}(\alpha)} \ P\{\bs{y}^{\top} \tbs{p} \le \bs{b}\} \ge 1-\epsilon$. %That is, one can probabilistically guarantee that a proper construction of the uncertainty set $\Cs{U}_{\epsilon}$ will lead to a safe approximation as in  \eqref{eq: chance_DRO}. 

%So far, we discussed how a RO model as in \eqref{eq: RO_Cons1} is connected to a \dro\ model, as in \eqref{eq: chance_DRO}. 

%%%%%%%%%%%%%%%%%%%%%%%%%%%%%%%%%%%%%
\subsection{Relationship with Function Regularization}
\label{sec: rev.rel_regularization}
%%%%%%%%%%%%%%%%%%%%%%%%%%%%%%%%%%%%%

%Suppose the ambiguity set $\Cs{P}$ is formed by the {\it optimal transport discrepancy}. Let us denote this set as  $\Cs{P}_{\gamma}$, where $\gamma$ is the size of the ambiguity set and controls the level of robustness. 
The goal of this section is to discuss the relationship of \dro/RO with the function regularization commonly used in machine learning. 

\subsubsection{\dro\ and Regularization}
Some papers have shown that \dro\ problems  via the {\it optimal transport discrepancy} and {\it $\phi$-divergences} are connected to regularization. When the optimal transport discrepancy is used,  as shown in \citet{shafieezadeh2015,blanchet2016robust,gao2016},  many  mainstream machine learning classification and regression models, including support vector machine (SVM), regularized logistic regression, and Least Absolute Shrinkage and Selection Operator (LASSO), have a direct distributionally robust interpretation that connects regularization to the protection from the disturbance in data. 
To state this result, we  first present a   duality theorem, due to \citet{blanchet2017DRO}, and we relegate  the technical details and assumptions to Section \ref{sec: rev.choice.ambiguity}. % We review these studies in Section \ref{sec: rev.distance} in details.   
On the other hand, when $\phi$-divergences are used, \dro\ problem is connected to variance regularization, see, e.g.,  \citet{duchi2016,namkoong2018variance}. 

Let us begin by defining the optimal transport discrepancy. Consider two probability measures $P_{1},  P_{2} \in \Fs{M}\measurespace$. Let $\Pi(P_{1},  P_{2})$ denote the set of all  probability measures on $\promeasurespace$ whose marginals are $P_{1}$ and $P_{2}$: 
\begin{equation*}
    \Pi(P_{1},  P_{2})=\sset*{ \pi \in\Fs{M}\promeasurespace}{\pi(A\times\Xi) = P_{1}(A), \pi(\Xi\times A) = P_{2}(A)   \forall A\in\Cs{F} }. 
\end{equation*}
An element of the above set is called a {\it coupling} or {\it transport plan}. 
Furthermore, suppose that  there is a lower semicontinuous function $c: \Xi \times \Xi \mapsto \Bs{R}_{+} \cup\{\infty\}$ with $c(s_{1},s_{2})=0$ if $s_{1}=s_{2}$.  
Then, the optimal transport discrepancy between  $P_{1}$ and $P_{2}$ is defined as\footnote{One can similarly define the optimal transport discrepancy between  two probability distributions $\Ts{P}_{1}$ and $\Ts{P}_{2}$ induced by $\txi$.}:
\begin{equation*} 
%\label{eq: rev.opt_transport}
\Fs{d}^{\text{W}}_{c}(P_{1},  P_{2}):= %\begin{cases}
\inf_{\pi\in \Pi(P_{1},  P_{2})}  \int_{\Xi\times \Xi} c(s_1,s_2) \pi(d s_1\times d s_2). %\Big)^{\frac{1}{p}}, & \text{if}\  1 \le p< \infty, %\\
%\inf_{K\in\Cs{S}(P_1,P_2)} K\textrm{-}\esssup d(\bs{s}_1,\bs{s}_2) , & \text{if}\  p= \infty,
%\end{cases}
\end{equation*}

\begin{theorem}{(\citet[Remark~1]{blanchet2017DRO})}
    %Without getting into technical details and assumptions, 
     \label{thm: rev.opt_transport_duality_no_details}
     Consider an ambiguity set of probability measures as $$\Cs{P}^{\text{W}}(P_{0}; \epsilon):=\sset*{P\in\Fs{M}\measurespace}{\Fs{d}^{\text{W}}_{c}(P,P_{0})\le \epsilon},$$ formed via the optimal transport discrepancy $\Bs{W}_{c}(P,P_{0})$, where $c$ is the transportation cost function, $\epsilon$ is the size of the ambiguity set (i.e., level of robustness), and $P_{0}$ is a nominal probability measure. %\footnote{In this paper, we use $\Cs{P}^{\text{W}}$ to denote both an ambiguity set of probability measures  and an ambiguity set of  distributions induced by $\txi$, formed via the optimal transport discrepancy. Whether  $\Cs{P}^{\text{W}}$  denotes  an ambiguity set of probability measures  or an ambiguity set of distributions induced by $\txi$ should be understood from the context and the distinction we make between a probability measure and a probability distribution.}.
      Then, for a fixed $\bs{x} \in \Cs{X}$, we have 
     $$
     %\min_{\bs{x} \in \Cs{X} } \ 
     \sup_{P \in \Cs{P}^{\text{W}}(P_{0};\epsilon) } \ 	\ee{P}{h(\bs{x},\cdot)} = 
     %\min_{\bs{x} \in \Cs{X} } 
     \inf_{\lambda \ge 0} \ \left\lbrace \lambda \epsilon +  \ee{P_{0}}{\sup_{s \in \Xi} \ \{h(\bs{x},s)- \lambda c(\tilde{s},s)\}} \right\rbrace.$$
\end{theorem}
%That is, 
%\begin{equation*}
%\min_{\bs{x} \in \Cs{X} } \max_{\Ts{P} \in \Cs{P}_{\gamma}} \ \ee{\Ts{P}}{h(\bs{x}, \txi)} =\min_{ \bs{x} \in \Cs{X}  } \ %\ee{\Ts{P}}{h(\bs{x}, \txi)} + \gamma \|\bs{x}\|_{r},
%\end{equation*}
%where $\|\cdot\|_{r}$ denotes the $L_{r}$ norm.  
We can use Theorem \ref{thm: rev.opt_transport_duality_no_details} to explicitly state the connection between \dro\ and regularization. We adopt the following two theorems from \citet{blanchet2017DRO}, due to their generality. However, similar results are obtained in other papers, see, e.g., \citet{shafieezadeh2015,gao2016}. 
\begin{theorem}{(\citet[Theorem~2--3]{blanchet2016robust})}
    \label{thm: rev.lin_reg_square_loss_reg_reg}
    Consider a given set of data \linebreak $\{\bs{\xi}^{i}:=(\bs{u}^{i},y^{i})\}_{i=1}^{N}$, where $\bs{u}^{i} \in \Bs{R}^{n}$ is a vector of covariates and $y^{i} \in \Bs{R}$ is the response variable.  Suppose that $\Ts{P}_{N}$ is the empirical probability distribution on $\{\bs{\xi}^{i}\}_{i=1}^{N}$,  $c(\bs{\xi}^{1},\bs{\xi}^{2}):=\| \bs{u}_{1} - \bs{u}_{2} \|_{q}^{2}$ if $y^{1}=y^{2}$, and  $c(\bs{\xi}^{1},\bs{\xi}^{2})=\infty$, otherwise. Let $\frac{1}{p}+\frac{1}{q}=1$. 
    Then, 
    \begin{itemize}
    \item  For a linear regression model with a square loss function $h_{1}(\bs{x}, \bs{\xi}):=(y-\bs{x}^{\top}\bs{u})^{2}$, we have 
    $$\inf_{\bs{x} \in \Bs{R}^{n} } \ \sup_{\Ts{P} \in \Cs{P}^{\text{W}}(\Ts{P}_{N}; \epsilon) } \ 	\ee{\Ts{P}}{h_{1}(\bs{x},\txi)} = \inf_{\bs{x} \in \Bs{R}^{n}}  \ \left\lbrace  \epsilon^{\frac{1}{2}} \| \bs{x} \|_{p} + \Big( \ee{\Ts{P}_{N}}{h_{1}(\bs{x},\txi)} \Big)^{\frac{1}{2}} \right\rbrace^{2},$$
    
    \item For a logistic regression model with cost function $h_{2}(\bs{x}, \bs{\xi}):=\log(1+e^{-y\bs{x}^{\top}\bs{u} })$, we have $$\inf_{\bs{x} \in \Bs{R}^{n} } \ \sup_{\Ts{P} \in \Cs{P}^{\text{W}}(\Ts{P}_{N}; \epsilon) } \ 	\ee{\Ts{P}}{h_{2}(\bs{x},\txi)} = \inf_{\bs{x} \in \Bs{R}^{n}}  \ \left\lbrace  \epsilon \| \bs{x} \|_{p} +  \ee{\Ts{P}_{N}}{h_{2}(\bs{x},\txi)}  \right\rbrace,$$
    
    \item For a SVM with Hinge loss $h_{3}(\bs{x}, \bs{\xi}):=(1-y\bs{x}^{\top}\bs{u})_{+} $, we have 
     $$\inf_{\bs{x} \in \Bs{R}^{n} } \ \sup_{\Ts{P} \in \Cs{P}^{\text{W}}(\Ts{P}_{N}; \epsilon) } \ 	\ee{\Ts{P}}{h_{3}(\bs{x},\txi)} = \inf_{\bs{x} \in \Bs{R}^{n}}  \ \left\lbrace  \epsilon \| \bs{x} \|_{p} +  \ee{\Ts{P}_{N}}{h_{3}(\bs{x},\txi)}  \right\rbrace.$$
    \end{itemize}
\end{theorem}

\begin{comment}
\begin{theorem}{\citet[Theorem~3]{blanchet2016robust}}
    \label{thm: rev.reg_reg}
    Consider a given set of data $\{\bs{\xi}^{i}:=(\bs{u}^{i},y^{i})\}_{i=1}^{N}$, where $\bs{u}^{i} \in \Bs{R}^{n}$ is a vector of covariates and $y^{i} \in \Bs{R}$ is the response variable.  Suppose that $\nomP$ is the empirical probability distribution on $\{\bs{\xi}^{i}\}_{i=1}^{k}$, $c(\bs{\xi}^{1},\bs{\xi}^{2}):=\| \bs{u}^{1} - \bs{u}^{2} \|_{q}^{2}$ if $y^{1}=y^{2}$, and  $c(\bs{\xi}^{1},\bs{\xi}^{2})=\infty$, otherwise. Furthermore, suppose that $h_{1}(\bs{x}, \bs{\xi}):=\log(1+e^{-y\bs{x}^{\top}\bs{u} })$ for a logistic regression model and $h_{2}(\bs{x}, \bs{\xi}):=(1-y\bs{x}^{\top}\bs{u})_{+} $ for a SVM with Hinge loss. 
    Then, we have 
    $$\inf_{\bs{x} \in \Bs{R}^{n} } \ \sup_{\Ts{P} \in \Cs{P}^{\text{W}}(\epsilon) } \ 	\ee{\Ts{P}}{h_{1}(\bs{x},\txi)} = \inf_{\bs{x} \in \Bs{R}^{n}}  \ \left\lbrace  \epsilon \| \bs{x} \|_{p} +  \ee{\nomP}{h_{1}(\bs{x},\txi)}  \right\rbrace,$$
    and 
    $$\inf_{\bs{x} \in \Bs{R}^{n} } \ \sup_{\Ts{P} \in \Cs{P}^{\text{W}}(\epsilon) } \ 	\ee{\Ts{P}}{h_{2}(\bs{x},\txi)} = \inf_{\bs{x} \in \Bs{R}^{n}}  \ \left\lbrace  \epsilon \| \bs{x} \|_{p} +  \ee{\nomP}{h_{2}(\bs{x},\txi)}  \right\rbrace,$$
    where $\frac{1}{p}+\frac{1}{q}=1$. 
\end{theorem}
\end{comment}

As stated in Theorem \ref{thm: rev.lin_reg_square_loss_reg_reg}, we can rewrite an {\it unconstrained} \dro\ model with the optimal transport discrepancy as a minimization problem, in which the objective function, in one hand, includes an expected-cost term with respect to the empirical distribution, and on the other hand, includes a regularization term. Two other interesting results  can be inferred from Theorem \ref{thm: rev.lin_reg_square_loss_reg_reg}  about the connection between \dro\ and regularization: (i) the shape of the transportation cost $c$ in the definition of the optimal transport discrepancy directly implies the type of  regularization, and (ii) the size of the ambiguity set is related to the regularization parameter. An important implication of these results is that one can judicially choose an appropriate regularization parameter for the problem in hand by using the \dro\ equivalent reformulation.  
We review the papers that draw this conclusion in Section \ref{sec: rev.distance}. 

Now, let us focus on \dro\ problems formulated via $\phi$-divergences. For two probability measures $P_{1}, P_{2} \in \Fs{M}\measurespace$, the $\phi$-divergence between $P_{1}$ and $P_{2}$ is defined as  $\Fs{d}^{\phi}(P_{1}, P_{2}):=\int_{\Xi}\phi\left(\frac{d P_{1}}{d P_{2}}\right) d P_{2}$, where the $\phi$-divergence function $\phi : \Bs{R}_{+} \rightarrow \Bs{R}_{+} \cup \{+ \infty\}$ is convex, and it satisfies the following properties: $\phi(1)=0$, $0\phi\left(\frac{0}{0}\right):=0$, and $a\phi\left(\frac{a}{0}\right):=a \lim_{t \rightarrow \infty} \frac{\phi(t)}{t}$ if $a>0$\footnote{One can similarly define the $\phi$-divergence between  two probability distributions $\Ts{P}_{1}$ and $\Ts{P}_{2}$ induced by $\txi$.}. 

\begin{theorem}{(\citet[Theorem~2]{duchi2016})}
    %Without getting into technical details and assumptions, 
     \label{thm: rev.phi_divergences_reg}
     Consider an ambiguity set of probability distributions as $$
        \Cs{P}^{\phi}(\nomP; \epsilon):= \sset*{\Ts{P} \in \P}{ \Fs{d}^{\phi}(\Ts{P},\nomP) \le \epsilon},$$
     formed via the $\phi$-divergence  $\Fs{d}^{\phi}(\Ts{P} ,\nomP)$, where $\epsilon$ is the size of the ambiguity set and $\nomP$ is the empirical probability distribution on a set of independently and identically distributed (i.i.d) data $\{\bs{\xi}^{i}\}_{i=1}^{N}$, according to $\trueP$.
     Furthermore, suppose that $\Cs{X}$ is compact, there exists a measurable function $M: \Omega \mapsto \Bs{R}_{+}$ such that for all $\bs{\xi} \in \Omega$, $h(\cdot, \bs{\xi})$ is $M(\bs{\xi})$-Lipschitz with respect to some norm $\|\cdot\|$ on $\Cs{X}$,  $\ee{\trueP}{M(\txi)^{2}}< \infty$, and $\ee{\trueP}{|h(\bs{x}_{0}, \txi)|}<\infty$ for some  $\bs{x}_{0} \in \Cs{X}$. 
     Then,   
     $$\sup_{\Ts{P} \in \Cs{P}^{\phi}(\Ts{P}_{N}; \frac{\epsilon}{N}) } \ \ee{\Ts{P}}{h(\bs{x},\txi)}= \ee{\Ts{P}_{N}}{h(\bs{x},\txi)}  + \Big( \frac{\epsilon}{N} \VVar{\Ts{P}_{N}}{h(\bs{x},\txi)} \Big)^{\frac{1}{2}}  + \gamma_{N}(\bs{x}),$$
     where $\gamma_{N}(\bs{x})$ is such that $\sup_{\bs{x} \in \Cs{X}}  \sqrt{N} |\gamma_{N}(\bs{x})| \rightarrow 0$ in probability. 
\end{theorem}

As Theorem \ref{thm: rev.phi_divergences_reg}, we can rewrite the inner problem of a model of the form \eqref{eq: DRO_Obj} with $\phi$-divergences as the expected cost  plus  a regularization term that accounts for the standard deviation of the cost, under the empirical distribution. 

\section{General Solution Techniques to Solve \dro\  Models}
\label{sec: rev.solution}
%%%%%%%%%%%%%%%%%%%%%%%%%%%%%%%%%%%%%

In this section, we discuss two approaches to solve \eqref{eq: DRO}. Let us first reformulate  \eqref{eq: DRO} as follows:
\begin{subequations}
\label{eq: DRO_Reformulation}
	\begin{align}
	\inf_{\bs{x} \in \Cs{X}, \theta } \  & \theta \\
	\text{s.t.} \quad & \theta \ge  \rrisk{P}{h(\bs{x},\txi)}, \; \forall P \in \Cs{P}\\
		& \rrisk{P}{\bs{g}(\bs{x},\txi)} \le \bs{0}, \; \forall P \in \Cs{P}.
	\end{align}	
\end{subequations}
Reformulation \eqref{eq: DRO_Reformulation} is a semi-infinite program (\sip), and at a first glance,  obtaining an optimal solution to this problem looks unreachable\footnote{
The study of \sip s is pioneered by  \citet{haar1924}, and followed up in \citet{charnes1962duality,charnes1963duality,charnes1969theory}, which focus on linear \sip s. The first- and second-order optimality conditions of general SIP are also obtained in \citet{hettich1977conditions,hettich1978SIP,hettich1995,nuernberger1985SIP,nuernberger1985Opt,still1999}. For reviews of the theory and methods for \sip s, we refer the readers to \citet{hettich1993,reemtsen1998,lopez2007}.}.
It is  well-known that even  convex \sip s cannot be solved directly with numerical methods, and in particular are not amenable to the use of methods such as interior point method. Therefore, a key step of the solution techniques to handle  the semi-infinite qualifier (i.e., $\forall P \in \Cs{P}$)  is to reformulate  \eqref{eq: DRO_Reformulation} as  an optimization problem that is amenable to the use of available optimization techniques and off-the-shelf solvers.  Of course, the complexity and tractability of such \sip s and their reformulations depend on the geometry and properties of both the ambiguity set $\Cs{P}$ and the functions $h(\bs{x},\txi)$ and $\bs{g}(\bs{x},\txi)$. As we shall see in details in Section \ref{sec: rev.choice.ambiguity}, proper assumptions on    $\Cs{P}$ and these functions are important in  most studies on \dro\ in order to obtain a solvable reformulation or approximation of \eqref{eq: DRO_Reformulation}.

In the context of \dro, there are two main approaches to handle the  semi-infinite quantifier $\forall P$ and to numerically solve \eqref{eq: DRO_Reformulation}. Both approaches  have their roots in the  \sip\ literature, and they both aim at getting rid of the  quantifier $\forall P$, but in different ways. 

\subsection{Cutting-Surface Method}
The first approach replaces   the quantifier $\forall P$ by  {\it for some finite atomic subset of} $\Cs{P}$. The idea is to successively solve a relaxed problem of \eqref{eq: DRO_Reformulation}  over a finitely generated inner approximations of the ambiguity set $\Cs{P}$. To be precise, this approach approximates the semi-infinite constraints for all  $P \in \Cs{P}$ by finitely many ones over a finite set of probability distributions. In each iteration of this approach, a new probability distribution is added to this finite set until optimality criteria are met. We refer to this as a {\it cutting-surface} method (also known as  {\it exchange method}, following the terminology in the \sip\ literature, see, e.g., \citet{mehrotra2014semi,hettich1993}). We refer to  \citet{pflug2007,rahimian2019,bansal2018} as examples of this approach in the context of \dro.   

The key requirements in order to use the cutting-surface method are the abilities to (i) solve a relaxation of \eqref{eq: DRO_Reformulation} with a finite number of probability distributions to optimally and (ii)  generate an $\epsilon$-optimal solution\footnote{For an optimization problem of the form $z^{*}=\min\sset*{\alpha(\bs{x})}{\beta(\bs{x}) \le \bs{0}}$, a point $\bs{x}_{0}$ is an $\epsilon$-optimal solution if $\beta(\bs{x}_{0}) \le \bs{0}$ and $\alpha(\bs{x}_{0}) \le z^{*}+ \epsilon$.} to a distribution separation subproblem \cite{luo2019}. 
\begin{theorem}{(\citet[Theorem~3.2]{luo2019})}
    \label{thm: SIP}
    Suppose that $\Cs{X}  \times \Cs{P} $ is compact, and  $\rrisk{P}{h(\bs{x},\txi)}$ and $\rrisk{P}{\bs{g}(\bs{x},\txi)}$ are continuous on $\Cs{X}  \times \Cs{P}$. Moreover, suppose that we have an oracle that generates an optimal solution $(\bs{x}_{k}, \theta_{k})$ to a relaxation of problem \eqref{eq: DRO_Reformulation} for any finite set $\Cs{P}_{k} \subseteq \Cs{P}$, and an oracle that generates an $\epsilon$-optimal solution of the distribution generation subproblem $$ \sup_{P \in \Cs{P}} \max\Bigg\{\rrisk{P}{h(\bs{x},\txi)}, \rrisk{P}{g_{1}(\bs{x},\txi)}, \ldots, \rrisk{P}{g_{m}(\bs{x},\txi)}\Bigg\}$$ for any $\bs{x} \in \Cs{X}$ and $\epsilon>0$. Suppose that iteratively  the relaxed master problem is solved to optimally and yields the solution $(\bs{x}_{k}, \theta_{k})$, and  the distribution separation subproblem is solved to $\frac{\epsilon}{2}$-optimality and yields the solution $P_{k}$. Then, the stopping criteria  $\rrisk{P}{h(\bs{x},\txi)} \le \theta_{k} + \frac{\epsilon}{2}$ and $\rrisk{P}{g_{j}(\bs{x},\txi)} \le  \frac{\epsilon}{2}$, $j=1, \ldots, m$, guarantee that  an $\epsilon$-feasible solution\footnote{For an optimization problem of the form $z^{*}=\min\sset*{\alpha(\bs{x})}{\beta(\bs{x}) \le \bs{0}}$, a point $\bs{x}_{0}$ is an $\epsilon$-feasible solution if $\beta(\bs{x}_{0}) \le \bs{\epsilon}$.} to problem \eqref{eq: DRO_Reformulation}, yielding an objective function value lower bounding the optimal value of \eqref{eq: DRO_Reformulation}, can be obtained in a finite number of iterations. 
\end{theorem}

It is worth noting that the distribution generation subproblem in the cutting-surface method may be a nonconvex optimization problem. One may efficiently solve \eqref{eq: DRO}  through the  cutting-surface method if the ambiguity set $\Cs{P}$ can be convexfied without causing a change to the optimal value. The following lemma states that if $\rrisk{P}{\cdot}$ is convex in $P$ on $\Fs{M}\measurespace$, then, it can be assumed  without loss of generality that $\Cs{P}$ is convex. 

\begin{lemma}
	\label{lem: G_Convex_Hull}
	Consider \eqref{eq: DRO}. For a fixed $x \in \Cs{X}$, suppose that $\rrisk{P}{\cdot}$ is  convex in $P$ on $\Fs{M}\measurespace$. Then, $\bs{x}^{*} \in \Cs{X}$ is an optimal solution to \eqref{eq: DRO} if and only if it is an optimal solution to the following problem: 
	\begin{equation}
	    \label{eq: rev.conv_formulation}
        \inf_{\bs{x} \in \Cs{X} } \ \sup_{P \in \conv{\Cs{P}}} \ 	\sset*{\rrisk{P}{h(\bs{x},\txi)}}{\sup_{P \in \conv{\Cs{P}} } \ \rrisk{P }{\bs{g}(\bs{x},\txi)} \le \bs{0}}.
    \end{equation}
\end{lemma}

\begin{proof}
    Problems  \eqref{eq: DRO} and \eqref{eq: rev.conv_formulation} can be reformulated, respectively, as  \linebreak $\min \sset*{\theta}{(x,\theta) \in \Cs{G}}$ and $\min \sset*{\theta}{(x,\theta) \in \Cs{G}^{\prime}}$, where 
    $$\Cs{G}:=\sset*{(x,\theta) \in \Bs{R}^{n+1}}{\bs{x} \in \Cs{X}, \; \rrisk{P}{h(\bs{x},\txi)} \le \theta, \; \rrisk{P}{\bs{g}(\bs{x},\txi)} \le \bs{0}, \;  \forall P \in \Cs{P}},$$
    and 
    $$\Cs{G}^{\prime}:=\sset*{(x,\theta) \in \Bs{R}^{n+1}}{\bs{x} \in \Cs{X}, \; \rrisk{P}{h(\bs{x},\txi)} \le \theta, \; \rrisk{P}{\bs{g}(\bs{x},\txi)} \le \bs{0}, \;  \forall P \in \conv{\Cs{P}}}.$$
    Because $\Cs{P} \subseteq \conv{P}$, we have $\Cs{G}^{\prime} \subseteq \Cs{G}$, and thus, an optimal solution to \eqref{eq: rev.conv_formulation} is optimal to \eqref{eq: DRO}. We now show that $\Cs{G} \subseteq \Cs{G}^{\prime} $. Consider an arbitrary $(\bs{x},\theta) \in \Cs{G}$.  For an arbitrary $P \in \conv{\Cs{P}}$, there exists a collection $\{P^{i}\}_{i \in \Cs{I}}$ such that $P=\sum_{i \in \Cs{I}} \lambda^{i} P^{i}$, where $\sum_{i \in \Cs{I}} \lambda^{i} =1$, $ P^{i} \in \Cs{P}$, $\lambda^{ i} \ge 0$, $i \in \Cs{I}$. Now, by the convexity of $\rrisk{P}{\cdot}$ in $P$ on $\Fs{M}\measurespace$, we have $\rrisk{P}{h(\bs{x},\txi)} \le \sum_{i \in \Cs{I}} \lambda^{i} \rrisk{P^{i}}{h(\bs{x},\txi)} \le \theta$ and $\rrisk{P}{\bs{g}(\bs{x},\txi)} \le \sum_{i \in \Cs{I}} \lambda^{i} \rrisk{P^{i}}{\bs{g}(\bs{x},\txi)} \le \bs{0}$. Thus, it follows that $(\bs{x},\theta) \in \Cs{G}^{\prime}$, and hence, $\Cs{G} \subseteq \Cs{G}^{\prime} $. 
\end{proof}

\subsection{Dual Method}

The second approach  %  $\Cs{P}$ cannot be convexfied without causing a change to the optimal value of \eqref{eq: DRO}, then one may apply a different approach 
to solve \eqref{eq: DRO}  handles the quantifier $\forall P$ through the dualization of  $\sup_{P \in \Cs{P}} \ \rrisk{P}{h(\bs{x},\txi)}$ and $\sup_{P \in \Cs{P}} \ \rrisk{P}{\bs{g}(\bs{x},\txi)} \le \bs{0}$. %by utilizing duality theory and getting rid of 
Under suitable regularity conditions, there is no duality gap between the primal problem and its dual, i.e., strong duality holds. Hence, the supremum can be replaced by an infimum which should hold for at least one corresponding solution in the dual space. We refer  to this approach as a {\it dual method}. Most of the existing papers in the \dro\ literature are focused on the dual method, see, e.g.,  \citet{delage2010,bertsimas2010minmax,wiesemann2013,ben2013}. 
A situation where one benefits from the application  of the dual method to solve \eqref{eq: DRO} arises in cases where the ambiguity set of probability distribution depends on decision $\bs{x}$ as formulated  below, see, e.g., \citet{luo2018,noyan2018}: 
\begin{equation}
\label{eq: D3RO}
\inf_{\bs{x} \in \Cs{X} } \ \sup_{P \in \Cs{P}(\bs{x})} \ 	\sset*{\rrisk{P}{h(\bs{x},\txi)}}{\sup_{P \in \Cs{P}(\bs{x}) } \ \rrisk{P}{\bs{g}(\bs{x},\txi)} \le \bs{0}},
\end{equation}
where, $\Cs{P}(\bs{x})$ denotes a {\it decision-dependent} ambiguity set of the probability distributions. 

The papers that rely on the dual method exploit linear duality, Lagrangian duality, convex analysis (e.g., support function, conjugate duality, Fenchel duality), and conic duality. A fundamental question is then under what conditions the strong duality holds. One  such condition is the existence of a probability measure that lies in the interior of the ambiguity set, i.e., the ambiguity set satisfies a Slater-type condition.  We  refer the readers to the optimization textbooks for results on linear and Lagrangian duality, see, e.g., \citet{bazaraa2013NLP,bertsekas1999NLP,ruszczynski2006NLP,rockafellar1974duality}. For detailed discussions of the duality theory in infinite-dimensional convex problems, we refer to \citet{rockafellar1974duality}, and we refer to \citet{isii1962} and \citet{shapiro2001duality} for duality theory in conic linear programs.
Below, we briefly present the results from conic duality %and conjugate duality 
that are widely used in the dualization of \dro\ models.

\begin{theorem}{(\citet[Proposition~2.1]{shapiro2001duality})}
\label{thm: rev.conic_duality}
For a linear mapping $A: \Cs{V} \mapsto \Cs{W}$, recall the definition of the adjoint mapping $A^{*}: \Cs{W}^{*} \mapsto \Cs{V}^{*}$, where $\langle w^{*}, Av \rangle= \langle A^{*}w^{*}, v \rangle$, $\forall v \in \Cs{V}$. 
Consider a conic linear optimization problem of the form
\begin{subequations}
\label{eq: rev.conic_primal}
\begin{align}
    \min_{v \in \Cs{C}} \ & \langle c,v \rangle \\
    \st \quad &   Av  \succcurlyeq_{\Cs{K}} b,
\end{align}
\end{subequations}
%where $A$ is a linear mapping and $A^{*}$ is its adjoint mapping, as defined in Section \ref{sec: rev.notation}. 
where, $\Cs{C} $ and $\Cs{K}$ are convex cones and subsets of linear spaces $\Cs{V}$ and $\Cs{W}$, respectively, such that for any $w^{*} \in \Cs{W}^{*}$, there exists a unique $v^{*} \in \Cs{V}^{*}$ with $\langle w^{*},Av \rangle=\langle v^{*},v \rangle$, with $v^{*}=A^{*}w^{*}$, for all $v \in \Cs{V}$. 
Then, the dual problem to \eqref{eq: rev.conic_primal} is written as 
\begin{subequations}
\label{eq: rev.conic_dual}
\begin{align}
    \max_{w^{*} \in \dual{K}} \ & \langle w^{*},b \rangle \\
    \st \quad &   A^{*}w^{*}  \preccurlyeq_{\dual{C}} c.
\end{align}
\end{subequations}
Moreover, there is no duality gap between \eqref{eq: rev.conic_primal}  and \eqref{eq: rev.conic_dual} and both problems have optimal solutions if and only if there exists a feasible pair $(v, w^{*})$ such that $\langle w^{*},Av-b \rangle=0$ and $\langle c-A^{*}w^{*},v \rangle=0$.  
\end{theorem}

%\begin{proposition}{\citet{shapiro2001duality}}
%\label{prop: rev.conic_duality_biconjugate}
%Define  $f(w)=\min \sset*{\langle c,v \rangle}{v \in \Cs{C}, \; Av  \succcurlyeq_{\Cs{K}} w}$. Then, $f^{*}(w^{*})$ is the characteristic  function of the feasible set of problem \eqref{eq: rev.conic_dual} and thus, the optimal value to \eqref{eq: rev.conic_dual} is equal to $f^{**}(w)=\sup_{w^{*} \in \Cs{W}^{*}} \{\langle w^{*},b \rangle - f^{*}(w^{*}) \}$.
%\end{proposition}

%\begin{theorem}{\citet[Fenchel-Moreau Duality]{shapiro2001duality}, \citet[Theorem~5]{rockafellar1974duality}}
%\label{thm: rev.conjugate_duality}
%Consider a convex, proper, and lower semicontinuous  function $f: \Cs{V} \mapsto \ol{\Bs{R}}$ on a linear space $\Cs{V}$. Then, $f^{**}=f$, i.e., $f(v) =\sup_{v^{*} \in \Cs{V}^{*}} \{\langle v^{*},v \rangle - f^{*}(v^{*}) \}, \; \forall v \in \Cs{V}$
%\end{theorem}

It is worth noting that other numerical methods to solve a SIP, such as penalty methods, see, e.g.,  \citet{lin2014,yang2016}, smooth approximation and projection methods, see, e.g.,  \citet{xu2014}, and primal methods, see, e.g.,  \citet{wang2015}, have not been  popular in the \dro\ literature, although there are a few exceptions. 
%approach whose framework can be summarized as the following
%stages: consider the Lagrangian dual of the inner max problem,
%then reformulate the min-max problem as a min-min (combining
%the min-min by min) problem with semi-infinite constraints,
%and finally recast the semi-infinite constraints as a linear semidefinite
%constraint by S-Lemma or dual method again.
\citet{liu2017primal} propose to discretize  \dro\ by a min-max problem in a finite dimensional space, where the ambiguity set is replaced by a set of  distributions on a discrete support set. %solve the min-max problem based on the first approach (discretized ambiguity set). 
Then, they consider lifting techniques to reformulate the discretized \dro\ as a saddle-point problem, if needed, and  implement a primal-dual hybrid algorithm to solve the problem. They showcase this method for cases where the ambiguity set is formed via the moment constraints as in \eqref{eq: moment-rob-set} or the Wasserstein metric, and they present the quantitative convergence of the optimal value and  optimal solutions. %They apply their results to portfolio optimization problems. 
Other iterative primal methods that have been proposed to solve a \dro\ model include \citet{lam2013} for $\chi^{2}$-distance, and \citet{ghosh2018sgd,namkoong2018,ghosh2018} for general $\phi$-divergences.

% % % % % % % % % % % % % % % % % % % % % % % % % % % % % % % % % % % % % % % % % % % %	
\section{Choice of Ambiguity Set of Probability Distributions}
\label{sec: rev.choice.ambiguity}
% % % % % % % % % % % % % % % % % % % % % % % % % % % % % % % % % % % % % % % % % % % %
%\subsection{Terminology}
% % % % % % % % % % % % % % % % % % % % % % % % % % % % % % % % % % % % % % % % % % % %

The ambiguity set of distribution in a \dro\ model provides a  flexible framework to model uncertainty  by allowing the modelers to incorporate partial information about the uncertainty, obtained from historical data or domain-specific knowledge. This information includes, but it is not limited to, support of the uncertainty, discrepancy from the reference distribution, descriptive statistics, and structural properties, such as symmetry and unimodality. 
Early \dro\ models  considered ambiguity sets based on the support and moment information, for
which techniques in global optimization for polynomial optimization problems and problem of moments are
applied to obtain reformulations, see, e.g., \citet{lasserre2001,bertsimas2006persistence,bertsimas2005optimal,popescu2005semidefinite,popescu2007,gilboa1989}. 
Since then, many researchers have incorporated information such as  descriptive statistics as well as the structural properties  of the underlying unknown true distribution into the ambiguity set. 

There are usually two  principles  to choose the ambiguity set: (1) $\Cs{P}$ should be chosen as small as possible, (2) $\Cs{P}$ should contain the unknown true distribution with certainty (or at least, with a high confidence). Abiding by these two principles not only reduces the conservatism of the problem but it also  robustifies the problem against the unknown true distribution. 
These two, in turn, give rise to two questions: (1) what should be  the shape of the ambiguity set and (2) what should be the size of the ambiguity set. 
We discuss the latter in Section  \ref{sec: rev.calibration}, and focus on the shape of the ambiguity set in this section. 

Except for a few exceptions, the common practice in constructing the ambiguity set is that first, the shape of the  set is determined by decision makers/modelers. In this step, data does not directly affect the choice of the shape of the ambiguity set. Then, the parameters that control the size of the ambiguity set are chosen in a data-driven fashion. We emphasize that albeit being a common practice,  the size and shape of the ambiguity set might not  necessarily be chosen separately. To make the transition between Section \ref{sec: rev.choice.ambiguity} and \ref{sec: rev.calibration} somewhat smoother, we devote Section \ref{sec: rev.kernel} to  review those papers that address these two questions simultaneously.   
 
When dealing with the question of the shape of the ambiguity set,  most  researchers, on one hand, have focused on the ambiguity sets that facilitate a tractable (exact or conservative approximate) formulation, such as linear program (LP), second-order cone program (SOCP), or to a lesser degree, semidefinite program (SDP), so that efficient computational techniques can be developed. On the other hand, many researchers have focused on the expressiveness of the ambiguity set by incorporating information such as  descriptive statistics as well as the structural properties  of the underlying unknown true distribution. 
%Early \dro\ models usually consider ambiguity sets based on the support and moments information, for
%which techniques in global optimization for polynomial optimization problem and problem of moments are
%applied to obtain reformulations, see, e.g., \citet{lasserre2001,bertsimas2006persistence,bertsimas2005optimal,popescu2005semidefinite,popescu2007,gilboa1989}. 
%Since then, many researchers have included other information about the underlying distribution in the ambiguity set. 

In what follows in this section, we review different approaches to model the distributional ambiguity. 
We acknowledge that the ambiguity sets in the literature are  typically categorized in two groups: {\it moment-based} and {\it discrepancy-based} ambiguity sets. In short, moment-based ambiguity sets contain distributions whose moments satisfy certain properties, while  discrepancy-based ambiguity sets contain distributions that are close to a nominal distribution in the sense of  some {\it discrepancy} measure. Within these two groups, some specific ambiguity sets have been given names, see, e.g., \citet{hanasusanto2015chance}. For example, 
\begin{itemize}
	\item {\it Markov} ambiguity set contains all distributions with known mean and support, 
	\item {\it Chebyshev} ambiguity set contains all distributions with  bounds on the first- and second-order moments, 
	\item {\it Gauss} ambiguity set contains all unimodal distributions from within the Chebyshev ambiguity set,
	\item {\it Median-absolute
	deviation} ambiguity set contains all symmetric distributions with known median
	and mean absolute deviation, 
	\item {\it Huber} ambiguity set contains all distributions
	with known upper bound on the expected Huber loss function, 
	\item {\it Hoeffding}	ambiguity set contains all componentwise independent distributions with a box support,  
	\item {\it Bernstein} ambiguity set contains all distributions from within
	the {\it Hoeffding} ambiguity set subject to marginal moment bounds, 
	\item {\it Choquet} ambiguity set contains all distributions that can be written as  an infinite convex combination of extremal distributions of the set, 
	\item {\it Mixture} ambiguity set contains all distributions that can be written as a  mixture of a parametric family of distributions.  
\end{itemize}
While we use the above terminology in this paper, we categorize  \dro\ papers into four groups: %, based on their studied ambiguity sets. 
\begin{itemize}
	\item  Discrepancy-based ambiguity sets (Section \ref{sec: rev.distance}),
	\item Moment-based ambiguity sets (Section \ref{sec: rev.moment}),
	\item Shape-preserving ambiguity sets (Section \ref{sec: rev.shape}),
	\item Kernel-based ambiguity sets (Section \ref{sec: rev.kernel}).
\end{itemize}
We briefly mentioned what is meant by discrepancy-based  and  moment-based  ambiguity sets. In short, shape-preserving  ambiguity sets contain distributions with similar structural properties (e.g., unimodality, symmetry). Kernel-based ambiguity sets  also contain distributions that are formed via a kernel and its parameters are close to the parameters of a nominal kernel function. 
\begin{comment}
\begin{itemize}
	\item Moment-based ambiguity sets contain distributions whose moments satisfy certain properties, 
	\item discrepancy-based ambiguity sets contain distributions that are close to a nominal distribution with respect to some discrepancy measure, 
	\item Shape-preserving-based ambiguity sets contain distributions with similar structural properties (e.g., unimodality, symmetry),
	\item Kernel-based ambiguity also sets contain distributions that are formed via a kernel and its parameters are close to the parameters of a nominal kernel function. 
\end{itemize}
\end{comment}
The above groups are not necessarily disjoint from a modeling perspective and there are some overlaps between them.  % are not necessarily disjoint from a modeling perspective, 
However, we  try to assign papers to these categories as close as possible to  what the authors explicitly or implicitly might have stated in their work. %  That is, there are ambiguity sets and papers that belong to more than groups. %For example, some of the papers that we categorize into shape-preserving-based ambiguity sets might also belong to the moment-based ambiguity sets. 

%As discussed, there are ambiguity sets and papers that belong to more than groups. 
%To clear up the jungle of \dro\ models, some papers propose unifying frameworks to model the distributional ambiguity. For example, 

We review  these four groups of ambiguity sets in Sections \ref{sec: rev.distance}--\ref{sec: rev.kernel}. 
%Within in each Section \ref{sec: rev.distance}--\ref{sec: rev.kernel}, we also review the papers that are applicable to the whole group in the last subsection of the corresponding section. 
Finally, we review the papers that are general and do not consider a specific form for the ambiguity set in Section \ref{sec: rev.general}.

%%%%%%%%%%%%%%%%%%%%%%%%%%%%%%%%%%%%%
\subsection{Discrepancy-Based Ambiguity Sets}
\label{sec: rev.distance}
%%%%%%%%%%%%%%%%%%%%%%%%%%%%%%%%%%%%%

In many situations, we have a {\it nominal} or {\it baseline} estimate of the underlying probability distribution. A natural way to hedge against  the distributional ambiguity is then to consider a neighborhood of the nominal  probability distribution by allowing some perturbations around it. So, the ambiguity set can be formed with all probability distributions whose {\it discrepancy} or {\it dissimilarity} to the nominal probability distribution is sufficiently small. More precisely, such an ambiguity set has the following generic form:
\begin{equation}
    \label{eq: rev.ambiguity_distance_generic}
    \Cs{P}^{\Fs{d}}(P_{0};\epsilon)=\sset*{P \in \Fs{M}\measurespace}{\Fs{d} (P,P_{0}) \le \epsilon},
\end{equation}
where $P_{0}$ denotes the nominal probability measure, and  $\Fs{d} : \Fs{M}\measurespace \times \Fs{M}\measurespace \mapsto \Bs{R}_{+} \cup \{\infty\}$ is a  functional that measures the discrepancy between two probability measure $P, P_{0} \in \Fs{M}\measurespace$, dictating the shape of the ambiguity set. Moreover, parameter $\epsilon \in [0, \infty]$ controls the size of the ambiguity set, and it can be interpreted as the decision maker's belief in $P_{0}$. Parameter $\epsilon$ is also referred to as the {\it level of robustness}.  

A generic ambiguity set of the form \eqref{eq: rev.ambiguity_distance_generic} has been widely studied in the \dro\ literature. 
We relegate the discussion about $P_{0}$ and $\epsilon$ to Section \ref{sec: rev.calibration}. In this section, we review different discrepancy functionals $\Fs{d}(\cdot, \cdot)$ that are used in the literature. These include
(i) {\it optimal transport discrepancy}, (ii) {\it $\phi$-divergences}, (iii) {\it total variation metric}, (iv) {\it goodness-of-fit test}, (v) {\it Prohorov metric}, (vi) {\it $\ell_{p}$-norm}, (vii) {\it $\zeta$-structure metric}, (viii) {\it Levy metric}, and (ix) {\it contamination neighborhood}. We emphasize that although all studied functionals $\Fs{d}$ can quantify the discrepancy  between two probability measures, they may or may not be a metric. For example, Prohorov and total variation are probability metrics, see, e.g.,  \citet{gibbs2002}, while {\it Kullback-Leibler} and {$\chi^{2}$-distance} from the family of $\phi$-divergences are not a probability metric. Thus, we refer to the models of the form \eqref{eq: rev.ambiguity_distance_generic} collectively as   {\it  discrepancy-based} ambiguity sets.  

\subsubsection{Optimal Transport Discrepancy}

%In Section \ref{sec: rev.rel_regularization}, we defined  and skipped the technical details. 
We begin this section by providing more details on the optimal transport discrepancy.
%Let us remind  the definition of the optimal transport discrepancy. 
Consider two probability measures $P_{1},  P_{2} \in \Fs{M}\measurespace$. Let $\Pi(P_{1},  P_{2})$ denote the set of all probability measures on $\promeasurespace$ whose marginals are $P_{1}$ and $P_{2}$: 
\begin{equation*}
    \Pi(P_{1},  P_{2})=\sset*{ \pi \in\Fs{M}\promeasurespace}{ \pi(A\times\Xi) = P_{1}(A),  \pi(\Xi\times A) = P_{2}(A)  \forall A\in\Cs{F} }. 
\end{equation*}
%An element of the above set is called a {\it coupling} or {\it transport plan}. 
Furthermore, suppose that  there is a lower semicontinuous function $c: \Xi \times \Xi \mapsto \Bs{R}_{+} \cup\{\infty\}$ with $c(s_{1},s_{2})=0$ if $s_{1}=s_{2}$.  
Then, the optimal transport discrepancy between  $P_{1}$ and $P_{2}$ is defined as:
\begin{equation} 
\label{eq: rev.opt_transport}
\Fs{d}^{\text{W}}_{c}(P_{1},  P_{2}):= %\begin{cases}
\inf_{\pi\in \Pi(P_{1},  P_{2})}  \int_{\Xi\times \Xi} c(s_1,s_2) \pi(d s_1\times d s_2). %\Big)^{\frac{1}{p}}, & \text{if}\  1 \le p< \infty, %\\
%\inf_{K\in\Cs{S}(P_1,P_2)} K\textrm{-}\esssup d(\bs{s}_1,\bs{s}_2) , & \text{if}\  p= \infty,
%\end{cases}
\end{equation}
If, in addition, % to the properties listed for $c$ in Section \ref{sec: rev.rel_regularization}, 
function $c$  is symmetric (i.e., $c(s_{1},s_{2})=c(s_{2},s_{1})$) and $c^{\frac{1}{r}}(\cdot)$ satisfies a triangle inequality for some $1 \le r < \infty$ (i.e.,  $c^{\frac{1}{r}}(s_{1},s_{2}) \le c^{\frac{1}{r}}(s_{1},s_{3}) + c^{\frac{1}{r}}(s_{3},s_{2})$), then, $\Fs{d}^{\text{W}}_{c^{{\frac{1}{r}}}}(P_{1},  P_{2})$  metricizes the weak convergence in  $\Fs{M}\measurespace$, see, e.g., \citet[Theorem~6.9]{villani2008}. If $\Xi$ is equipped with a metric $d$ and $c(\cdot)=d^{r}(\cdot)$, then $ \Fs{d}^{\text{W}}_{c}(P_{1},  P_{2})$ is called {\it Wasserstein metric of order $r$ or $r$-Wasserstein metric}, for short\footnote{Wasserstein metric of order $1$ is sometimes referred to as {\it Kantorovich} metric. Wasserstein metric of order $\infty$ is defined as $\inf_{\pi\in  \Pi(P_{1},P_{2})} \pi\textrm{-}\esssup \  d(s_{1},s_{2})$, where $\pi\textrm{-}\esssup_{\Xi \times \Xi} \ [\cdot]$ is the essential supremum with respect to measure $\pi$: $\pi\textrm{-}\esssup_{\Xi \times \Xi} \ d(s_{1},s_{2})=\inf\{a \in \Bs{R}: \pi(s \in \Xi: \exists s^{\prime} \in \Xi \ \st \  d(s,s^{\prime})>a)=0\}$.}. 

The optimal transport discrepancy \eqref{eq: rev.opt_transport} can be used to form an ambiguity set of probability measures as follows: %. Suppose that $P_{0}$ denotes a nominal probability measure in  $\Fs{M}\measurespace$. Then, an optimal-transport-based ambiguity set can be formed as: 
\begin{equation}
\label{eq: rev.opt.transport.set}
\Cs{P}^{\text{W}}(P_{0}; \epsilon):=\sset*{P\in \Fs{M}\measurespace}{\Fs{d}^{\text{W}}_{c}(P,P_{0})\le \epsilon}. 
\end{equation}
%where $P_0$ is a nominal probability distribution, and $\Bs{W}_{p}(\cdot,\cdot):\;\P\times\P\to\mathbb{R}$ 
%is the $L_{p}$-Wasserstein metric (or, Wasserstein metric of order $p$), for $p \in [1,\infty]$: 
%\begin{equation} \label{eq: Kantorovich_metric}
%\Bs{W}_{p}(P_1, P_2):= \begin{cases}
%\inf_{K\in\Cs{S}(P_1,P_2)} \Big( \int_{\Xi\times \Xi} d^{p}(\bs{s}_1,\bs{s}_2) K(d\bs{s}_1\times d \bs{s}_2) \Big)^{\frac{1}{p}}, & \text{if}\  1 \le p< \infty, \\
%\inf_{K\in\Cs{S}(P_1,P_2)} K\textrm{-}\esssup d(\bs{s}_1,\bs{s}_2) , & \text{if}\  p= \infty,
%\end{cases}
%\end{equation}
Over the past few years, there has been a significant growth in the popularity of the optimal transport discrepancy to model the distributional ambiguity in \dro, in both operations research and machine learning communities, see, e.g., \citet{pflug2007,mehrotra2014,mohajerin2018,gao2016,chen2018chance,blanchet2018structural,lee2015,luo2019,shafieezadeh2015,sinha2018,lee2018stat,shafieezadeh2018,singh2018}. 
Pioneered by the work of \citet{pflug2007}, most of the literature has focused on the Wasserstein metric. Before we review these papers, we present a  duality result on $\sup_{P \in \Cs{P}^{\text{W}}(P_{0}; \epsilon)} \ \ee{P}{h(\bs{x},\txi)}$, proved in a general form in \citet{blanchet2017DRO}. 

Because the infimum in the defintion of \eqref{eq: rev.opt_transport} is attained for a lower semicontinuous function $c$ \cite{villani2008,rachev1998}, we can rewrite $\sup_{P \in \Cs{P}^{\text{W}}(P_{0}; \epsilon)} \ \ee{P}{h(\bs{x},\txi)}$ as follows:
\begin{equation}
    \label{eq: rev.opt_transport_primal}
    %\sup_{\Ts{P} \in \Cs{P}^{\text{W}}} \ \ee{\Ts{P}}{\bs{g}(\bs{x},\txi)} = 
    \sup_{\pi \in \Phi_{P_{0}, \epsilon} }  \ \int_{\Xi} h(\bs{x},s) \pi(\Xi \times d s), 
\end{equation}
where 
\begin{equation*}
    \begin{split}
        & \Phi_{P_{0}, \epsilon}  := \\ 
        &  {} \sset*{ \pi \in \Fs{M}\promeasurespace}{\pi \in \cup_{P \in \Fs{M}\measurespace } \Pi(P_{0}, P), \; \int_{\Xi\times \Xi} c(s_1,s_2) \pi(d s_1\times d s_2) \le \epsilon}.
    \end{split}    
\end{equation*}
 %, and with some abuse of notation, for $s \in \Xi$, we write $h(\bs{x},\bs{\xi}(s))$ as $h(\bs{x},s)$.
Recall that $\Cs{S}\measurespace$ is the collection of all $\Cs{F}$-measurable functions $Z: \measurespace \mapsto (\ol{\Bs{R}}, \Cs{B}(\overline{\Bs{R}}))$. 
With the primal problem \eqref{eq: rev.opt_transport_primal}, we have a dual problem 
\begin{equation}
    \label{eq: rev.opt_transport_dual}
    \inf_{(\lambda, \phi) \in \Lambda_{c,h(\bs{x}, \cdot)} } \  \left\lbrace  \lambda \epsilon +  \int_{\Xi} \phi(s)  P_{0}(ds) \right\rbrace ,
\end{equation}
where 
$$\Lambda_{c,h(\bs{x}, \cdot)}:=\sset*{(\lambda, \phi)}{\lambda \ge 0, \ \phi \in \Cs{S}\measurespace, \ \phi(s_{1}) + \lambda c(s_1,s_2) \ge h(\bs{x}, s_{2}), \forall  s_1,s_2 \in \Xi}.$$

%Now, we are ready to state a strong duality result, due to . 
\begin{theorem}{(\citet[Theorem~1]{blanchet2017DRO})}
    \label{thm: rev.opt_transport_duality} 
    For a fixed $\bs{x} \in \Cs{X}$, suppose that $h(\bs{x}, \cdot)$ is upper semicontinuous and $P_{0}$-integrable, i.e., $\int_{\Xi} |h(\bs{x}, \txi(s))| P_{0}(ds) \linebreak < \infty$. Then, 
    $$\sup_{\pi \in \Phi_{P_{0}, \epsilon} }  \ \int_{\Xi} h(\bs{x},s) \pi(\Xi \times d s) = \inf_{(\lambda, \phi) \in \Lambda_{c,\bs{g}(\bs{x}, \cdot)} } \  \left\lbrace  \lambda \epsilon +  \int_{\Xi} \phi(s)  P_{0}(ds) \right\rbrace.$$
    Moreover, there exists a dual optimal solution of the form $(\lambda, \phi_{\lambda})$, for some $\lambda \ge 0$, where $\phi_{\lambda}(s_{1}):=\sup_{s_2 \in \Xi} \ \{h(\bs{x},s_{2})- \lambda c(s_{1},s_{2}) \}$. In addition, any feasible $\pi^{*} \in \Phi_{P_{0},\epsilon}$ and $(\lambda^{*}, \phi_{\lambda^{*}}) \in \Lambda_{c,\bs{g}(\bs{x}, \cdot)}$ are primal and dual optimizers, satisfying 
    $$ \int_{\Xi} h(\bs{x},s) \pi^{*}(\Xi \times d s) = \lambda^{*} \epsilon +  \int_{\Xi} \phi_{\lambda^{*}}(s)  P_{0}(ds),$$ if and only if 
    \begin{subequations}
    \label{eq: rev.opt_trasnport_conditions}
    \begin{align}
        & h(\bs{x},s_{2})- \lambda^{*} c(s_{1},s_{2})= \sup_{s_{3} \in \Xi} \ \{h(\bs{x},s_{3})- \lambda^{*} c(s_{1},s_{3}) \}, \ \pi^{*}\text{-almost surely},\\
        & \lambda^{*} \Big( \int_{\Xi\times \Xi} c(s_1,s_2) \pi(d s_1\times d s_2) - \epsilon \Big)=0.
    \end{align}
    \end{subequations}
\end{theorem}

\begin{corollary}
    \label{cor: rev.opt_transport_duality} 
    Suppose that $h(\bs{x}, \cdot)$ is upper semicontinuous and $P_{0}$-integrable. Then, 
    \begin{equation}
    \label{eq: rev.opt_transport_duality_final}
    \sup_{P \in \Cs{P}^{\text{W}}(P_{0}; \epsilon)} \ \ee{P}{h(\bs{x},\txi)}=   \inf_{\lambda \ge 0} \ \left\lbrace \lambda \epsilon +  \ee{P_{0}}{\sup_{s \in \Xi} \ \{h(\bs{x},s)- \lambda c(\tilde{s},s)} \right\rbrace.
\end{equation}
\end{corollary}

The importance of Theorem \ref{thm: rev.opt_transport_duality} and Corollary  \ref{cor: rev.opt_transport_duality}   is that (1) the transportion cost $c(\cdot, \cdot)$ is only known to be lower semicontinuous, (2) function $h(\bs{x},\txi)$ is assumed to be upper semicontinuous and integrable, and (3)  $\Xi$ is a general Polish space. In fact, 
there are only mild conditions on $h(\bs{x}, \cdot)$ and function $c$, and $P_{0}$ can be any probability measure. % defined on a compact space $\Xi$. 
%As a consequence of Theorem \ref{thm: rev.opt_transport_duality}, we have 
%\begin{equation}
%    \label{eq: rev.opt_transport_duality_final}
%   \sup_{\pi \in \Phi_{\nomP, \epsilon} }  \ \int_{\Xi} h(\bs{x},s) \pi(\Xi \times d s) = \inf_{\lambda \ge 0} \ \left\lbrace \lambda \epsilon +  \ee{P_{0}}{\sup_{s \in \Xi} \ \{h(\bs{x},s)- \lambda c(\tilde{s},s)} \right\rbrace,  
%\end{equation}
%which implies 
Moreover, $\sup_{P \in \Cs{P}^{\text{W}}(P_{0};\epsilon)} \ \ee{P}{h(\bs{x},\txi)}$ can be obtained by solving   a univariate reformulation of the dual problem \eqref{eq: rev.opt_transport_dual}, where it involves an expectation with respect to $P_{0}$ and a linear term in the level of robustness $\epsilon$. We shall shortly comment on similar results in the literature but under stronger assumptions. 
As shown in Section \ref{sec: rev.rel_regularization}, by using Theorem  \ref{thm: rev.opt_transport_duality} or its weaker forms, researchers have shown many mainstream machine learning algorithms, such as regularized logistic regression and LASSO, have a \dro\ representation, see, e.g., \citet{blanchet2016robust,blanchet2017groupwise,blanchet2017Semi,gao2017,shafieezadeh2015,shafieezadeh2017}. %To make the connection between regularization and \dro\ more precise, we state the following results from   \citet{blanchet2016robust}. 

%Now, we review the papers that use the optimal transport discrepancy to form the ambiguity set of distributions. 

%study distributionally robust multistage stochastic programs where the reference data and information structure is represented as a tree. The proposed model hedges against worst-case in an ambiguous neighborhood of the reference model by utilizing the multistage {\it nested distance} formed via the Wasserstein metric. They assume the tree structure and scenario values are fixed, while the probabilities are changing. 
%They present a successive convex programming algorithm to find the optimal decisions asymptotically, and  as an example, they study a production/inventory control problem. 

While  a strong duality result for \dro\ formed via the optimal transport discrepancy is provided in \citet{blanchet2017DRO} under  mild assumptions by utilizing Fenchel duality, \citet{mohajerin2018} and \citet{gao2016} are also among  notable papers in this area. Below,  we first highlight the main differences of \citet{mohajerin2018} and \citet{gao2016} with \citet{blanchet2017DRO}. Then, we comment on their main contributions. 

%As mentioned before, \citet{blanchet2016robust} study \dro\ under a general setting with mild conditions as follows: (1) the transportion cost $c(\cdot, \cdot)$ is only known to be lower semicontinuous, (2) function $\bs{g}(\bs{x},\txi)$ is assumed to be upper semicontinuous and integrable, and (3) the nominal distribution $\nomP$ is supported on  a general Polish space. 
In \citet{mohajerin2018},  it is assumed that the transportation cost $c(\cdot, \cdot)$ is a norm on $\Bs{R}^{n}$, function $h(\bs{x},\txi)$ has specific structures, and the nominal probability measure $P_{0}$ is the empirical distribution of data supported on $\Bs{R}^{n}$. 
On the other hand, \citet{gao2016} consider a more general setting than the one in \citet{mohajerin2018}, but slightly more  restricted than that of \citet{blanchet2016robust}. More precisely, in contrast to \citet{blanchet2016robust}, it is assumed in \citet{gao2016} that the transportation cost $c(\cdot, \cdot)$ forms a metric on the underlying Polish space. 

%Below, we review the main contributions of \citet{mohajerin2018} and \citet{gao2016}. 
\citet{mohajerin2018} study data-driven \dro\ problems %where an ambiguity set of infinite-dimensional probability distributions is formed around the empirical distribution of data, 
formed via $1$-Wasserstein metric utilizing an arbitrary  norm on $\Bs{R}^{n}$. The main contribution of  \citet{mohajerin2018} is in proving a strong duality result for the studied problem and to reformulate it as a finite-dimesnional convex program for different cost functions, including a pointwise maximum of finitely many concave functions, convex functions, and sums of maxima of concave functions. This contribution is of  importance as most of the previous research on \dro\ formed via Wasserstein ambiguity sets reformulates the problem as a finite-dimensional nonconvex program and 
relies on global optimization techniques, such as difference of convex programming,  to solve the problem, see, e.g., \cite[Theorem~6]{wozabal2012}.  In addition,  \citet{mohajerin2018}   propose a procedure to construct an extremal distribution (respectively, a sequence of distributions) that attains the worst-case expectation precisely (or, asymptotically). They further show that their solutions enjoy  finite-sample and asymptotic consistency guarantees. The results were applied to the mean-risk portfolio optimization and to the uncertainty quantification problems. 

\citet{gao2016} study \dro\ problems formed via  $p$-Wasserstein metric utilizing an arbitrary metric on a Polish space $\Xi$. Recognizing the fact that the ambiguity set should be chosen judicially for the application in hand, they argue that by using the Wasserstein metric the resulting distributions hedged against are more reasonable than those resulting from other popular choices
of sets, such as $\phi$-divergence-based sets, see Section \ref{sec: rev.phi}. %(1) The resulting distributions hedged against are more reasonable than those resulting from other popular choices
 %(2) the problem of determining the worst-case expectation over the resulting set of distributions has desirable tractability properties. 
They prove a strong duality result for the studied problem by utilizing Lagrangian duality and approximate the worst-case distributions (or obtain a worst-case distribution, if it exists) explicitly
via the first-order optimality conditions of the dual reformulation. Using this, they show data-driven \dro\ problems can be approximated by robust optimization problems.

%We acknowledge that  fundamental results for \dro\ problem formed via the optimal transport discrepancy and under the conditions discussed above are developed 
In addition to the papers by \citet{blanchet2017DRO,mohajerin2018,gao2016}, there  are other research on \dro\ problems formed via the optimal transport discrepancy, but under more restricted assumptions,  %are worth highlighting. 
that move the frontier of research in this area. 
In the following review, we mention the properties  of the transportation cost $c(\cdot, \cdot)$ in the definition of the optimal transport discrepancy, function $\bs{g}(\bs{x},\txi)$ or $h(\bs{x},\txi)$, and the nominal distribution $\nomP$ and its underlying space as studied in these papers. %, to emphasize the applicability of the strong duality results, developed in   \citet{blanchet2016robust,mohajerin2018,gao2016}.  
\citet{zhao2018wass} study a data-driven distributionally robust two-stage
stochastic linear program over a Wasserstein ambiguity set, with  $1$-Wasserstein metric utilizing $\ell_1$-norm. By developing a strong duality result, they reformulate the problem as a semi-infinite linear two-stage robust optimization problem. In addition, under mild conditions, they  derive a closed-form expression
of the worst-case distribution whose parameters can
be obtained by solving a traditional two-stage robust optimization
model.  They also show the convergence of the problem to the corresponding stochastic program under the true unknown probability distribution as the data points increase. %which is subsequently solved using a Benders decomposition algorithm.

\citet{hanasusanto2018} derive conic programming reformulation to distributionally robust  two-stage  stochastic linear programs formed via $p$-Wasserstein metric utilizing an arbitrary norm. In particular, by relying on the strong duality result from \citet{mohajerin2018} and \citet{gao2016}, they show that when the ambiguity set is formed via the  $2$-Wasserstein metric around a discrete distribution, the resulting model is equivalent to a copositive program of polynomial size  (if the problem
has complete recourse) or it can be approximated  by a sequence of
copositive programs of polynomial size  (if for any fixed $\bs{x}$ and $\bs{\xi}$, the dual of the second-stage problem is feasible% problem has sufficiently expensive recourse
). Moreover,  by using nested hierarchies of  semidefinite approximations of the % tractable convex cones to approximate the 
(intractable) copositive cones
from the inside, they obtain sequences
of tractable conservative approximations to the problem. They also show that the two-stage distributionally robust stochastic linear program with non-random cost function in the second stage, where the ambiguity set is formed via the $1$-Wasserstein metric around a discrete distribution %and there are no support constraints, 
is equivalent
to a  linear program. They further extend their result to a case where optimized certainty equivalent (OCE) \citep{bental1986,bental2007OCE} is used as a risk measure. As applications, they demonstrate their results for the least absolute deviations  regression and multitask learning problems. 

For random variables  supported on a compact set and a bounded continuous function $h(\bs{x}, \cdot)$,  \citet{luo2019} study \eqref{eq: DRO_Obj} formed via the  $1$-Wasserstein metric utilizing an arbitrary norm, around the empirical distribution of data. They present an equivalent SIP reformulation of the problem by reformulating the inner problem as a conic linear program. In order to solve the resulting SIP, they propose a finitely convergent  exchange method when  the cost  function $h$ is a general nonlinear function in $\bs{x}$, and a central cutting-surface method with a linear rate of convergence when the cost function $h(\cdot, \bs{\xi})$ is convex in $\bs{x}$ and $\Cs{X}$ is convex. They investigate a logistic regression model to exemplify their algorithmic ideas, and the benefits of using $1$-Wasserstein metric.  

 \citet{pflug2014}  study a \dro\ approach to single- and two-stage stochastic  programs formed via the $p$-Wasserstein metric utilizing an arbitrary norm. They assume that all probability distributions in the ambiguity set are  supported on discrete, fixed atoms, while only the  probabilities of atoms are changing in the ambiguity set.
Hence, the ambiguity set can be represented as a subset of a finite-dimensional space. To solve the resulting problem, they apply the exchange method, proposed in \citet{pflug2007}. %to solve the problem, a successive convex programming algorithm to find the optimal decisions. 
\citet{mehrotra2014} study a distributionally robust ordinary least squares problem, where the ambiguity set of probability distribution is formed via $1$-Wasserstein metric utilizing $\ell_{1}$-norm. Similar to \citet{pflug2014}, they restrict the ambiguity set of distributions to all discrete distributions and show that the resulting problem can be solved by using an equivalent SOCP reformulation. 

Unlike \citet{pflug2014}  and \citet{mehrotra2014} that only allow varying the  probabilities on atoms identical to those of the nominal distribution, the ambiguity set is allowed to contain  an infinite-dimensional distribution in \citet{wozabal2012}.  \citet{wozabal2012} study a \dro\ approach to single-stage stochastic programs, where the distributional ambiguity in the constraints and objective  function  is modeled via  $1$-Wasserstein metric utilizing $\ell_{1}$-norm around the empirical distribution. Because such a model has a higher complexity than that of those considered in \citet{pflug2014}  and \citet{mehrotra2014}, they propose to reformulate the problem into an equivalent   finite-dimensional, nonconvex saddle-point optimization problem, under appropriate conditions. The key ideas in \citet{wozabal2012} to obtain such a reformulation are that (i) at any level of precision and in the sense of Kantorovich distance, every distribution in the ambiguity set can be  approximated  via a probability distribution supported on a uniform number of atoms, and (ii) considering only the  extremal distributions in the ambiguity set suffices to obtain the equivalent  reformulation. Furthermore, for a portfolio selection problem complemented via a broad class of convex risk measures appearing in the constraints, they obtain an equivalent finite-dimensional, nonconvex,  semidefinite  saddle-point optimization problem. They propose to solve such a reformulated problem via the exchange method, proposed in \citet{pflug2007}. %\alertHR{Compact space $\Xi \subseteq \Bs{R}^{d}$}

\citet{pichler2017} study a \dro\ model with a distortion risk measure and form the ambiguity set of distributions via $p$-Wasserstein metric utilizing an arbitrary norm. They quantitatively investigate the effect of  the variation of the ambiguity set on  the optimal value and the optimal solution in the resulting optimization problem, as the number of data points increases. %These findings strenghen some recent convergence results on \dro\ where the nominal distribution of the Wasserstein ball is constructed by the empirical probability distribution. 
They illustrate their results in the context of a two-stage stochastic program with recourse. %\alertHR{Refer to it later for some discussion of literature, and also for defining $\zeta$-metric and it's special cases.}

A class of data-driven distributionally robust fractional optimization problems, representing a reward-risk ratio, is studied in \citet{ji2017} as follows:
\begin{equation}
\inf_{\bs{x} \in \Cs{X} } \ \sup_{P \in \Cs{P}} \ 	\frac{\Cs{R}^{1}_{P}\left[h(\bs{x},\txi)\right]}{\Cs{R}^{2}_{P}\left[h(\bs{x},\txi)\right]},
\end{equation}
where $\Cs{R}^{1}_{P}: \Cs{Z} \mapsto \Bs{R}$ is a reward measure and $\Cs{R}^{2}_{P}: \Cs{Z} \mapsto \Bs{R}_{+}$ is a  nonnegative risk measure. 
Assuming that the underlying distribution is discrete, \citet{ji2017} model the  ambiguity about  discrete distributions using the  $1$-Wasserstein metric utilizing $\ell_1$-norm, around the empirical distribution. % over a finite support
They  provide  a nonconvex reformulation for the resulting model and propose a bisection algorithm to obtain the  optimal value by solving a sequence of convex programming problems. As in \citet{postek2016}, the reformulation is obtained through investigating the support function of the ambiguity set and the convex conjugate of the ratio function. They further apply their results to portfolio optimization problem for the Sharpe ratio \cite{sharpe1966} and Omega ratio \cite{keating2002}.

Motivated by the drawback of moment-based \dro\ problems,  \citet{gao2017dep} study \dro\ formed via various ambiguity sets of probability distributions that incorporate the dependence structure between the uncertain parameters.
In the case that there exists a linear dependence structure, they consider probability distributions around a nominal distribution, in the sense of  $p$-Wasserstein metric utilizing an arbitrary norm, satisfying a second-order moment constraint. They also study cases with different rank dependencies between the uncertain parameters. They obtain tractable reformulations of these models  and apply their results to a portfolio optimization problem. 
Along the same lines as \citet{gao2017dep},  \citet{pflug2017review} study a \dro\ approach to portfolio optimization via  the $1$-Wasserstein metric utilizing an arbitrary norm. They address the case where the dependence structure between the assets is uncertain while the marginal distributions of the assets are known.  %In addition, they review the literature under ambiguous joint distribution of the portfolio's assets. 

\citet{noyan2018} study \dro\ model with decision-dependent ambiguity set, where the ambiguity set is formed via the $p$-Wasserstein metric utilizing $\ell_{p}$-norm. They consider two types of ambiguity sets: (1) {\it continuous} ambiguity set, where there is ambiguity in both probability distribution of $\txi$ and its realizations, and (2) {\it discerte} ambiguity set, where there is only ambiguity in the probability distribution of $\txi$, while the realizations are fixed.
They apply their results to problems in machine scheduling and humanitarian logistics. 
\citet{rujeerapaiboon2018reduction} study continuous and discrete scenario reduction \citep{dupavcova2003scenario,heitsch2003scenario,heitsch2009modeling,heitsch2009reduction,arpon2018}, where  
$p$-Wasserstein metric utilizing $\ell_{p}$-norm is used as a measure of discrepancy between distributions.

\paragraph{Discrete Problems}
We now review \dro\ models over Wasserstein ambiguity sets, with discrete decisions. 
\citet{bansal2018} study a distributionally robust integer program with  pure binary first-stage and mixed-binary second-stage variables  on a finite set of scenarios as follows:
\begin{equation*}
    \min_{\bs{x}}\sset*{\bs{c}^{\top}\bs{x} +\max_{P \in \Cs{P}} \ee{P}{h(\bs{x}, \txi)}}{\bs{A}\bs{x} \ge \bs{b}, \; \bs{x} \in \{0,1\}^{n}},
\end{equation*}
where 
\begin{equation*}
    h(\bs{x},\bs{\xi})=\min_{\bs{y}}\sset*{\bs{q}^{\top}(\bs{\xi})\bs{y}(\bs{\xi})}{\bs{W}(\bs{\xi})\bs{y}(\bs{\xi}) \ge \bs{r}(\bs{\xi}) - \bs{T}(\bs{\xi}) \bs{x}, \; \bs{y}(\bs{\xi}) \in \{0,1\}^{q_{1}} \times \Bs{R}^{q-q_{1}}}.
\end{equation*}
They propose a  decomposition-based L-shaped algorithm and a cutting surface algorithm to solve the resulting model. They investigate the conditions and the ambiguity sets under which the proposed algorithm is finitely convergent. They show that the ambiguity set of distributions formed via $1$-Wasserstein metric utilizing an arbitrary norm satisfy these conditions. 
\citet{xu2018mip} study a mixed 0-1
linear program, where the coefficients of the objective functions are affinely dependent on the random vector $\txi$. They seek a bound on the worst-case expected optimal value to this problem, where the worst-case is taken with respect to an ambiguity set of discrete distributions formed via $2$-Wasserstein metric utilizing $\ell_{2}$-norm around the empirical distribution of data. Under mild
assumptions, they reformulate the problem into a copositive program, which 
leads to a tractable semidefinite-based approximation.

\paragraph{Chance Constraints}
In this section, we review distributionally robust chance-constrained programs over Wasserstein ambiguity sets, see, e.g., \citet{jiang2016chance,chen2018chance,xie2018wass,yang2018control}. \citet{ji2018chance}  study a distributionally robust individual chance constraint, where the ambiguity set of distributions is formed via $1$-Wasserstein metric utilizing  $\ell_{1}$-norm, and $g(\bs{x}, \txi)$ in \eqref{eq: DRO_Cons} is defined as 
$$g(\bs{x}, \txi):=\mathbbm{1}_{[\bs{a}(\txi)^{\top}\bs{x} \le \bs{b}(\txi)]}(\txi).$$ 
%hey assume the chance constraint is affine in both $\txi$ and $\bs{x}$. 
For the case that the underlying distribution is supported on the same atoms as those of the empirical distribution, they provide mixed-integer LP reformulations for the linear random right-hand side case, i.e., $g(\bs{x}, \txi):=\mathbbm{1}_{[\bs{a}^{\top}\bs{x} \le \txi]}(\txi)$, and the linear random technology matrix case, i.e., $g(\bs{x}, \txi):=\mathbbm{1}_{[\txi^{\top}\bs{x} \le \bs{b}]}(\txi)$, and provide techniques to strengthen the formulations. 
For the case that the underlying distribution is infinitely supported, they propose an exact mixed-integer SOCP reformulation for models with random right-hand side, while a relaxation is proposed for constraints with a random
technology matrix. They show that this mixed-integer SOCP relaxation is exact when the decision variables are binary or bounded general integer.

\citet{chen2018chance} study data-driven distributionally robust chance constrained programs, where the ambiguity set of distributions is formed via $p$-Wasserstein metric utilizing an arbitrary norm. For individual linear chance constraints with affine dependency on the uncertainty, and for joint  chance constraints with right-hand side affine uncertainty, they provide an exact deterministic reformulation as a mixed-integer conic program. When $\ell_{1}$-norm or $\ell_{\infty}$-norm are used as the transportation cost in the  definition of Wasserstein metric, the chance-constrained program can be reformulated as a mixed-integer LP. They leverage the structural insights into the worst-case distributions, and  show that both the CVaR and the Bonferroni approximation may give solutions that are inferior to the optimal solution of their proposed reformulation. %\alertHR{Refer to it for more literature on DRO Chance}

\paragraph{Statistical Learning}

\dro\ problems formed via the optimal transport discrepency has been widely studied in the context of statistical learning. We already mentioned \citet{mehrotra2014} as an example in this area. Below, we review the latest developments of \dro\ in the context of statistical learning.  
A data-driven distributionally robust maximum likelihood estimation model to infer the inverse of the covariance matrix of a normal random vector is proposed in \citet{nguyen2018}. 
They form the ambiguity set of distributions with all normal distributions close enough to a nominal distribution characterized by the sample mean and sample covariance matrix, in the sense of  the $2$-Wasserstein metric utilizing $\ell_1$-norm. By leveraging an analytical formula for the Wasserstein distance between two normal distributions, they obtain an equivalent SDP reformulation of the problem. When there is no prior sparsity information on the inverse covariance matrix, they propose a closed-form expression for the estimator that can be interpreted as a nonlinear shrinkage estimator. Otherwise, they propose a sequential quadratic approximation algorithm to obtain the estimator by solving the equivalent SDP. They apply their results to linear discriminant
analysis, portfolio selection, and solar irradiation patterns inference problems. 

\citet{lee2015} study a distributionally robust  framework for finding  support vector machines   via the $1$-Wasserstein metric. They  provide SIP formulation of the resulting model and propose a cutting-plane algorithm to solve the problem.
\citet{lee2017,lee2018stat} study a distributionally robust statistical learning problem formed via the  $p$-Wasserstein metric utilizing $\ell_{p}$-norm, motivated by a domain (i.e., measure) adaption problem. This problem arises when training data are generated according to an unknown source domain $\Ts{P}$, but the learned hypothesis is evaluated on another unknown but related  target domain
$\Ts{Q}$. In this problem, it is assumed that a set of labeled data (covariates and responses) is drawn from $\Ts{P}$ and a set of unlabeled covariates is drawn from $\Ts{Q}$. It is further assumed that the domain drift is due to an unknown deterministic  transformation on the covariates space that  preserves the  distribution of the response conditioned on  the covariates. Under these assumptions and some further regularity conditions,  they prove a generalization bound and generalization error guarantees for the problem.

\citet{gao2018hypothesis}  develop a novel distributionally robust framework  for hypothesis testing where the ambiguity set of distribution is constructed by $1$-Wasserstein metric utilizing an arbitrary norm, around the empirical distribution. The goal is to obtain the optimal decision rule as well the least favorable distribution by minimizing the maximum of the worst-case type-I and type-II errors. 
They  develop a convex safe
approximation of the resulting problem and show that such an approximation
renders a nearly-optimal decision rule among the family of all possible tests.
By exploiting the structure of the least favorable distribution, they also develop a
finite-dimensional convex programming reformulation of the safe  approximation. %\alertHR{Refer to other papers in this area.}
%\alertHR{Remove this detail} 
%They also show that, under these different convex relaxations, they can derive similar insights as \citet{goldenshluger2016} that the expected value of the worst-case loss will have a clear statistical meaning and corresponds to Hellinger distance, Jensen-Shannon divergence, etc. 

%\citet{sinha2018}: Principled Adversarial Training

%{\bf Statistical Learning, DRO, and Regularization}

We now turn our attention to the connection between \dro\ and regularization in statistical learning. 
\citet{pflug2012,pichler2013,wozabal2014} draw the connection between robustification and regularization, where as in Theorem \ref{thm: rev.lin_reg_square_loss_reg_reg},  the shape of the transportation cost in the definition of the optimal transport discrepancy directly implies the type of regularization, and (ii) the size of the ambiguity set dictates the regularization parameter. \citet{pichler2013} studies worst-case values of lower semicontinuous and law-invariant risk measures, including  spectral
and distortion risk measures, over an ambiguity set of distributions formed via the   $p$-Wasserstein metric utilizing an arbitrary norm around the empirical distribution. They show when the  function $h(\bs{x},\txi)$ is linear in $\txi$, the worst-case value is the sum of the risk of $h(\bs{x},\txi)$ under the nominal distribution and a regularization term. 
\citet{pflug2012} and \citet{wozabal2014}  show the worst-case value of a convex law-invariant risk measure
over an ambiguity set of distributions, formed via the  $p$-Wasserstein metric utilizing $\ell_{p}$-norm around the empirical distribution,  reduces to the sum of the nominal risk
and a regularization term whenever the  function $h(\bs{x},\txi)$ is affine in $\txi$.% and the ambiguity set does not include any support constraints. 
They provide closed-form expressions for risk measures such as expectation, sum of expectation and  standard deviation, 
CVaR, distortion risk measure, Wang transform, proportional hazards transform, 
the Gini measure, and sum of expectation and mean absolute deviation from the median. They apply their results to a portfolio selection problem. %They derive  closed-form expression for the worst-case risk of $h(\bs{x},\txi)$ for various convex risk measures, whenever   $h(\bs{x},\txi)$ is linear and the distribution of $\txi$ ranges over a Wasserstein ambiguity set without support constraints.
Important parts of the derivation of results in \citet{pflug2012,pichler2013,wozabal2014}   are Kusuoka's representation of risk measures \citep{kusuoka2001,shapiro2013kusuoka} and Fenchel-Moreau theorem \citep{rockafellar1997,ruszczynski2006optimization}. 

%from wozabal2014
%El Ghaoui and Lebret (1997) investigate a linear
%regression problem with ambiguous data, where the ambiguity
%set is defined by the Froebnius norm and use conic
%programming techniques to show that the problem is equivalent
%to Tikhonov regularization. Bertsimas et al. (2004)
%show that robust linear programming problems where the
%ambiguity sets are balls in a normed space can be reformulated
%to convex optimization problems involving the dual
%norm, which can also be interpreted as regularized versions
%of the original problems. Similar results are obtained in
%Gotoh and Takeda (2011), Gotoh et al. (2013).
%The remainder of this paper is structured as

In the context of statistical learning, the connection between \dro\ and regularization was first made in \citet{shafieezadeh2015}, to the best of our knowledge. In fact, they study a distributionally robust logistic regression, where an ambiguity set of %infinite-dimensional 
probability distributions, supported on an open set, is formed around the empirical distribution of data and via the $1$-Wasserstein metric utilizing an arbitrary norm. They show the resulting problem admits an equivalent reformulation as a tractable convex program. As stated in Theorem \ref{thm: rev.lin_reg_square_loss_reg_reg}, this problem can be interpreted as a standard regularized  logistic regression, where the size of the ambiguity set dictates the regularization parameter.
They further propose a distributionally robust approach
based on Wasserstein metric to compute upper and lower confidence bounds on the
misclassification probability of the resulting classifier, based on  the optimal values of two  linear programs.

\citet{shafieezadeh2017} extend the work of \citet{shafieezadeh2015} and study distributionally robust supervised learning (regression and classification) models. They introduce a new generalization technique using ideas from \dro, whose ambiguity set contains all infinite-dimensional distributions in the Wasserstein neighborhood of the empirical  distribution.  They show that the classical robust and the distributionally robust learning models are equivalent if the data satisfies a dispersion condition (for regression) or a separability condition (for classification). By imposing bound on the decision (i.e., hypothesis)  space, they  improve the upper confidence bound on the out-of-sample performance proposed in \citet{mohajerin2018} and prove a generalization bound that does not rely on the complexity of the hypothesis space. This is unlike the traditional generalization
bounds that are derived by controlling the complexity of the hypothesis space, in terms of Vapnik-Chervonenkis (VC)-dimension, covering numbers,  or Rademacher complexities \cite{bartlett2002,shalev2014ML}, which are usually difficult to calculate and interpret in practice. 
They extend their results to the case that the unknown hypothesis is searched from the space of nonlinear functionals. Given a symmetric and positive definite kernel function, such a setting gives rise to  a lifted \dro\ problem that searches for a linear hypothesis  over  a {\it reproducing kernel Hilbert space} (RKHS).  

\citet{gao2017}  study \dro\ problems formed via  the $p$-Wasserstein metric utilizing an arbitrary norm,  around the empirical  distribution. They  identify a broad class of cost functions, for which  such a  \dro\ is asymptotically
equivalent to a regularization problem with a gradient-norm penalty under the nominal distribution. For linear function class, this equivalence is exact and results in a new interpretation  for discrete
choice models, including multinomial logit, nested logit, and generalized extreme value choice models. %Further details on equivalance in this ppaer. See the cited  papers. 
They also obtain lower and upper bounds on the worst-case expected cost in terms of regularization. 

\citet{mohajerinesfahani2018inverse} study a data-driven inverse optimization  problem to learn the objective function of the decision maker, given the  historical  data on uncertain parameters and decisions. In an environment with imperfect information% due to  measurement errors,  implementation errors, or bounded rationality
, they propose a \dro\ model formed via the $p$-Wasserstein metric utilizing an arbitrary norm to minimize the worst-case risk of the predicted error. Such a model  can be interpreted as a regularization of the corresponding empirical risk minimization problem. They present  exact (or safe approximation) tractable convex programming reformulation for different combinations of risk measures and error functions.

%By making connection between empirical likelihood, \dro, and optimal transport theory,   \citet{blanchet2016robust} introduce {\it robust Wasserstein profile inference} to estimate the parameter of interest  as a novel tool to be used in a wide range of inference problems and machine learning algorithms, such as LASSO, regularized regressions, among others. They model this tool as a \dro\ problem formed via the $p$-Wasserstein metric utilizing an arbitrary norm.
\citet{blanchet2017groupwise} study group-square-root LASSO (group LASSO  focuses on variable selection in settings where some predictive variables, if selected, must be chosen as a group). They model this problem as a \dro\ problem formed via   the $p$-Wasserstein metric utilizing an arbitrary norm. 
A  method for (semi-) supervised learning  based on data-driven \dro\ via $p$-Wasserstein metric utilizing an arbitrary norm, is proposed in \citet{blanchet2017Semi}. This method  enhances the generalization error by using the unlabeled data to restrict the support of the worst-case distribution in the resulting \dro. %Moreover, they discuss the large sample behavior of the optimal ambiguity set, which exposes topics such as  dimension reduction in semi-supervised learning.
They select the level of robustness using cross-validation, and they discuss the nonparametric behavior of an optimal selection of the level of robustness. 

\citet{chen2018regression} study a \dro\ approach to linear regression using an $\ell_{1}$-norm cost function, where the ambiguity set of distributions is formed via $p$-Wasserstein metric utilizing an arbitrary norm. They show that this \dro\ formulation can be relaxed to a convex optimization problem. By selecting proper norm spaces for the Wasserstein
metric, they are able to recover several commonly used regularized regression models. 
%They give guidance on the selection of the regularization parameter from the standpoint of a confidence region. 
They establish performance guarantees %for the solution to the formulation under mild conditions. 
%One is related to its 
on both the out-of-sample behavior (prediction bias)  and the  discrepancy between the estimated and true regression planes (estimation bias), which elucidate the role of the regularizer.
They study the  application of the proposed model to outlier detection, arising in an abnormally high radiation exposure in CT exams, and show it achieves a higher performance than M-estimation \cite{huber2009RobustStat}. %\alertHR{See the paper for more details regularization and robustness.}

%\citet{shafieezadeh2018}: Kalman filtering

%\citet{singh2018}: minmax distribution estimation

\paragraph{Choice of the Transportation Cost}
When forming a Wasserstein ambiguity set, the transportation cost function $c(\cdot, \cdot)$ should be chosen besides the nominal probability measure $P_{0}$ and the size of the ambiguity set $\epsilon$.  \citet{blanchet2017transport} propose a comprehensive approach for designing the ambiguity set in a  data-driven way, using the role of the transportation cost $c(\cdot,\cdot)$ in the definition of the $p$-Wasserstein metric. They apply various metric-learning procedures to estimate $c(\cdot,\cdot)$ from the training data, where they associate a relatively high transportation cost to two locations if transporting mass between these locations substantially impacts performance. This mechanism induces enhanced out-of-sample performance by focusing  on regions of relevance, while improving the generalization error. Moreover, 
this approach connects the metric-learning procedure to estimate the parameters of adaptive regularized estimators. They select the level of robustness using cross-validation. %, and they use stochastic gradient descent to solve the resulting \dro. 
\citet{blanchet2017doubly} propose a data-driven robust optimization approach to optimally inform the transportation cost in the definition of the $p$-Wasserstein metric. This additional layer of robustification within a suitable parametric family of transportation costs  does not exist in the metric-learning approach, proposed in \citet{blanchet2017transport}, and it allows to enhance the generalization properties of regularized estimators while reducing the variability in the out-of-sample performance error.

\paragraph{Multistage Setting}
The single- and two-stage stochastic programs in \citet{pflug2014} are extended in \citet{analui2014} and \citet{pflug2014} to the multistage case, where the reference data and information structure is represented as a tree. In these papers it is assumed that the tree structure and scenario values are fixed, while   the probabilities are changing only in an ambiguous neighborhood of the reference model by utilizing the multistage {\it nested distance}, formed via the Wasserstein metric. Both papers further apply their results to a  multiperiod production/inventory control problem.  
Built upon the above results, \citet{glanzer2018} show that a scenario tree can be constructed out of data such that it converges (in terms of the nested distance) to the
true model in probability at an exponential rate.
\citet{glanzer2018} also study a \dro\ framework formed via nested distance that allows for setting up bid and ask prices for acceptability pricing of contingent claims. 
Another study of multistage  linear optimization can also be found in \citet{baziermatte2018}.

\subsubsection[Phi-Divergences]{\texorpdfstring{$\phi$-Divergences}{Phi-Divergences}}
\label{sec: rev.phi}
Another popular way to model the distributional ambiguity is to use {\it $\phi$-divergences}, a  class of measures used in information theory.
A $\phi$-divergence measures the discrepancy between two probability measures $P_{1}, P_{2} \in \Fs{M}\measurespace$  as $\Fs{d}^{\phi}(P_{1}, P_{2}):=\int_{\Xi}\phi\left(\frac{d P_{1}}{d P_{2}}\right) d P_{2}$\footnote{One can similarly define the $\phi$-divergence between  two probability distributions $\Ts{P}_{1}$ and $\Ts{P}_{2}$ induced by $\txi$.}, where the $\phi$-divergence function $\phi : \Bs{R}_{+} \rightarrow \Bs{R}_{+} \cup \{+ \infty\}$ is convex, and satisfy the following properties: $\phi(1)=0$\footnote{The assumption $\phi(1)=0$ is without loss of generality because the function $\psi(t)= \phi(t)+ c (t-1)$ yields  identical discrepancy measure to $\phi$ \cite{pardo2005}.}, $0\phi\left(\frac{0}{0}\right):=0$, and $a\phi\left(\frac{a}{0}\right):=a \lim_{t \rightarrow \infty} \frac{\phi(t)}{t}$ if $a>0$. Note that a  $\phi$-divergence does not necessarily induce a metric on the underlying space. For detailed information on  $\phi$-divergences, we refer to \citet{read1988,vajda1989,pardo2005}.

A $\phi$-divergence can be used to model the distributional ambiguity as follows:
\begin{equation}
\label{eq: rev.phi_set}
\Cs{P}^{\phi}(P_{0};\epsilon):= \sset*{P \in \Fs{M}\measurespace}{ \Fs{d}^{\phi}(P, P_{0}) \le \epsilon},
\end{equation}
where as before $P_{0}$ is a nominal probability measure and $\epsilon$ controls the size of the ambiguity set. Table \ref{T: rev.phi} presents a list of commonly used $\phi$-divergence functions in \dro\ and their conjugate functions $\phi^{*}$. 

Before we review the papers that model the distributional ambiguity via the $\phi$-divergences, we present a  duality result on $\sup_{P \in \Cs{P}^{\phi}(P_{0};\epsilon)} \ \ee{P}{h(\bs{x},\txi)}$. 

\begin{theorem}
    \label{thm: rev.phi_duality} 
    Suppose that $\epsilon >0$ in \eqref{eq: rev.phi_set}. Then, for a fixed $\bs{x} \in \Cs{X}$, we have 
    $$\sup_{P \in \Cs{P}^{\phi}(P_{0};\epsilon)} \ \ee{P}{h(\bs{x},\txi)} = \inf_{(\lambda, \mu) \in \Lambda_{\phi,h(\bs{x}, \cdot)} } \  \left\lbrace  \mu + \lambda \epsilon +  \int_{\Xi} (\lambda\phi)^{*}( h(\bs{x}, s) -\mu )  P_{0}(ds) \right\rbrace,$$
    where $\Lambda_{\phi,h(\bs{x}, \cdot)}:=\sset*{(\lambda, \mu)}{\lambda \ge 0, \ h(\bs{x}, s) -\mu -\lambda \lim_{t \rightarrow \infty} \frac{\phi(t)}{t} \le 0, \forall  s \in \Xi}$, with the interpretation that $(\lambda\phi)^{*}(a)=\lambda\phi^{*}(\frac{a}{\lambda})$ for $\lambda \ge 0$. Here, $(0\phi)^{*}(a)=0\phi^{*}(\frac{a}{0})$, which equals to  $0$ if $a\le 0 $ and $+\infty$ if $a>0$. 
\end{theorem}

The above result can be obtained by taking the Lagrangian dual of   \linebreak $\sup_{P \in \Cs{P}^{\phi}(P_{0};\epsilon)} \ \ee{P}{h(\bs{x},\txi)}$, and  we refer the readers to \citet{ben2013,bayraksan2015,love2013} for a detailed derivation. 

% \usepackage{booktabs}
\begin{table}[!]
\footnotesize
\centering
\caption{Examples of $\phi$-divergence functions, their conjugates $\phi^{*}(a)$, and their \dro\ counterparts}
\begin{adjustbox}{max width=\textwidth}
\begin{tabular}{lllllll}
\toprule
Divergence  & $\phi(t)$  & $\phi(t), \; t \ge 0$  & $\Fs{d}^{\phi}(P_{1}, P_{2})$ & $\phi^{*}(a)$ & \dro\ Counterpart\\
\midrule
Kullback-Leibler    & $\phi_{\text{kl}}(t)$ & $t \log t -t +1$  & $\int_{\Xi} \log\left(\frac{d P_{1}}{d P_{2}}\right) d P_{1}$   &   $e^{a}-1$ & Convex program\\
Burg entropy & $\phi_{\text{b}}(t)$  & $- \log t + t -1 $ & $\int_{\Xi} \log\left(\frac{d P_{2}}{d P_{1}}\right) d P_{2}$   & $-\log(1-a), \; a<1$  & Convex program\\
$J$-divergence & $\phi_{\text{j}}(t)$  & $(t-1) \log t$      & $\int_{\Xi} \log\left(\frac{d P_{1}}{d P_{2}}\right) (d P_{1} - d P_{2})$ & No closed form  & Convex program\\
$\chi^{2}$-distance      & $\phi_{\text{c}}(t)$  & $\frac{1}{t}(t-1)^{2}$     & $\int_{\Xi} (\frac{(d P_{1} - d P_{2})^{2}}{d P_{1}}$  & $2-2\sqrt{1-a}, \; a<1$  & SOCP\\
Modified $\chi^{2}$-distance   & $\phi_{\text{mc}}(t)$ & $(t-1)^{2}$     & $\int_{\Xi} (\frac{(d P_{1} - d P_{2})^{2}}{d P_{2}}$   & $\begin{cases}
-1 \quad & a<-2\\
a +\frac{a^{2}}{4} \quad & a \ge -2
\end{cases}$  & SOCP\\
Hellinger distance    & $\phi_{\text{h}}(t)$  & $(\sqrt{t}-1)^{2}$     & $\int_{\Xi} (\sqrt{d P_{1}} - \sqrt{d P_{2}})^{2}$    & $\frac{a}{1-a}, \; a<1$ & SOCP\\
$\chi$-divergence of order $\theta>1$ & $\phi_{\text{ca}^{\theta}}(t)$ & $|t-1|^{\theta}$                                          & $\int_{\Xi} |1- \frac{d P_{1}}{d P_{2}}|^{\theta}d P_{2}$      & $a+(\theta-1)\left(\frac{|a|}{\theta}\right)^{\frac{\theta}{\theta-1}}$ & SOCP\\
Variation distance  & $\phi_{\text{v}}(t)$ & $|t-1|$  & $\int_{\Xi} |d P_{1} - d P_{2}|$   & 
$\begin{cases}
-1 \quad & a\le -1 \\
a  \quad & -1 \le a \le 1
\end{cases}$ & LP    \\
Cressie-Read   & $\phi_{\text{cr}^{\theta}}(t)$ & $\frac{1-\theta+\theta t - t^{\theta}}{\theta(1-\theta)}$ & $\frac{1}{\theta(1-\theta)} (1-\int_{\Xi} d P_{1}^{\theta} d P_{2}^{1-\theta} )$  & $\frac{1}{\theta}\big(1-a(1-\theta)\big)^{\frac{\theta}{1-\theta}}-\frac{1}{\theta}, \; a < \frac{1}{1-\theta}$ & SOCP\\
\bottomrule
\end{tabular}
\end{adjustbox}
\label{T: rev.phi}
\end{table}

The robust counterpart
of linear and nonlinear optimization problems with an uncertainty set of parameters defined via general $\phi$-divergence is studied in \citet{ben2013}. As it is presented in Table \ref{T: rev.phi}, when  the uncertain parameter is a finite-dimensional probability vector, the robust counterpart is tractable for most of the choices of $\phi$-divergence function considered in the literature. 
The use of $\phi$-divergence to model the distributional ambiguity in \dro\ is  systematically introduced in \citet{bayraksan2015} and \citet{love2015}. To elucidate  the use of $\phi$-divergences for models with different sources of data and
decision makers with different risk preferences, they  present a classification of
$\phi$-divergences based on the notions of {\it suppressing} and {\it popping} a scenario. 
The situation that a scenario with a positive nominal probability  ends up having a zero worst-case probability is called suppressing. On the contrary, the situation that a scenario with a zero nominal probability  ends up having a positive worst-case probability is called popping.
These notions give rise to four categories of $\phi$-divergences. For example, they show that the variation distance can both suppress and pop scenarios, while Kullback-Leibler divergence  can only suppress scenarios. Furthermore, they analyze the value of data and propose a  decomposition algorithm to solve the dual of the resulting \dro\ model formed via a general $\phi$-divergence. 

%\paragraph{Chance Constraints}

%As mentioned in Section \ref{sec: rev.rel_chance}, some papers address the conservatism of RO in the context of \dro. 
Motivated by the difficulty in choosing the ambiguity set and the fact that all probability distributions in the set are treated equally (while those outside the set are completely ignored),  \citet{ben2010soft} propose to minimize the expected cost under the nominal distribution while the maximum expected cost over an infinite nested family of ambiguity sets, parametrized by $\epsilon$, is bounded from above. 
More specifically, they allow a varying level of feasibility for each family of probability distributions, where the maximum allowed expected cost for distributions in a set with parameter $\epsilon$ is proportional to $\epsilon$. They refer to this approach as {\it soft robust optimization} and relate the feasibility region induced by this approach to the convex risk measures. They illustrate that the ambiguity sets formed via $\phi$-divergences are related to an optimized certainty equivalent risk measure formed via $\phi$-functions \citep{bental2007OCE}. Furthermore, they show that the complexity of the soft robust approach is equivalent to that of solving a small
number of standard corresponding  \dro\ (i.e., \dro\ with one ambiguity set)  problems. In fact, by showing that standard \dro\ is concave in $\epsilon$, they solve the soft robust model by a bisection method. They also investigate how much larger a feasible region implied by the soft robust approach can cover compared to the standard \dro, without compromising the objective value. Furthermore, they study the downside probability guarantees implied by both the soft robust and standard robust approaches. They also apply their results to portfolio optimization and asset allocation problems.

A data-driven \dro\ approach to  chance-constrained problems modeled via $\phi$-divergences is studied in \citet{yanikoglu2012}. They  propose safe approximations to these ambiguous chance constraints. Their approach is capable of handling joint chance constraints, dependent uncertain parameter, and a general nonlinear function  $\bs{g}(\bs{x},\txi)$. 

\citet{hu2013ambiguous} and \citet{jiang2016chance} show that distributionally robust chance-constrained programs formed via $\phi$-divergences  can   be   transformed into a  chance-constrained problem under the nominal distribution but with an adjusted risk level. For a general $\phi$-divergence, a  bisection line search algorithm to obtain the perturbed risk level is proposed in \citet{hu2013ambiguous,jiang2016chance}. In addition, closed-form expressions for the adjusted risk level are obtained for the case of the variation distance (see, \citet{hu2013ambiguous} and \citet{jiang2016chance}), and Kullback-Leibler divergence and $\chi^2$-distance (see, \citet{jiang2016chance}). 
For the ambiguous probabilistic programs formed via $\phi$-divergences,  similar results to the chance-constrained programs are  shown in \citet{hu2013ambiguous}.  \citet{hu2013ambiguous} show that the ambiguous probability minimization problem can be transformed into a corresponding problem under the nominal distribution. %, and that the ambiguous chance-constrained problem can be transformed into a chance-constrained problem under the nominal distribution but with an adjusted risk level. 
In particular,
they show that these problems have the same complexity as the corresponding pure probabilistic programs. 
%They elaborate on how to obtain the adjusted level of risk for different $\phi$ functions. In particular, they present a closed-form expression for the variation distance.  For a general $\phi$-divergence, they propose a bisection line search algorithm to obtain the perturbed risk level. 
%Similar to \citet{hu2013ambiguous}, \citet{jiang2016chance}  derive an equivalent reformulation of distributionally robust chance constraints formed via $\phi$-divergences. They show that the resulting model is equivalent to a chance constraint under the nominal distribution with a perturbed risk level. They propose a bisection line search algorithm to obtain the perturbed risk level for general $\phi$-divergences. Moreover, they present closed-form expressions for Kullback-Leibler divergence, $\chi^2$-distance, and variation distance. 

\paragraph{Statistical Learning}

\citet{hu2018} study distributionally robust supervised learning, where the ambiguity set of distributions is formed via $\phi$-divergences. They prove that such a \dro\ model for a classification   problem gives a  classifier that is optimal for the  training set distribution rather than being robust against all distributions in the ambiguity set. They argue such a pessimism  comes from two sources: the particular
losses used in classification and the over-conservation
of the ambiguity set formed via  $\phi$-divergences. Motivated by this observation, they propose an ambiguity set that incorporates prior expert structural information on the distribution. More precisely, they introduce a latent variable from a prior distribution. While such a distribution can change in the ambiguity set, they leave the ambiguous joint distribution of data conditioned on the latent variable intact.   
\citet{duchi2016} show that the inner problem of  a data-driven \dro\ formed around the empirical distribution, with $\epsilon=\frac{\chi^{2}_{1, 1-\alpha}}{N}$ has an almost-sure asymptotic expansion. Such an expansion is equivalent to the expected cost under the empirical distribution plus  a regularization term that accounts for the standard deviation of the objective function. They also show that the set of the optimal solutions of the \dro\ model converges to that of the stochastic program under the true underlying distribution, provided that $h(\bs{x},\txi)$ is lower-semicontinuous. 

\paragraph{Specific $\phi$-Divergences}
In this section, we review papers that consider specific $\phi$-divergences. 

\subparagraph{Kullback-Leibler Divergence}

\citet{calafiore2007} investigates the optimal robust portfolio and worst-case distribution for a data-driven distributionally robust portfolio optimization problem with a mean-risk objective. Motivated by the application, they consider the variance and absolute deviation as measures of risk.

\citet{hu2012kullback} study a variety of distributionally robust optimization problems, where the ambiguity is in either the objective function or constraints. They  show that the ambiguous chance-constrained problem can be reformulated as a chance-constrained problem under the nominal distribution but with an adjusted risk level. They further show that when the chance safe region is bi-affine in $\bs{x}$ and $\txi$\footnote{Recall the discussion following \eqref{eq: SO_Obj} and \eqref{eq: SO_Cons}, where we gave a characterization of $A(\bs{x})$ as $\bs{a}(\bs{x})^{\top}\txi \le \bs{b}(\bs{x})$ and $\bs{a}(\txi)^{\top} \bs{x} \le \bs{b}(\txi)$. A safe region characterized by a bi-affine expression in $\txi$ and  $\bs{x}$ means that  both $\bs{a}(\bs{x})$ and $\bs{b}(\bs{x})$ are affine in $\bs{x}$ for the form $\bs{a}(\bs{x})^{\top}\txi \le \bs{b}(\bs{x})$, and both $\bs{a}(\txi)$ and $\bs{b}(\txi)$ are affine in $\txi$ for the form $\bs{a}(\txi)^{\top}\bs{x} \le \bs{b}(\txi)$.}
%\footnote{We say a function $f(\bs{x},\bs{\xi})$ is bi-affine if the function $\bs{\xi} \mapsto f(\bs{x},\bs{\xi})$ is affine for any fixed $\bs{x}$ and the function $\bs{x} \mapsto f(\bs{x},\bs{\xi})$ is affine for any fixed $\bs{\xi}$.}
, and the nominal distribution belongs to the exponential families of distributions, both the nominal and worst-case distribution belong to the same distribution family.

\citet{blanchet2018structural} study a \dro\ approach to extreme value analysis in order to estimate the tail distributions and consequently, extreme quantiles. They form the ambiguity set of distributions by the class of R\'{e}yni divergences  \citep{pardo2005}, that includes  Kullback-Leibler as a special case\footnote{The class of R\'{e}yni divergences is defined as  $\Fs{d}^{\text{R}}_{r}(P_{1}, P_{2}):=\frac{1}{1-r}\int_{\Xi}\left(\frac{d P_{1}}{d P_{2}}\right)^{r-1} d P_{1}$. This class is not a  $\phi$-divergence, but $\Fs{d}^{\text{R}}_{r}(P_{1}, P_{2})$ can be rewritten as $h(\Cs{D}_{\phi}(P_{1}, P_{2}))$, where $h(t)=\frac{1}{r-1}\log [(r-1)t +1] $ and $\phi(t)=\frac{t^{r}-r(t-1)-1}{r-1}$ \citep{pardo2005}.}.
Kullback-Leibler is also used for the \dro\ approach to hypothesis testing in \citet{levy2009,gul2017,gul2017asymptotically}.

\subparagraph{Burg Entropy} 

\citet{wang2016} model the distributional ambiguity via the Burg entropy to consider all probability distributions that make the observed data achieve a certain level of likelihood. They present statistical analyses of their model using Bayesian statistics and empirical likelihood theory. To test the
performance of the model, they apply it to the newsvendor problem and the portfolio selection problem. 

\citet{wiesemann2013} study Markov decision processes where the transition Kernel is known. They use Burg entropy to construct a confidence region that contains the unknown probability distribution  with a high probability, based on an observation history.
It is shown in \citet{lam2016} that a \dro\ model formed via  the Burg entropy around the empirical distribution of data gives rise to a confidence bound on the expected cost  
that recovers the exact asymptotic statistical guarantees provided by the Central Limit Theorem.

\subparagraph{$\chi^2$-Distance}

\citet{hanasusanto2013} propose a  robust data-driven dynamic programming approach which  replaces the expectations in the dynamic programming  recursions with worst-case expectations over an ambiguity set of distributions. Their motivation to propose such a scheme is to  mitigate the poor out-of-sample performance of the data-driven dynamic programming approach under sparse training data. The proposed method combines convex parametric function approximation methods (to model the dependence
on the endogenous state) with nonparametric kernel regression method (to model the dependence on the exogenous state). %where the conditional expectations are estimated via kernel regression () in conjunction with parametric value function approximations. 
They show the conditions under which the resulting \dro\ model, formed via   $\chi^2$-distance, reduces to a tractable conic  program. They apply their results to problems arising in index tracking and wind energy
commitment applications. 
\citet{klabjan2013} study optimal inventory control for a single-item multiperiod periodic review stochastic lot-sizing problem under uncertain demand, where the distributional ambiguity is modeled via $\chi^{2}$-distance. They show that the resulting model generalizes the Bayesian model, and it can be interpreted as minimizing demand-history-dependent risk measures.

\subparagraph{Modified $\chi^2$-Distance}

A {\it stochastic dual dynamic programming} (SDDP) approach to solve a distributionally robust multistage optimization model formed via the modified $\chi^2$-distance is porposed in \citet{philpott2018}.

\subparagraph{Variation Distance}

Variation distance, or $\ell_{1}$-norm, as defined in Table \ref{T: rev.phi}, can be used to safely approximate several ambiguity sets formed via $\phi$-divergences, including $\chi$-divergence of order 2, $J$-divergence, Kullback-Leibler divergence, and Hellinger distance. The following lemma states the above result more formally.
\begin{lemma}
    \label{lem: rev.TV}
    The following relationship holds between $\phi$-divergences, as defined in Table \ref{T: rev.phi}:
    \begin{equation}
        \label{eq: rev.TV}
        \frac{1}{4}\big(\Fs{d}^{\phi_{\text{v}}}(P ,P_{0}) \big)^{2} \le \Fs{d}^{\phi_{\text{h}}}(P , P_{0})   \le \Fs{d}^{\phi_{\text{kl}}}(P , P_{0})  \le \Fs{d}^{\phi_{\text{j}}}(P, P_{0})  \le \Fs{d}^{\phi_{\text{ca}^{2}}}(P , P_{0}), 
    \end{equation}
    which implies
    \begin{equation}
        \label{eq: rev.TV_set}
        \Cs{P}^{\phi_{\text{ca}^{2}}}(P_{0}; \epsilon)  \subseteq \Cs{P}^{\phi_{\text{j}}}(P_{0} ;\epsilon)   \subseteq \Cs{P}^{\phi_{\text{kl}}}(P_{0}; \epsilon)   \subseteq  \Cs{P}^{\phi_{\text{h}}}(P_{0}; \epsilon)  \subseteq \Cs{P}^{\phi_{\text{v}}}(P_{0};2\epsilon^{\frac{1}{2}}). 
    \end{equation}
\end{lemma}
\begin{proof}
The first two inequalities in \eqref{eq: rev.TV} can be found in e.g., \citet[p.~99]{reiss1989}\footnote{As shown for e.g., in \citet{reiss1989} and \cite{gibbs2002}, $\Fs{d}^{\phi_{\text{h}}}(P , P_{0})   \le \Fs{d}^{\phi_{\text{kl}}}(P , P_{0})$. However,  in \citet[Lemma~1]{jiang2016} this relationship has been shown incorrectly as $\Fs{d}^{\phi_{\text{h}}}(P ,P_{0})   \le \big(\Fs{d}^{\phi_{\text{kl}}}(P , P_{0})\big)^{\frac{1}{2}}$.}  and the last two inequalities can be found in e.g., \citet[Lemma~1]{jiang2016}. Then, \eqref{eq: rev.TV_set} follows from \eqref{eq: rev.TV}. 
\end{proof}

\subsubsection{Total Variation Distance} 
For two probability measures $P_{1}$, $P_{2} \in \Fs{M}\measurespace$, the total variation distance is defined as $d_{\text{TV}}(P_{1},P_{2}):=\sup_{A \in \Cs{F}} \ |P_{1}(A)-P_{2}(A)|$. When  $P_{1}$ and $P_{2}$ are absolutely continuous with respect to a measure $\nu \in \Fs{M}\measurespace$, with Radon-Nikodym derivaties $f_{1}$ and $f_{2}$, respectively, then, $\Fs{d}^{\text{TV}}(P_{1},P_{2})=\frac{1}{2} \int_{\Xi} |f_{1}(s)-f_{2}(s)| \nu(ds)$. 
Note that the total variation distance can be obtained from other classes of probability metrics: (1) it is a $\phi$-divergence with  $\phi(t)=\frac{1}{2}|t-1|$, (2) it is half of the $\ell_{1}$-norm, and (3) it is obtained from the optimal transport discrepancy \eqref{eq: rev.opt_transport} with 
\begin{equation}
    c(s_{1},s_{2})=\begin{cases} 0, & \text{if} \ s_{1}=s_{2},\\
1, & \text{if} \ s_{1} \neq s_{2}.
\end{cases}
\end{equation} 
The total variation distance can be used to model the distributional ambiguity as follows:
\begin{equation}
\label{eq: rev.tv_set}
\Cs{P}^{\text{TV}}(P_{0};\epsilon):= \sset*{P \in \Fs{M}\measurespace}{ \Fs{d}^{\text{TV}}(P,P_{0}) \le \epsilon},
\end{equation}
where as before $P_{0}$ is a nominal probability measure and $\epsilon$ controls the size of the ambiguity set. 

The total variation distance  between $P_{1}$ and $ P_{2}$ is also related to the {\it one-sided} variation distances $\frac{1}{2} \int_{\Xi} (f_{1}(s)-f_{2}(s))_{+} \nu(ds)$ and $\frac{1}{2} \int_{\Xi} (f_{2}(s)-f_{1}(s))_{+} \nu(ds)$ \cite{rahimian2019}, which are $\phi$-divergences with  $\phi(t)=\frac{1}{2}(t-1)_{+}$ and $\phi(t)=\frac{1}{2}(1-t)_{+}$, respectively. However, unlike the total variation distance, the one-sided variation distances are not a probability metric. 

Before we review the papers that model the distributional ambiguity via the total variation distance, we present a  duality result on $\sup_{P \in \Cs{P}^{\text{TV}}(P_{0};\epsilon)} \ \ee{P}{h(\bs{x},\txi)}$. 

\begin{theorem}{(\citet[Theorems~1--2]{jiang2018}, \citet[Proposition~3]{rahimian2019}, \citet{shapiro2017DRSP})}
    \label{thm: rev.tv_duality} 
    For a fixed $\bs{x} \in \Cs{X}$, we have
    \begin{equation*}
        \begin{split}
            \sup_{P \in \Cs{P}^{\text{TV}}(P_{0};\epsilon)} \ & \ee{P}{h(\bs{x},\txi)} \\
            & {}= \begin{cases}
	    \ee{P_{0}}{h(\bs{x},\txi)}, &  \epsilon=0,\\
	    \epsilon  \  \nu\textrm{-}\esssup_{s \in \Xi} h(\bs{x}, \txi(s)) + (1-\epsilon) \cccvar{P_{0}}{\epsilon}{h(\bs{x}, \txi)}, &  0<\epsilon<1,\\
	    \nu\textrm{-}\esssup_{s \in \Xi} h(\bs{x}, \txi(s)), &  \epsilon\ge 1,
	    \end{cases}    
	    \end{split}
    \end{equation*}
    where $\nu\textrm{-}\esssup_{s \in \Xi} h(\bs{x}, \txi(s))=\inf\Big\{a \in \Bs{R}: \nu\{s \in \Xi: h(\bs{x},\txi(s))>a)=0\} \Big\}$.
\end{theorem}

\begin{remark}{(\citet[Proposition~3]{rahimian2019}, \citet{shapiro2017DRSP})}
    Let $\Cs{P}^{\text{OTV}}(P_{0};\epsilon)$ denote the ambiguity set formed via either of the  one-sided variation distances. Then,  for a fixed $\bs{x} \in \Cs{X}$, 
    $\sup_{P \in \Cs{P}^{\text{TVO}}(P_{0};\frac{\epsilon}{2})}$ can be obtained by the right-hand side of the result in Theorem \ref{thm: rev.tv_duality}. 
\end{remark}

\citet{jiang2018} study distributionally robust two-stage stochastic programs  \linebreak formed via the  total variation distance. They discuss how to find the nominal probability distribution and analyze the convergence of the problem to the corresponding stochastic program under the true unknown probability distribution.   
\citet{rahimian2019} study distributionally robust convex optimization problems with a finite sample space. They study how the uncertain parameters affect the optimization. In order to  do so, they define the notion of ``effective" and ``ineffective" scenarios. According to their definitions, a subset of scenarios is effective if their removal from the support of the worst-case distribution, by forcing their probabilities to zero in the ambiguity set, changes the optimal value of the \dro\ problem. They propose easy-to-check conditions to identify the effective and ineffective scenarios for the case that the distributional ambiguity is modeled via the total variation distance. 
\citet{rahimian2019NV} extends the work of \citet{rahimian2019} to distributionally robust newsvendor problems with a continuous sample space. They derive a closed-form expression for the optimal solution and identify the maximal effective subsets of demands. %Furthermore, they propose to calibrate the level of robustness based on the notions of maximal effective subsets, as well as the price of optimism/pessimism and nominal/worst-case regrets. The latter notions evaluate the relative performance of the optimal solutions in the SO, RO, and DRO settings. 

\subsubsection{Goodness-of-Fit Test}

\citet{postek2016} review and derive computationally tractable reformulations of distributionally robust risk constraints over discrete probability
distributions for various risk measures and ambiguity sets formed using statistical
goodness-of-fit tests or probability metrics, including $\phi$-divergences, Kolmogrov-Smirnov, Wasserstein, Anderson-Darling, Cramer-von Mises, Watson, and Kuiper.  They exemplify the results in portfolio optimization and antenna
array design problems. 
\citet{bertsimas2018RO} and \citet{bertsimas2018SAA} propose a systematic view on how to choose statistical goodness-of-fit test to construct an ambiguity set of distributions that guarantee the implication \eqref{eq: rev.prob.guarantee} (recall Theorem \ref{thm: rev.chanceDRO}). They consider the situation that (i) $\trueP=P^{\text{true}} \circ \txi^{-1}$ may have continuous support, and the components of $\txi$ are independent, (ii)  $\trueP$ may have continuous support, and data are drawn from its marginal distributions
asynchronously, and (iii) $\trueP$ may have continuous support, and data are drawn from its joint distribution. They also study a wide range of statistical hypothesis tests, including $\chi^{2}$, G, Kolmogrov-Smirnov, Kuiper, Cramer-von Mises, Watson, and Anderson-Darling goodness-of-fit tests, and they characterize the geometric shape of the corresponding ambiguity sets.  

\subsubsection{Prohorov Metric}

For two probability measures $P_{1}, P_{2} \in \Fs{M}\measurespace$, the Prohorov metric is defined as $$\Fs{d}^{\text{p}}(P_{1},P_{2}):=\inf \sset*{\gamma >0 }{P_{1}\{A\} \le P_{2}\{A^{\gamma}\}+\gamma \ \text{and} \ P_{2}\{A\} \le P_{1}\{A^{\gamma}\}+\gamma \; \forall A \in \Cs{F}},$$ 
where $A^{\gamma}:=\sset*{s \in \Xi}{\inf_{s^{\prime} \in A} \ d(s,s^{\prime}) \le \gamma}$ \citep{gibbs2002}.  
The Prohorov metric takes values in $[0,1]$ and can be used to model the distributional ambiguity as follows:
\begin{equation}
\label{eq: rev.prohorov_set}
\Cs{P}^{\text{p}}(P_{0};\epsilon)  := 
\sset*{P \in \Fs{M}\measurespace}{ \Fs{d}^{\text{p}}(P,P_{0}) \le \epsilon},
\end{equation}
where as before $P_{0}$ is a nominal probability measure and $\epsilon$ controls the size of the ambiguity set. 
A specialization of the Prohorov metric to the univariate distributions is called {\it Levy} metric, which is defined as  \cite{gibbs2002} 
\begin{equation*}
    \begin{split}
        \Fs{d}^{\text{L}}&(P_{1},P_{2}) :=\\
        & \inf \sset*{\gamma >0 }{P_{2}\{(-\infty,t-\gamma]\} -\gamma \le P_{1}\{(-\infty,t]\} \le P_{2}\{(-\infty,t+\gamma]\}  +  \gamma, \; \forall t \in \Bs{R}}.
    \end{split}
\end{equation*}
The Levy metric can be used to model the distributional ambiguity as follows:
\begin{equation}
\label{eq: rev.levy_set}
\Cs{P}^{\text{L}}(P_{0};\epsilon) := \sset*{P \in \Fs{M}\measurespace}{ \Fs{d}^{\text{L}}(P,P_{0}) \le \epsilon}.
\end{equation}
%where as before $P_{0}$ is a nominal probability measure and $\epsilon$ controls the size of the ambiguity set. 

\citet{erdougan2006} study an optimization problem subject to a set of parameterized convex constraints. Similar to the argument in Section \ref{sec: rev.rel_chance}, they study a \dro\ approach to this problem, where the distributional ambiguity is modeled by the Prohorov metric. They also consider a scenario approximation scheme of the problem. By extending the work of \citep{campi2004,calafiore2005}, they  provide an upper bound on the number of samples required to  guarantee that the sampled problem is
a good approximation for the associated ambiguous
chance-constrained problem with a high probability.
%\citet{erdougan2006} interpret some of their results if the  variation distance is used instead of the Prohorov metric.

\subsubsection[Lp-Norm]{\texorpdfstring{$\ell_{p}$-Norm}{Lp-Norm}}

\citet{calafiore2006} study distributionally robust individual linear chance-constrained problem, and provide convex conditions that guarantee the satisfaction of the chance constraint   within the family of radially-symmetric nonincreasing densities whose supports are defined by means of the $\ell_{1}$- and $\ell_{\infty}$-norm\footnote{Consider the sets $\Cs{H}(\bs{A},\bs{\xi}_{0}):=\sset*{\bs{\xi}=\bs{\xi}_{0} + \bs{A} \omega}{\|\omega\|_{\infty} \le 1}$ and $\Cs{E}(\bs{B},\bs{\xi}_{0}):=\sset*{\bs{\xi}=\bs{\xi}_{0} + \bs{B} \omega}{\|\omega\|_{1} \le 1}$, where $\bs{A}$ is a diagonal positive-definite matrix and $\bs{B}$ is a positive-definite matrix. 
A random vector $\txi$ has a probability distribution $P$ within the class of radially-symmetric nonincreasing densities supported on  $\Cs{H}(\bs{A},\bs{\xi}_{0})$ (respectively, $\Cs{E}(\bs{B},\bs{\xi}_{0})$) if $\txi - \ee{P}{\txi}= \bs{A} \omega$ (respectively, $\txi - \ee{P}{\txi}= \bs{B} \omega$), where $\omega$ is a random vector having the probability density $f_{\omega}$ such that 
$f_{\omega}(\omega)= t(\|\omega\|_{\infty})$ for $\|\omega\|_{\infty} \le 1$ and $0$ otherwise (respectively, $f_{\omega}(\omega)= t(\|\omega\|_{1})$ for $\|\omega\|_{1} \le 1$ and $0$ otherwise) and $t(\cdot)$ is a nonincreasing function. The class of radially-symmetric distributions contains for example Gaussian, truncated Gaussian, uniform distribution on ellipsoidal support, and nonunimodal densities \citep{calafiore2006}}.
\citet{mevissen2013} study distributionally robust polynomial optimization, where the  distribution of the uncertain parameter is estimated using polynomial basis functions via the $\ell_{p}$-norm. They show that
the optimal value of the problem is the limit of a sequence of
tractable SDP relaxations of polynomial optimization problems. They also provide a  finite-sample  consistency  guarantee for the data-driven uncertainty sets, and an asymptotic guarantee on the solutions of the SDP relaxations. They apply their techniques to  a water network optimization problem.

\citet{jiang2018} study  distributionally robust two-stage  stochastic programs \linebreak formed via $\ell_{\infty}$-norm. 
\citet{huang2017} study extend the work of \citet{jiang2018} to the multistage setting. They formulate the problem into a problem that contains a convex combination of expectation and CVaR in the objective function of each stage  to remove the nested multistage minmax structure in the objective function. They  analyze the convergence of the resulting \dro\ problem to  the corresponding multistage stochastic program under the true unknown probability distribution. They test their results on the hydrothermal scheduling problem.  

\subsubsection[Zeta-Structure Metrics]{\texorpdfstring{$\zeta$-Structure Metrics}{Zeta-Structure Metrics}}

%Let us  begin this section by defining $\zeta$-structure metrics. 
Consider  $P_{1}, P_{2} \in \Fs{M}\measurespace$ and let $\Cs{Z}$ be a family of real-valued measurable functions $z: \probspace{d} \mapsto (\Bs{R},\Fs{B}(\Bs{R}))$. The   $\zeta$-structure metric is defined as $\Fs{d}^{\Cs{Z}}(P_{1},P_{2}):=\sup_{ z \in \Cs{Z}} \Big| \ee{P_{1}}{z(\txi)} - \ee{P_{2}}{z(\txi)} \Big|$. A wide range of metrics in probability theory can be written as special cases of the above family of metrics \cite{zhao2015,pichler2017}. Let us introduce them below. 

\begin{itemize}
    \item {\it Total variation metric} $\Fs{d}^{\text{TV}}(P_{1},P_{2})$: $$\Cs{Z}=\sset*{z}{\|z\|_{\infty} \le 1}, $$ where $\|z\|_{\infty}= \sup_{\bs{\xi}  \in \Omega } \ |z(\bs{\xi})|$.

    \item {\it Bounded Lipschitz metric} $\Fs{d}^{\text{BL}}(P_{1},P_{2})$:	   $$\Cs{Z}=\sset*{z}{\|z\|_{\infty} \le 1, \; z \ \text{is Lipschitz continuous}, \; L_{1}(z) \le 1 },$$ where $L_{1}(z)=: \sup \sset*{|z(\bs{u})-z(\bs{v})|/d(\bs{u},\bs{v})}{ \bs{u} \neq \bs{v} }$, is the Lipschitz modulus. 

    \item {\it Kantorovich metric} $\Fs{d}^{\text{K}}(P_{1},P_{2})$: $$\Cs{Z}=\sset*{z}{z \; \text{is Lipschitz continuous}, \; L_{1}(z) \le 1 }.$$

    \item {\it Fortet-Mourier metric} $\Fs{d}^{\text{FM}}(P_{1},P_{2})$: $$\Cs{Z}=\sset*{z}{z \; \text{is Lipschitz continuous}, \; L_{q}(z) \le 1},$$ where   
    \begin{equation*}
        \begin{split}
                L_{q}& (z)=: \\
                & \inf \sset*{L}{ |z(\bs{u})-z(\bs{v})| \le L \cdot d(\bs{u},\bs{v}) \cdot \max(1, \|\bs{u}\|^{q-1}, \|\bs{v}\|^{q-1}), \forall \bs{u},\bs{v} \in \Bs{R}^{d} },
        \end{split}
    \end{equation*}
    with $\|\cdot\|$ as the Euclidean norm. Note that when $q=1$, Fortet-Mourier metric is the same as the Kantorovich metric.

    \item {\it Uniform (Kolmogorov) metric} $\Fs{d}^{\text{U}}(P_{1},P_{2})$:  $$\Cs{Z}=\sset*{z}{z=\mathbbm{1}_{(-\infty, t]}, \; t \in \Bs{R}^{n} }.$$

\end{itemize}

The class of $\zeta$-structure metrics  may be used to model the distributional ambiguity as follows:
\begin{equation}
\label{eq: rev.zeta_set}
\Cs{P}^{\Cs{Z}}(P_{0};\epsilon):= \sset*{P \in \Fs{M}\measurespace}{ \Fs{d}^{\Cs{Z}}(P,P_{0}) \le \epsilon},
\end{equation}
where as before $P_{0}$ is a nominal probability measure and $\epsilon$ controls the size of the ambiguity set. 

\begin{lemma}
    \label{lem: rev.zeta}
    Suppose that the support $\Omega$ of $\txi$ is bounded with diameter $\theta$, i.e.,   $\theta:=\sup\{d(\bs{\xi}_{1}, \bs{\xi}_{2}): \bs{\xi}_{1}, \bs{\xi}_{2} \in \Omega\}$, where $d$ is metric. Then, the following relationship holds between $\zeta$-structure metrics:
    \begin{subequations}
        \label{eq: rev.zeta}
        \begin{align*}
            & \Fs{d}^{\text{BL}}(P,P_0) \le \Fs{d}^{\text{K}}(P,P_0)\\
            & \Fs{d}^{\text{K}}(P,P_0) \le \Fs{d}^{\text{TV}}(P,P_0)\\
            & \Fs{d}^{\text{U}}(P,P_0) \le \Fs{d}^{\text{TV}}(P,P_0)\\
            & \Fs{d}^{\text{K}}(P,P_0) \le \Fs{d}^{\text{FM}}(P,P_0)\\
            & \Fs{d}^{\text{FM}}(P,P_0) \le \max\{1, \theta^{q-1}\}\Fs{d}^{\text{K}}(P,P_0).
        \end{align*}
    \end{subequations}
\end{lemma}

\begin{proof}
The proof is immediate from \citet[Lemmas~1--4]{zhao2015}.
\end{proof}
\citet{zhao2015} study distributionally robust two-stage stochastic programs via $\zeta$-structure  metrics. They discuss how to construct the ambiguity set from historical data while utilizing a family of $\zeta$-structure  metrics. They propose solution approaches to solve the resulting problem, where the true unknown distribution is discrete or continuous. They further analyze the convergence of the DRO problem to  the corresponding stochastic program under the true unknown probability distribution. They test their results on newsvendor and facility location problems.  

\citet{pichler2017} study a \dro\ model with a expectation as the risk measure and form the ambiguity set of distribution via $\zeta$-structure metric. They investigate how the variation of the ambiguity set would affect the optimal value and the optimal solution in the resulting optimization problem. They illustrate their results in the context of a two-stage stochastic program with recourse.

\subsubsection{Contamination Neighborhood}
The contamination neighborhood \linebreak around a nominal probability measure $P_{0}$ is defined as 
\begin{equation}
    \label{eq: rev.contamination_set}
	\Cs{P}^{\text{c}}(P_{0};\epsilon) = \sset*{P \in \Fs{M}\measurespace} {P= (1-\epsilon) P_{0} + \epsilon Q, \; Q \in  \Fs{Q}},
\end{equation}
where $\Fs{Q} \subseteq \Fs{M}\measurespace$ and $\epsilon \in [0,1]$. 

This ambiguity set is extensively used in the context of robust statistics, see, e.g., \citet{huber1973,huber2009RobustStat},  and it has also been used in the economics literature, see, e.g., \citet{nishimura2004search,nishimura2004contamination}. 
\citet{bose2009} study ambiguity aversion in a mechanism design problem using a maximin expected utility model of \citet{gilboa1989}. %They model the ambiguity in the problem via the $\epsilon$-contamination neighborhood.
The contamination neighborhood is also used in the context of statistical learning, see, e.g., \citet{duchi2019} and hypothesis testing, see, e.g.,  \citet{huber1965}. 

%\subsubsection{Levy Metric}

\subsubsection{General Discrepancy-Based Ambiguity Sets}

We devote  this subsection to the papers that consider general discrepancy-based models. 
\citet{postek2016} review and derive tractable reformulations of distributionally robust risk constraints over  discrete probability
distributions and for function $\bs{g}(\bs{x},\txi)$ in $\txi$. They provide a comprehensive list  for  risk measures and ambiguity sets, formed  using statistical
goodness-of-fit tests or probability metrics. 
They consider risk measures such as (1) expectation, (2) sum of expectation and standard deviation/variance, (3) variance, (4) mean absolute deviation from the median, (5) Sharpe ratio, (6) lower partial moments, (7) certainty equivalent, (8) optimized certainty equivalent, (9) shortfall risk, (10) VaR, (11) CVaR, (12) entropic VaR, (13) mean absolute deviation from the mean, (14) distortion risk measures, (15) coherent risk measures, and (16) spectral risk measures.  They also consider (1) $\phi$-divergences, (2) Kolmogrov-Smirnov, (3) Wasserstein, (4) Anderson-Darling, (5) Cramer-von Mises, (6) Watson, and (7) Kuiper to model the distributional ambiguity. For each pair of risk measure and ambiguity set, they obtain a tractable reformulation by relying on the conjugate duality for the risk measure and the support function of the ambiguity set (i.e., the convex conjugate of the
indicator function of the ambiguity set). They exemplify the results in portfolio optimization and antenna array design problems. 

A connection between \dro\ models formed via  discrepancy-based ambiguity sets and law invariant risk measures is made in \citet{shapiro2017DRSP} as described in Theorem \ref{thm: duality_rho_law}. They specifically derive law invariant risk measures  for cases when Wasserstein metric, $\phi$-divergences, and total variation distance is  used to model the distributional ambiguity. They also propose a SAA approach to solve the corresponding dual of these problems, and establish the statistical properties of the optimal solutions and optimal value, similar to the results for the risk-neutral stochastic programs, see, e.g.,  \citet{shapiro2014SP,shapiro2003montecarlo}. 

%%%%%%%%%%%%%%%%%%%%%%%%%%%%%%%%%%%%%
\subsection{Moment-Based Ambiguity Sets}
\label{sec: rev.moment}
%%%%%%%%%%%%%%%%%%%%%%%%%%%%%%%%%%%%%

A common approach to model the ambiguity set is moment based, in which the ambiguity set contains all probability distributions whose moments satisfy certain properties. We categorize this type of models into several subgroups, although there are some overlaps. %These subgroups are not disjoint, and there are papers that can belong to more than one subgroups. In these cases, we have tried to 
%We review these papers in the subgroup that the authors  claim their work belong to, when a paper introduces concepts that overlap across subgroups.  

\subsubsection{Chebyshev}
\label{sec: rev.Chebyshev}

\citet{scarf1958} models the distributional ambiguity in a \linebreak newsvendor problem, where only the mean and variance of the random demand is known. He obtains a closed-form expression for the optimal order quantity and shows that the worst-case probability distribution is supported on only two points. 
Motivated by the Scarf's seminal work, other researchers have investigated the Chebyshev ambiguity set in the context of the newsvendor model. 
\citet{gallego1993} study multiple extensions of the problem studied in \citet{scarf1958}. These include the situations where there is a recourse opportunity, a fixed ordering cost, a random production output,  and a scare resource for multiple competing products.

Unlike the ambiguity sets studied in \citet{scarf1958} and \citet{gallego1993}, the mean and covariance matrix can be unknown themselves and belong to some uncertainty sets.  
\citet{ghaoui2003worst} study a distributionally robust one-period portfolio optimization, where the worst-case VaR  over an ambiguity set of distributions with a known mean and covariance matrix is minimized. They show that this problem can be reformulated as a SOCP. Moreover, they show that minimizing worst-case VaR with respect to such an ambiguity set can be interpreted as a RO model where the worst-case portfolio loss with respect to an ellipsoid uncertainty set is minimized. 
They extend their study to the case that the first two order moments are only known to belong to a convex (bounded) uncertainty set, and they show the conditions under which the resulting model can be cast as  a SDP. In particular, for  independent polytopic uncertainty sets for the mean and covariance (so that the mean and covariance belong to the Cartesian product of these two sets), the problem can be reformulated as a SOCP. Also, for sets with componentwise bound on the mean and covariance, they cast the problem as a SDP (see also \citet{halldorsson2003} for a similar result). Moreover, they show that in the presence of  additional  information on the distribution, besides the first two order moments, including  constraints on the support and Kullback-Leibler divergence, an upper bound on the worst-case VaR can be obtained by solving a SDP. 
Motivated by the work in \citet{ghaoui2003worst}, \citet{li2016law}  showcases the results in the context of a risk-averse portfolio optimization problem.  
Unlike \citet{ghaoui2003worst} that considers polytopic and interval uncertainty sets for the mean and covariance, \citet{lotfi2018} assume that the unknown mean and covariance belong to an ellipsoidal uncertainty set. They study the worst-case VaR and worst-case CVaR optimization problems, subject to an expected return constraint. %The ambiguity set of distribution contains all distributions with fixed first two order moments, taking values from the ellipsoidal set. 
They show that both problems can be reformulated as  SOCPs. 

\citet{goldfarb2003} study a distributionally robust portfolio selection problem, where the asset returns $\txi$ are formed by a linear factor model of the form  $\txi=\bs{\mu} + \bs{A} \tbs{f} + \tbs{\epsilon}$, where $\bs{\mu}$ is the vector of mean returns, $\tbs{f} \sim N(\bs{0}, \bs{\Sigma})$ is the vector of random returns that derives the market, $\bs{A}$ is the factor loading matrix, and $\tbs{\epsilon} \sim N(\bs{0}, \bs{B})$ is the vector of residual returns with a diagonal matrix $\bs{B}$. It is assumed that $\tbs{\epsilon}$ is independent of $\tbs{f}, \bs{F}$, and $\bs{B}$. Thus, $\txi \sim N(\bs{\mu}, \bs{A}\bs{\Sigma}\bs{A}^{\top} + \bs{B})$; hence, the  uncertainty in the mean is independent of  the uncertainty in the covariance matrix of the returns. 
Under the assumption that the covariance matrix $\bs{\Sigma}$ is known, \citet{goldfarb2003} study three different models to form the uncertainty in $\bs{B}$, $\bs{A}$, and $\bs{\mu}$ as follows:
\begin{align}
    & \Cs{U}_{\bs{B}}=\sset*{\bs{B}}{\bs{B}=\text{diag}(\bs{b}), \; b_{i} \in [\ul{b}_{i}, \ol{b}_{i}], \; i=1, \ldots, d}, \label{eq: A}\\
    & \Cs{U}_{\bs{A}}=\sset*{\bs{A}}{\bs{A}=\bs{A}_{0}+ \bs{C}, \; \|\bs{c}_{i}\|_{g} \le \rho_{i}, \; i=1, \ldots, d}, \label{eq: B}\\
    & \Cs{U}_{\bs{\mu}}=\sset*{\bs{\mu}}{\bs{\mu}=\bs{\mu}_{0}+ \bs{\zeta}, \; |\zeta_{i}| \le \gamma_{i}, \; i=1, \ldots, d}, \label{eq: mu}
\end{align}
where $\bs{c}_{i}$ denotes the $i$-th column of $\bs{C}$, and $\|\bs{c}_{i}\|_{g}=\sqrt{\bs{c}_{i}^{\top} \bs{G} \bs{c}_{i}^{\top}}$ denotes the elliptic norm of $\bs{c}_{i}$ with respect to a symmetric positive definite matrix $\bs{G}$. 
Calibrating the uncertainty sets $\Cs{U}_{\bs{B}}$, $\Cs{U}_{\bs{A}}$, and $\Cs{U}_{\bs{\mu}}$ involves choosing parameters $\ul{d}_{i}$, $\ol{d}_{i}$, $\rho_{i}$, $\gamma_{i}$, $i=1, \ldots, d$, vector $\bs{\mu}_{0}$, and matrices $\bs{A}_{0}$ and $\bs{G}$. 
Given this setup, \citet{goldfarb2003} study a \dro\ approach to different portfolio optimization problems for the return $\txi^{\top} \bs{x}$ on the portfolio $\bs{x}$, where $\sum_{i=1}^{n} x_{i}=1$. This includes:
(1) minimum variance, $\Var{\cdot}$, subject to a minimum expected return constraint
$$\min_{\bs{x} \ge \bs{0}} \ \max_{\bs{A} \in \Cs{U}_{\bs{A}}, \bs{B} \in \Cs{U}_{\bs{B}}} \sset*{\Var{\txi^{\top} \bs{x}}}{\min_{\bs{\mu} \in \Cs{U}_{\bs{\mu}}} \e{\txi^{\top} \bs{x}} \ge \alpha , \; \sum_{i=1}^{n} x_{i}=1}, $$
(2) maximum expected return subject to a maximum variance constraint 
$$\max_{\bs{x} \ge \bs{0}} \ \min_{\bs{\mu} \in \Cs{U}_{\bs{\mu}}} \sset*{\e{\txi^{\top} \bs{x}}}{\max_{\bs{A} \in \Cs{U}_{\bs{A}}, \bs{B} \in \Cs{U}_{\bs{B}}} \Var{\txi^{\top} \bs{x}} \le \lambda, \; \sum_{i=1}^{n} x_{i}=1},  $$
(3) maximum Sharpe ratio 
$$\max_{\bs{x} \ge \bs{0}} \ \min_{\bs{\mu} \in \Cs{U}_{\bs{\mu}}, \bs{A} \in \Cs{U}_{\bs{A}}, \bs{B} \in \Cs{U}_{\bs{B}}} \sset*{\frac{\e{\txi^{\top} \bs{x}}-\bs{\xi}_{0}^{\top} \bs{x}}{\sqrt{\Var{\txi^{\top}\bs{x}}}}}{\sum_{i=1}^{n} x_{i}=1},$$
where $\bs{\xi}_{0}$ is a risk-free return rate, and (4) maximum expected return subject to a maximum VaR constraint
$$\max_{\bs{x} \ge \bs{0}} \ \min_{\bs{\mu} \in \Cs{U}_{\bs{\mu}}} \sset*{\e{\txi^{\top} \bs{x}}}{\max_{\bs{\mu} \in \Cs{U}_{\bs{\mu}}, \bs{A} \in \Cs{U}_{\bs{A}}, \bs{B} \in \Cs{U}_{\bs{B}}} \vvar{\beta}{\txi^{\top} \bs{x}} \ge \alpha, \; \sum_{i=1}^{n} x_{i}=1}.  $$
Note that the constraint $\vvar{\beta}{\txi^{\top} \bs{x}} \ge \alpha$ is equivalent to $P\{\txi^{\top} \bs{x} \le \alpha\} \le \beta$. 
They show that all the above four classes of problems can be reformulated as SOCPs. They further assume the covariance matrix $\bs{\Sigma}$ or its inverse are unknown and belong to ellipsoidal uncertainty sets, and  show that  the above problems can be reformulated as SOCPs. 
\citet{ghaoui2003worst} study a similar linear factor model as the one in \citet{goldfarb2003}, but they assume that the  uncertainty in the mean is not independent  
of the uncertainty in the covariance matrix of the
returns. When the factor matrix $\bs{A}$ belongs to ellipsoidal uncertainty set, they show that an upper bound on the worst-case VaR can be computed by solving a SDP.

 \citet{li2013} study a distributionally robust approach for a single-period portfolio selection problem. They consider a set of reference means and variances, and they form the ambiguity set by all distributions whose means and variance are in a pre-specified distance from the reference means and variances set (in the regular sense of a point from a set via a norm). For the case that moments take values outside
the reference region, since evaluation based on its worst-case performance can be overly-conservative,
they consider a penalty term that further accounts for measure discrepancy between the moments in and outside the reference region. Moreover, for the case that the reference region is a conic set, they obtain an equivalent SDP reformulation. 

\citet{grunwald2004game} confine  the ambiguity set to distributions with fixed first order moments $\bs{\tau}$. By varying $\bs{\tau}$, they obtain a collection of maximum generalized entropy distribution and relate it to the exponential family of distributions.

\citet{rujeerapaiboon2018chebyshev} derive Chebyshev-type bounds on the worst-case right and left tail of a product of 
nonnegative symmetric  random variables. They assume that the mean is known, but the covariance matrix might be known or bounded above by a matrix inequality. They show that if both the mean and covariance matrix are known, these bounds can
be obtained by solving a SDP. For the case that the covariance matrix is bounded above, they show that (i) the bound on the left tail is equal to the  bound on the left tail under the known covariance setting, and (ii) the bound on the right tail is equal to the bound on the right tail under the known mean and covariance setting, for a sufficiently large tail. They extend their results to construct Chebyshev bounds for sums, minima,
and maxima of nonnegative random variables. %, using some permutation-symmetric functional. 

\subsubsection{Delage and Ye}
\label{sec: rev.DelageYe}

%In an attempt to unify all the past research on the Chebyshev ambiguity set, \citet{delage2010} study \dro\ problems where the distributional ambiguity is modeled through the information on the support and the first two order moments. However, 
Unlike the ambiguity sets studied in \citet{scarf1958} and \citet{gallego1993}, \citet{delage2010} allow the mean and covariance matrix to be unknown themselves. This ambiguity set is defined as follows \citep{delage2010}: 
\begin{equation}
\label{eq: rev.DelageYe}
\Cs{P}^{DY}:=\sset*{P \in\Fs{M}\measurespace}{ 
\begin{aligned}  
& P\{\txi \in \Omega\}=1,\\
& \Big(\ee{P}{\txi} - \bs{\mu}_{0}\Big)^{\top}\bs{\Sigma}_{0}^{-1}\Big(\ee{P}{\txi} - \bs{\mu}_{0}\Big) \le \varrho_{1},\\  	
& \ee{P}{(\txi-\bs{\mu}_{0})(\xi-\bs{\mu}_{0})^{\top}} \preccurlyeq \varrho_{2} \bs{\Sigma}_{0}  
\end{aligned}  
}.
\end{equation} 
The first constraint denotes the smallest closed convex set $\Omega \subseteq \Bs{R}^{d}$ that contains $\txi$ with probability one (w.p. $1$), i.e., $\Omega$ is the support of $\Ts{P}=P \circ \txi^{-1}$ w.p. $1$. 
The second constraint  ensures that the mean of $\txi$ lies in an ellipsoid of size $\varrho_{1}$ and centered around the nominal mean estimate $\bs{\mu}_{0}$. Note that we can equivalently write this constraint as 
$$  \ee{P}{\begin{pmatrix}
    -\bs{\Sigma}_{0}       & \bs{\mu}_{0} -\txi \\
    (\bs{\mu}_{0} -\txi)^{\top} & -\varrho_{1}
\end{pmatrix}} \preccurlyeq \bs{0}.$$
The third constraint defines the  second central-moment matrix of $\txi$ %lies in a positive semidefinite cone defined 
by a matrix inequality. The parameters $\varrho_{1}$ and $\varrho_{2}$ control the level of confidence in    $\bs{\mu}_{0}$ and  $ \bs{\Sigma}_{0}$, respectively. 
Note that the ambiguity sets with a known mean and covariance matrix  can be seen as a special case of \eqref{eq: rev.DelageYe}, with $\varrho_{1}=0$ and $\varrho_{2}=1$. 
\citet{delage2010}   propose data-driven methods to form confidence regions for the mean and the covariance matrix of the random vector $\txi$ using the concentration inequalities of \citet{mcdiarmid1998}, and provide probabilistic guarantees
that the solution found using the resulting \dro\ model 
yields an upper bound on the out-of-sample performance with respect to the true distribution of the random
vector. 
A conic generalization of the ambiguity set $\Cs{P}^{\text{DY}}$, beyond the first and second moment information is also studied in  \citet{delage2009DY}.  
Below, we present a duality result for $\sup_{\Ts{P} \in \Cs{P}^{\text{DY}}} \ \ee{P}{h(\bs{x},\txi)}$ given a fixed $\bs{x} \in \Cs{X}$, due to \citet{delage2010}. % show that the resulting problem is equivalent to a semi-infinite convex conic optimization problem as follows. 
\begin{theorem}{(\citet[Lemma~1]{delage2010})}
    \label{thm: rev.dual_DelageYe}
    For a fixed $\bs{x} \in \Cs{X}$, suppose that Slater's constraint qualification conditions are satisfied, i.e., there exists a strictly feasible $P$ to  $\Cs{P}^{DY}$, and $h(\bs{x}, \txi)$ is $P$-integrable for all  $P \in \Cs{P}^{DY}$. Then, \linebreak $\sup_{P \in \Cs{P}^{\text{DY}}} \ \ee{P}{h(\bs{x},\txi)}$ is equal to the optimal value of the following semi-infinite convex conic optimization  problem:
    \begin{align*}
        \inf_{\bs{Y},\bs{y},r,t} \ & r+ t \\
        \st \quad & r \ge h(\bs{x},\bs{\xi})-\bs{\xi}^{\top}\bs{Y}\bs{\xi} - \bs{\xi}^{\top} \bs{y}, \; \forall \bs{\xi} \in \Omega,\\
        & t \ge (\varrho_{2} \bs{\Sigma}_{0} + \bs{\mu}_{0} \bs{\mu}_{0}^{\top})\bullet \bs{Y} + \bs{\mu}_{0}^{\top} \bs{y} + \sqrt{\varrho_{1}} \| \bs{\Sigma}_{0} ^{\frac{1}{2}}(\bs{y}+ 2 \bs{Y}\bs{\mu}_{0})\|,\\
        & \bs{Y} \succcurlyeq 0,
    \end{align*}
    where $\bs{Y} \in \Bs{R}^{d \times d}$ and $\bs{y} \in \Bs{R}^{d}$. 
\end{theorem}

The reformulated problem in Theorem \ref{thm: rev.dual_DelageYe} is  polynomial-time solvable under the following assumptions \citep{delage2010}:
\begin{itemize}
	\item The sets $\Cs{X}$ and $\Omega$ are convex and compact, and are both equipped with oracles that confirm the feasibility of a point $\bs{x}$ and $\txi$, or provide  a hyperplane that separates the infeasible point from its corresponding feasible set in time polynomial in the dimension of the set. 
	\item Function $h(\bs{x},\txi):=\max_{k \in \{1,\ldots, K\}} h_{k}(\bs{x},\txi)$ is piecewise and is such that for each $k$,  $h_{k}(\bs{x},\txi)$ is convex in $\bs{x}$ and concave in $\txi$. In addition, for any given pair $(\bs{x},\txi)$, one can evaluate $h_{k}(\bs{x},\txi)$, find a supergradient of $h_{k}(\bs{x},\txi)$ in $\txi$, and find a subgradient of $h_{k}(\bs{x},\txi)$ in $\bs{x}$, in time polynomial in the dimension  of  $\Cs{X}$ and $\Omega$. 
\end{itemize}
As a special case where $\Omega$ is an ellipsoid,  the resulting reformulation in Theorem \ref{thm: rev.dual_DelageYe} reduces to a SDP of finite size. Motivated by the computational challenges of solving a semidefinite reformulation of \eqref{eq: DRO_Obj} formed via \eqref{eq: rev.DelageYe}, \citet{cheng2018}  propose an approximation method  to reduce the dimensionality of the resulting \dro. This approximation method relies on the principal component analysis  for the optimal lower
dimensional representation of the variability in random samples. They show that this approximation yields a relaxation of the original problem and  give theoretical bounds on the gap between the
original problem and its approximation. 

\citet{popescu2007} study a class of stochastic optimization problems, where the objective function is characterized with one- or two-point support functions. They show that when the ambiguity set of distributions is formed with all distributions with known mean and covaraince, the problem reduces to a deterministic parametric quadratic program. In particular, this result holds for increasing concave utilities with convex or concave-convex derivatives.

\citet{goh2010tractable} study a \dro\ approach to a  stochastic linear optimization problem with expectation constraints, where the support and mean of the random parameters belong to a conic-representable set, while the covariance matrix is assumed to be known.

\paragraph{Discrete Problems}

Under the assumption that the mean and covariance are known, \citet{natarajan2017SDP} investigate  the worst-case expected value of the maximum of a linear function of random variables as follows:
\begin{equation*}
    \sup_{P \in \Cs{P}} \ \ee{P}{Z(\txi)},
\end{equation*}
where $Z(\txi)=\max\sset*{\txi^{\top}\bs{x}}{\bs{x} \in \Cs{X}}$. 
The set $\Cs{X}$ is specified with either a finite number of points or a bounded feasible region to a mixed-integer LP. To obtain an upper bound, they approximate the copostive programming reformulation of the problem, presented in \citet[Theorem~3.3]{natarajan2011mixed}, with a SDP.  They  show that the complexity of computing this bound is closely related to characterizing the convex hull of the quadratic forms of the points in the feasible region. 

\citet{xie2018integer} study a \dro\ approach to a two-stage stochastic program with a simple integer round-up recourse function, defined as follows:
\begin{equation*}
    \min_{\bs{x}}\sset*{\bs{c}^{\top}\bs{x} +\max_{P \in \Cs{P}} \ee{P}{h(\bs{x}, \txi)}}{\bs{A}\bs{x} \ge \bs{b}, \; \bs{x} \in \Bs{R}^{n}},
\end{equation*}
where 
\begin{equation*}
    h(\bs{x},\bs{\xi})=\min_{\bs{u},\bs{v}}\sset*{\bs{q}^{\top}\bs{u}+ \bs{r}^{\top}\bs{v}}{\bs{u} \ge \bs{\xi}-\bs{T} \bs{x}, \; \bs{v} \ge \bs{T} \bs{x} -\bs{\xi}, \; \bs{u}, \bs{v} \in \Bs{Z}_{+}^{q}}.
\end{equation*}
The ambiguity set is formed by the product of one-dimensional ambiguity sets for each component of the random parameter $\txi$, formed with marginal distributions with known support and mean. They obtain a closed-form expression for the inner problem corresponding to each component, and they reformulate the problem as a mixed-integer SOCP. 

\citet{ahipasaoglu2016distributionally} study distributionally robust project crashing problems. They assume the underlying joint probability distribution of the activity durations lies in an ambiguity set of distributions with the given mean, standard deviation, and correlation information. The goal is to  select  the means and standard deviations to minimize the worst-case expected makespan for the project network with respect to the ambiguity set of distributions. Unlike the typical use of the SDP solvers to directly solve the problem, they  exploit the problem structure to reformulate it as a convex-concave saddle point problem over the
first two moment variables in order to solve the formulation in  polynominal time. %\alertHR{See the paper for other DRO research in this area.} 

A distributionally robust approach to an individual chance constraint with binary decisions is studied in \citet{zhang2018}.
They consider the following individual chance constraints with 
$g_{j}(\bs{x}, \txi)$, $j=1, \ldots, m$,  in \eqref{eq: DRO_Cons} is defined as 
$$g_{j}(\bs{x}, \txi):=\mathbbm{1}_{[\txi^{\top}\bs{x} \le \bs{b}]}(\txi),$$
where $\bs{x} \in \{0,1\}^{n}$. 
They form the ambiguity set of distributions by all joint distributions whose marginal means and covarinces satisfy the constraints in \eqref{eq: rev.DelageYe}. They  reformulate the chance constraints as binary  second-order conic (SOC) constraints. 

\paragraph{Risk and Chance Constraints}

Risk-based \dro\ models formed via the  ambiguity set \eqref{eq: rev.DelageYe} are also studied in the literature. 
\citet{bertsimas2010minmax} study a risk-averse distributionally robust two-stage stochastic linear optimization problem where  the mean and the covariance matrix are  known, and a convex nondecreasing piecewise linear disutility function is used to model risk. When the  second-stage objective function's coefficients   are random, they obtain a tight polynomial-sized SDP formulation. They also provide an explicit construction for a sequence of (worst-case) distributions that asymptotically attain the optimal value. They prove that this problem is NP-hard when the right-hand side is random, and  further show that under the special case that the extreme points of the dual  of the second-stage problem are explicitly known, the problem admits a SDP reformulation. An  explicit construction of the worst-case distributions is also given.  The results are applied to the production-transportation
problem and a single facility minimax distance problem. 
\citet{li2016law} obtains a closed-form expression to the worst-case of the class of law invariant coherent risk measures, where the worst case is taken with respect to all distributions with the same mean and covariance matrix.

\citet{zymler2013VaR} extend the work of \citet{ghaoui2003worst}  with known first and second order moments to a portfolio of derivatives, and develop two worst-case VaR models to capture the nonlinear dependencies between
the derivative returns and the underlying asset
returns. They introduce worst-case polyhedral VaR with convex piecewise-linear relationship between the derivative return and the asset returns. They also show that minimizing worst-case polyhedral VaR is equivalent to a convex SOCP.  A worst-case quadratic VaR with (possibly nonconvex) quadratic relationships between the derivative return and the asset returns is also introduced,  and they show that minimizing worst-case quadratic VaR is equivalent to a convex SDP. These worst-case VaR measures are equivalent to the worst-case CVaR of the underlying polyhedral or quadratic loss function, and they are coherent. As in \citet{ghaoui2003worst}, \citet{zymler2013VaR} show that optimization of these new worst-case VaR has a RO interpretation over an uncertainty set, asymmetrically oriented
around the mean values of the asset returns. 
Using the result from \citet{zymler2013Chance}, 
\citet{rujeerapaiboon2016} show that the worst-case VaR of the quadratic approximation of a  portfolio
growth rate can be expressed as the optimal value
of a SDP. 

%\citet{hanasusanto2015NV} study a distributionally robust newsvendor model with a mean-risk objective, as a convex combination of the worst-case CVaR and the worst-case expectation. The worst case is taken over all demand distributions within a {\it multimodal} ambiguity set, i.e., a mixture of a finite number of modes, with known conditional information on the ellipsoid supports, means, and covariances.  They cast the resulting model as  an exact semidefinite program, and obtain a conservative semidefinite approximation by using quadratic decision rules to approximate the recourse decisions. They further robustify their model against ambiguity in estimating mean-covariance information, caused from ambiguity about the mixture probabilities. They assume the mixture probabilities are close enough to a  nominal  probability vector in the sense of $\chi^2$-distance. For this case, they also obtain exact semidefinite reformulation as well as a conservative semidefinite approximation. 

\citet{chen2010joint} summarize and develop different
approximations to the individual
chance constraint used in the robust optimization as 
the consequence of applying different bounds on CVaR. These bounds, in turn, can be written as an optimization problem over an uncertainty set.  
For instance, they show that when the uncertainties are characterized
only by their means and covariance, the corresponding uncertainty set is an ellipsoid. 
 %\alertHR{Upon checking the paper again, this doesn't seem very relevant. Triple check.}
\citet{calafiore2006} provide
explicit results for enforcement of the individual chance
constraint over an ambiguity set of distributions. When only the information on the mean and covariance are considered, the worst-case chance constraint is equivalent to a convex second-order conic (SOC) constraint. With additional information on the symmetry, the worst-case chance constraint can be safely approximated via a convex SOC constraint. Additionally, when the means are known and individual elements are  known to belong with probability one to independent
bounded intervals, the worst-case chance constraint can be safely approximated via a convex SOC constraint. 

\citet{zymler2013Chance} study a safe approximation to distributionally robust individual and joint chance constraints based on the worst-case CVaR. Under the assumptions that the ambiguity set is formed via distributions with fixed mean and covariance, and the chance safe regions  are bi-affine in $\bs{x}$ and $\txi$, they obtain an exact SDP reformulation of the worst-case CVaR. They show that the CVaR approximation is in fact exact for individual chance constraints
whose constraint functions are either convex or (possibly nonconconvex) quadratic in $\txi$ by relying on nonlinear Farkas lemma and $\Cs{S}$-lemma, see, e.g.,  \citet{polik2007survey}. 

\citet{chen2010joint} extend their idea to the joint chance constraint by using bounds for order statistics. They show that the resulting approximation
for the joint chance constraint outperforms the Bonferroni approximation, and the constraints of the approximation  are
second-order conic-representable. \citet{zymler2013Chance}  show that the  CVaR approximation is exact for joint chance constraints whose constraint functions depend linearly on $\txi$.  They evaluate the performance of their approximation for joint chance constraint in the context of a water reservoir control problem for hydro power generation and show it outperforms the Bonferroni approximation and the method of \citet{chen2010joint}.

Motivated by the fact that chance constraints do not take into account the magnitude of the violation, \citet{xu2012optimization} study a probabilistic envelope constraint. This approach can be interpreted as a continum of chance constraints with nondecreasing target values and probabilities. They show that when the first two order moments are known, an ambigious probabilistic envelope constraint is equivalent to a deterministic SIP, which is called as a {\it comprehensive robust optimization} problem  \citep{ben2006comprehensive,ben2010soft}. In other words, ambiguous probabilistic envelope constraint  alleviates the  ``all-or-nothing" view of the standard RO that ignores realizations outside of the uncertainty set.  We refer to \citet{yang2016chance} for an extension of the work in \citet{xu2012optimization} to the nonlinear inequalities.

\paragraph{Statistical Learning}
\citet{lanckriet2002} present a \dro\ approach to a binary classification problem to minimize the worst-case probability of missclassification where the  mean and covariance matrix
of each class are known. They show that for a linear hypothesis, the problem can be formulated as a SOCP. They also investigate the case where  the mean and covariance are unknown and belong to convex uncertainty sets. They show that when the mean is unknown and belongs to an ellipsoid, the problem is a SOCP. On the other hand, when the mean is known and covariance belongs to a matrix norm ball, the problem is a SOCP and adopts a regularization term. For a nonlinear hypothesis, they seek a kernal function to map into a higher-dimensional covariates-response space such that a linear hypothesis in that space corresponds to a nonlinear hypothesis in the original covariate-response space. Using this idea, the model is reformulated as an SOCP. %such that  so that  Mercer kernels.  

\paragraph{Multistage Setting}

\citet{xin2018moment} study a multistage distributionally robust newvendor problem where the support and the first two order moments of the demand distribution are known at each stage. They provide a formal definition of the time consistency of  the optimal policies and study this  phenomena in the context of the newsvendor problem. They further relate time consistency to rectangularity of measures, see, e.g., \citet{shapiro2016rectangular}, and provide sufficient conditions for time consistency. 
Unlike \citet{xin2018moment} that suppose  the demand process is stage-wise independent, \citet{xin2018martingle} assume  that the demand process is a martingale. 
They form the ambiguity set by all distributions with a known support and mean  at each stage. They obtain the optimal policy and a two-point worst-case probability distribution, one of which is zero,  in closed forms. They also show that for any initial inventory level, the optimal policy and random demand (distributed according to the worst-case distribution) is such that for all stages, either demand is greater than or equal to the inventory  or demand is zero, meaning that all future demands are also zero. 

\citet{yang2018game} and \citet{vanparys2016constrained} study a stochastic optimal control model to minimize the worst-case probability that a system remains in a safe region for all stages. \citet{yang2018game} forms the ambiguity set at each stage by all distributions for which the  componentwise mean of random parameters is within an interval, while the covariance is in a positive semidefinite cone. \citet{vanparys2016constrained} form the ambiguity set by all distributions with a known mean and covariance.

\subsubsection{Generalized Moment and Measure Inequalities}
\label{sec: rev. measure_marginal_moments}

In this section we review an ambiguity set that allows to model the support of the random vector, and impose bounds on the probability measure  as well as functions of the random vector as follow:
\begin{equation}
\label{eq: moment-rob-set}
\Cs{P}^{MM}:=\sset*{P\in\M}{ \nu_1 \preceq P \preceq \nu_2,\; \int_{\Xi} \bs{f} d P \in [\bs{l}, \bs{u}]},
\end{equation}
where  $\nu_{1}, \nu_{2} \in\M$ are two given measures that impose lower and upper bounds on a measure $P \in \M$, and $\bs{f}:=[f_1,\ldots,f_m]$ is a vector of measurable functions on $\measurespace$, with $m \ge 1$.
The first constraint in \eqref{eq: moment-rob-set} enforces a preference relationship between probability measures. To ensure that $P$ is a probability measure, i.e., $P \in \Fs{M}\measurespace$, we set $l_1=u_1=1$ and $f_1=1$ in the above definition of $\Cs{P}^{MM}$. 
\citet{shapiro2004minmax} propose this framework, and special cases of it appear in \citet{popescu2005semidefinite}, \citet{bertsimas2005optimal}, \citet{perakis2008regret}, \citet{mehrotra2014semi}, among others. Note that if the first constraint in \eqref{eq: moment-rob-set} is disregarded (i.e., we only have $P \succeq 0$), then we can form the constraints of a classical problem of moments, see, e.g., \citet{landau1987moments}. Using this unified set, one can impose bounds on the standard moments, by setting the $i$th entry of $\bs{f}$ to have the form: $f_i(\txi):=(\xi_1)^{k_{i1}}\cdot (\xi_2)^{k_{i2}} \cdots (\xi_d)^{k_{id}}$, where $k_{ij}$ is a nonnegative integer indicating the power of $\xi_j$ for the $i$th moment function. Other possible choices for the 
functions $\bs{f}$ include the mean absolute deviation, the (co-)variances, semi-variance, higher order moments, and Huber
loss function. Moreover, proper choices of $\bs{f}$ will give the flexibility to impose structural properties on the probability distribution, see, e.g., \citet{popescu2005semidefinite} and \citet{perakis2008regret} to model the unimodality and  symmetry of distributions within this framework (see also Section \ref{sec: rev.shape}).

Below, we present a duality result $\sup_{P \in \Cs{P}^{\text{MM}}} \ \ee{P}{h(\bs{x},\txi)}$, given a fixed $\bs{x} \in \Cs{X}$. 
\begin{theorem}{(\citet[Proposition~2.1]{shapiro2004minmax})}
    \label{thm: rev.dual_generalizedmoments}
    For a fixed $\bs{x} \in \Cs{X}$, suppose that $h(\bs{x},\txi)$ is $\nu_{2}$-integrable, i.e., $\int_{\Xi} |h(\bs{x},\txi)| d \nu_{2} < \infty$, as defined in \eqref{eq: moment-rob-set}. Moreover, suppose that $\bs{f}$ is $\nu_{2}$-integrable, and there exists  $\nu_1 \preceq P \preceq \nu_2$ such that  $\int_{\Xi} \bs{f} d P \in (\bs{l}, \bs{u})$. If $\sup_{P \in \Cs{P}^{\text{MM}}} \ \ee{P}{h(\bs{x},\txi)}$  is finite, then, it can be written  as the optimal value of the following problem:
    \begin{align*}
        \inf_{\bs{r},\bs{t}} \ & \bs{r}^{\top} \bs{u} -\bs{t}^{\top} \bs{l} + \Psi(\bs{r},\bs{t}) \\
        \st \quad & \bs{r},\bs{t} \ge \bs{0},
    \end{align*}
    where 
    $$\Psi(\bs{r},\bs{t})=\int_{\Xi} \Big(h(\bs{x},s)+ (\bs{t}-\bs{r})^{\top}\bs{f}(s) \Big)_{+}  \nu_{2}(ds) - \int_{\Xi} \Big(-h(\bs{x},s)- (\bs{t}-\bs{r})^{\top}\bs{f}(s) \Big)_{+} \nu_{1}(ds).$$
\end{theorem}

%By using the theories of Lagrangian duality and conjugate duality, it can be shown that there is no duality gap between \eqref{eq: DRO_Obj} and its dual, under suitable constraint qualification conditions \citet{shapiro2004minmax}.  
\citet{shapiro2004minmax} focus on a special case of \eqref{eq: moment-rob-set}, where the first constraint is written as $(1-\epsilon)P^{*} \preceq P \preceq (1+\epsilon)P^{*}$, for some reference measure $P^{*}$, and they identify the coherent risk measure corresponding to the studied  \dro. They further study the class of problems with convex objective function $h$ and two-stage stochastic programs.  \citet{popescu2005semidefinite,bertsimas2005optimal,mehrotra2014semi}  study the classical problem of moments, i.e., ambiguity set is formed via only  the second constraints in \eqref{eq: moment-rob-set}. When   $\bs{f}$ are moment functions, \citet{mehrotra2014semi} show that under mild conditions (continuous function $h$ and compact support $\Omega$), the optimal value of a sequence of problems of the form  \eqref{eq: DRO_Obj}, where the ambiguity set is constructed via an increasing  number of moments of the underlying probability distributions, with moments matched to those  under a reference distribution, converges to the optimal value of a problem of the form  \eqref{eq: SO_Obj} under the reference distribution. 
Moreover, using the SIP reformulation of \eqref{eq: DRO_Obj}, \citet{mehrotra2014semi} propose a cutting surface method to solve a convex \eqref{eq: DRO_Obj}. This method %is not polynomial time, but it 
can be applied to problems where
bounds of moments are of arbitrary order, and possibly, bounds on nonpolynomial moments are available. 

\citet{royset2017} study a \dro\ model with a decision-dependent ambiguity set, where the ambiguity set has the form of \eqref{eq: moment-rob-set}, without the second set of constraints, and the first constraint is formed via the decision-dependent cumulative distribution functions (cdf). They establish the convergence properties of the solutions to this problem by exploiting and refining results in variational analysis. 

Besides \citet{shapiro2004minmax}, there are other studies that focus on special types of cost function $h$. Two-stage stochastic programs have received much attention in this class.
\citet{chen2018discrete} consider a  two-stage stochastic
linear complementarity problem, where the underlying random data are
continuously distributed. They study a distributionally robust approach to this problem, where the ambiguity set of distributions is formed via \eqref{eq: moment-rob-set} without the first constraint, and propose a discretization scheme to solve the problem. They investigate the asymptotic behavior of the approximated solution in the number of discrete partitions of the sample space $\Xi$. As an application, they study robust game in a duoploy market where two players need to make strategic
decisions on capacity for future production with anticipation of Nash-Cournot
type competition after demand uncertainty  is observed. 
There are studies that consider only lower order moments, up to order 2. 
\citet{ardestanijaafari2016} study distributionally robust multi-item newsvendor problem, where the ambiguity set of distribution contains all distributions with a known budgeted support, mean, and partial first order moments. To provide a reformulation of the problem, they propose  a conservative approximation scheme for maximizing the sum of piecewise linear functions over polyhedral uncertainty set based on the relaxation of an associated mixed-integer LP. They show that for the above studied newsvendor problem such an approximation is exact and  it is a linear program.

\paragraph{Discrete Problems}
\citet{bansal2018} study a (two-stage) distributionally robust integer program with pure binary first-stage and mixed-binary second stage decisions on a finite set of scenarios. They propose a  decomposition-based L-shaped algorithm and a cutting surface algorithm to solve the resulting model. They investigate the conditions and ambiguity set of distribution under which the proposed algorithm is finitely convergent. They show that ambiguity set of distributions formed via \eqref{eq: moment-rob-set} without the first constraint,  satisfy these conditions. 
\citet{hanasusanto2016} study a finite adaptability scheme to approximate the following   two-stage distributionally robust linear program, with binary recourse decisions and optimized certainty equivalent as a risk measure:
\begin{equation*}
    \min_{\bs{x}} \ \max_{P \in \Cs{P}}  \sset*{\txi^{\top} \bs{C} \bs{x} + \rrisk{P}{h(\bs{x}, \txi)}}{\bs{A}\bs{x} \ge \bs{b}, \; \bs{x} \in \{0,1\}^{q_{1}} \times \Bs{R}^{n-q_{1}}},
\end{equation*}
where 
\begin{equation*}
    h(\bs{x},\bs{\xi})=\min_{\bs{y}}\sset*{\bs{q}^{\top}\bs{Q}\bs{y}(\bs{\xi})}{\bs{W}\bs{y}(\bs{\xi}) \ge \bs{R}\bs{\xi} - \bs{T} \bs{x}, \; \bs{y}(\bs{\xi}) \in \{0,1\}^{q_{2}} },
\end{equation*}
and 
$\rrisk{P}{h(\bs{x}, \txi)}$ is an optimized certainty equivalent risk measure corresponding to the utility  function $u$:  $\rrisk{P}{h(\bs{x}, \txi)}=\inf_{\eta \in \Bs{R}} \eta + \ee{P}{u\big(h(\bs{x}, \txi)-\eta\big)}$ \citep{bental1986,bental2007OCE}. 
As an alternative to the affine recourse approximation, they pre-determine a set of finite recourse decisions here-and-now, and implement the best among them after the realization is observed. They form the ambiguity set of distributions as in \eqref{eq: moment-rob-set} but without the first constraint, where the support is assumed to be a polytope and functions $f_{i}$ are also convex piecewise linear in $\txi$. They derive an equivalent  mixed-integer LP for the resulting model. They also obtain upper and lower bounds on the probability with which any of these recourse decisions is chosen under any ambiguous distribution  as  linear programs.  
\citet{postek2018sip} study a  two-stage stochastic integer program, where the second-stage problem is a mixed-integer program. They model the distributional ambiguity by all distributions whose mean and mean-absolute deviation are known. While they show that the problem reduces to a two-stage stochastic program when there is no discrete variables, they  develop a general approximation framework for the \dro\ problem with integer variables. They apply their results to  a surgery block allocation problem.

\paragraph{Risk and Chance Constraints}

\citet{bertsimas2005optimal} study the worst-case bound on the probability of a multivariate random vector falling outside a semialgebreic confidence region  (i.e., a set described via polynomial inequalities) over an ambiguity set of the form  \eqref{eq: moment-rob-set}, where  functions $\bs{f}$ are represented by all polynomials of up to $k$th-order. For the univariate case, they obtain the result as a SDP. In particular, they obtain closed-form bounds, when $k \le 3$. For the multivariate case, they show that such a bound can be obtained via a family of SDP relaxations, yielding  a sequence of increasingly stronger, asymptotically
exact upper bounds% by solving semidefinite optimization problems a convex optimization approach that approximates the problem  by a hierarchy of increasingly accurate bounds
, each of which is calculated via a SDP. A special case of \citet{bertsimas2005optimal} appears in \citet{vandenberghe2007}, where the confidence region is described via linear and quadratic inequalities, and the first two order moments are assumed to be known within the ambiguity set. 

Building from \citet{chen2018discrete}, \citet{liu2017} study a distributionally robust reward-risk
ratio model, based on a variation of the Sharpe ratio. The ambiguity set contains all distributions whose componentwise   means and covariances are restricted to  intervals. They turn this problem into a model with a distributionally robust inequality constraint, and further reformulate this model as a nonconvex SIP. They approximate the semi-infinite constraint with  an entropic risk measure approximation\footnote{For a measurable function $Z \in \Cs{Z}_{\infty}(Q)$, the entropic risk meaure is defined as $\frac{1}{\gamma} \ln \ee{Q}{\exp{(-\gamma Z)}}$, where $\gamma>0$ \citep{liu2017}.}  and provide an iterative method to solve the resulting model. They provide statistical analysis to assess the likelihood of the true probability distribution lying in the ambiguity set, and provide a  convergence analysis of the optimal value and solutions of the data-driven distributionally robust reward-risk ratio problems. The  results are applied to  a portfolio optimization problem.

\citet{nemirovski2006convex} study a convex approximation, referred to as {\it Bernstein} approximation, to an ambiguous joint chance-constrained problem of the form 
\begin{subequations}
\label{eq: DRO_joint_chance}
\begin{align}
    \min_{\bs{x} \in \Cs{X} } \ & h(\bs{x}) \\
    \st \quad  &  \inf_{P \in \Cs{P} } \ P \left \lbrace \txi : g_{i0}(\bs{x}) + \sum_{j=1}^{d} \tilde{\xi}_{j} g_{ij}(\bs{x}) \le 0, \; i=1, \ldots, m \right \rbrace  \ge 1- \epsilon.
\end{align}
\end{subequations}
\begin{theorem}{(\citet[Theorem~6.2]{nemirovski2006convex})}
    \label{thm: rev.Bernstein}
    Suppose that the ambiguous joint chance-constrained problem \eqref{eq: DRO_joint_chance} is such that (i) the components of the random vector $\txi$ are independent of each other, with finite-valued moment generating functions, (ii)  function $h(\bs{x})$ and all functions $g_{ij}$, $i=1, \ldots, m$, $j=0, \ldots, d$, are convex and well defined on $\Cs{X}$, and (iii)  the ambiguity set of probability distributions $\Cs{P}$ forms a convex set. 
    Let $\epsilon_{i}$, $i=1, \ldots, m$, be positive real values such that $\sum_{i=1}^{m} \epsilon_{i} \le \epsilon$. Then, the problem 
    \begin{align*}
        \min_{\bs{x} \in \Cs{X} } \ & h(\bs{x}) \\
        \st \quad  &  \inf_{t >0 } \ \big [g_{i0}(\bs{x}) + t \hat{\Psi} (t^{-1}\bs{z}^{i}[\bs{x}]) - t \log \epsilon_{i} \big] \le 0, \; i=1, \ldots, m,
    \end{align*}
    where $z^{i}(\bs{x})=\big(g_{i1}(\bs{x}), \ldots, g_{id}(\bs{x})\big)$ and $$\hat{\Psi}(\bs{z}):=\max_{Q_{1} \times \ldots \times Q_{d} \in \Cs{P}} \sum_{j=1}^{d} \log \Big( \int_{\Xi} \exp\{z_{j}s\}d Q_{j}(s) \Big),$$
    is a conservative approximation of problem \eqref{eq: DRO_joint_chance}, i.e., every feasible solution to the approximation is feasible for the chance-constrained problem \eqref{eq: DRO_joint_chance}. This approximation is a convex program and is efficiently solvable, provided that all $g_{ij}$ and $\hat{\Psi}$ are efficiently computable, and $\Cs{X}$ is computationally tractable.  
\end{theorem}

\citet{hanasusanto2017ambiguous} study a distributionally robust joint chance constrained stochastic program  where each chance constraint is linear in $\txi$, and the technology matrix and right hand-side are affine in $\bs{x}$. They form the ambiguity set of distributions as in \eqref{eq: moment-rob-set} without the first constraint. They show that the pessimistic model (i.e., the chance constraint holds for every distribution in the set) is conic-representable if the technology matrix is constant in $\bs{x}$, the support set is a cone, and $f_{i}$ is positively homogeneous. They also show the optimistic model (i.e., the chance constraint holds for at least one distribution in the set) is also conic-representable if the technology matrix is constant in $\bs{x}$. 
They apply  their results to problems in project management and image reconstruction. 
While their formulation is exact for the distributionally
robust chance constrained project crashing problem, the size of the formulation grows in the number of paths in the network. 
For other research in chance-constrained optimization problem, we refer to \citet{xie2017optimized,xie2018joint}.

\paragraph{Statistical Learning}

\citet{fathony2018} study a distributonally robust approach to graphical models for leveraging the graphical structure among the  variables. 
The proposed model in \citet{fathony2018} seeks a predictor to make a probabilistic prediction $\hat{P}(\hat{y}|\bs{u})$ over all possible label assignments so that it minimizes the worst-case conditional expectation of the prediction loss $l(\hat{y},\bar{y})$ with respect to $\bar{P}(\bar{y}|\bs{u})$ as follows:
\begin{align*}
	\min_{\hat{P}(\hat{y}|\bs{u})} \ \max_{\bar{P}(\bar{y}|\bs{u})} \ & \ee{\substack{\bs{U} \sim \breve{P} \\ \hat{Y}|
			\bs{U} \sim \hat{P}\\ \bar{Y}|
			\bs{U} \sim \bar{P}}}{l(\hat{Y},\bar{Y})} \\
	\st \quad & \ee{\substack{\bs{U} \sim \breve{P} \\ \bar{Y}|
			\bs{U} \sim \bar{P}}}{\Phi(\bs{U},Y)}   =\breve{\Phi}, 
\end{align*}
where $\Phi(\bs{U}, Y)$ is a given feature function and $\breve{\Phi}=\ee{(\bs{U},Y) \sim \breve{P}}{\Phi(\bs{U},Y)}$.  The  worst-case in the above formulation is taken with respect to all conditional distributions of the predictor, conditioned on the covariates. This conditional distribution $\bar{P}(\bar{y}|\bs{u})$ is such that the first-order moment of the  feature function $\Phi(\bs{U}, Y)$ matches the first-order moment under the empirical  joint distribution of the covariates and labels, $\breve{P}$. 
\citet{fathony2018} show that the \dro\  approach enjoys the consistency guarantees of probabilistic graphical models, see, e.g., \citet{lafferty2001}, and has the advantage of  incorporating
customized loss metrics during the training as in large margin models, see, e.g., \citet{tsochantaridis2005}. 

\subsubsection{Moment Matrix Inequalities}

In this section we review an ambiguity set that generalizes both the ambiguity set $\Cs{P}^{\text{DY}}$ \eqref{eq: rev.DelageYe} and the ambiguity set $\Cs{P}^{\text{MM}}$ \eqref{eq: moment-rob-set} as follows:
\begin{equation}
\label{eq: rev.MMI}
\Cs{P}^{MMI}:=\sset*{P\in\M}{ \bs{L} \preccurlyeq \int_{\Xi} \bs{F} d P \preccurlyeq \bs{U}},
\end{equation}
where $\bs{F}:=[\bs{F}_1,\ldots,\bs{F}_m]$, with $\bs{F}_{i}$ be a symmetric matrix in $\Bs{R}^{n_{i} \times n_{i}}$ or scalar with measurable components on $\measurespace$. Similarly, let $\bs{L}:=[\bs{L}_1,\ldots,\bs{L}_m]$ and $\bs{U}:=[\bs{U}_1,\ldots,\bs{U}_m]$ be the vectors of symmetric matrices or scalars. 
As in \eqref{eq: moment-rob-set}, to ensure that $P$ is a probability measure, i.e., $P \in \Fs{M}\measurespace$, we set $\bs{L}_1=\bs{U}_{1}= [1]_{1 \times 1}$ and $\bs{F}_1=[1]_{1 \times 1}$ in the above definition of $\Cs{P}^{MMI}$. 
We generalize this ambiguity set from the ambiguity set proposed in \citet{xu2018matrix}, where the moment constraint are either in the form of equality or upper bound.  Note that as a special case of $\Cs{P}^{MMI}$, we can set  $\bs{F}_{i}$, $\bs{L}_{i}$, and $\bs{U}_{i}$ to be scalars, $i=2, \ldots, m$, to recover the second constraint in the ambiguity set $\Cs{P}^{MM}$, defined in \eqref{eq: moment-rob-set}. 
Moreover, by setting $\bs{F}_{2}$ to be a matrix as $\begin{pmatrix}
    -\bs{\Sigma}_{0}       & \bs{\mu}_{0} -\txi \\
    (\bs{\mu}_{0} -\txi)^{\top} & -\varrho_{1}
\end{pmatrix}$, $\bs{F}_{3}$ to be a matrix as $(\txi-\bs{\mu}_{0})(\xi-\bs{\mu}_{0})^{\top}$, $\bs{L}_{2}=-\bs{\infty}$, $\bs{U}_{2}=\bs{L}_{3}= \bs{0}$, and $\bs{U}_{3}=\varrho_{2} \bs{\Sigma}_{0}$, we can recover \eqref{eq: rev.DelageYe}. 

Below, we present a duality result on $\sup_{P \in \Cs{P}^{\text{MMI}}} \ \ee{P}{h(\bs{x},\txi)}$, given a fixed $\bs{x} \in \Cs{X}$. 
\begin{theorem}
    \label{thm: rev.dual_MMI}
    For a fixed $\bs{x} \in \Cs{X}$, suppose that $h(\bs{x},\txi)$ and $\bs{F}$ are integrable for all $P \in \Cs{P}^{\text{MMI}}$. 
    %Moreover, suppose that $h(\bs{x},\txi)$ is upper semicontinuous, and $F$ are bounded and continuous. 
    In addition, suppose that the following Slater-type condition holds:
    $$(-\bs{U}, \bs{L}) \in \inte{\left \lbrace \Big(-\int_{\Xi} \bs{F} d P, \int_{\Xi} \bs{F} d P \Big) - \Cs{K} \; \Big| \; P \in \M\right \rbrace}, $$
    where $\Cs{K}:=\Cs{S}_{+}^{n_{1}} \times \ldots \Cs{S}_{+}^{n_{m}} \times \Cs{S}_{+}^{n_{1}} \times \ldots \Cs{S}_{+}^{n_{m}}$. 
    If $\sup_{P \in \Cs{P}^{\text{MM}}} \ \ee{P}{h(\bs{x},\txi)}$ is finite, then, it can be written  as the optimal value of the following problem:
    \begin{align*}
        \inf_{\bs{W},\bs{Y}} \ & \sum_{i=1}^{m}  \bs{W}_{i} \bullet  \bs{U}_{i}  -\sum_{i=1}^{m} \bs{Y}_{i} \bullet \bs{L}_{i}  \\
        \begin{split}
            \st \quad & \sum_{i=1}^{m}  \bs{W}_{i} \bullet \int_{\Xi} \bs{F}_{i}(s) P(ds)  - \sum_{i=1}^{m}  \bs{Y}_{i} \bullet \int_{\Xi} \bs{F}_{i}(s) P(ds)   \\ 
            & \quad {} \ge \int_{\Xi} h(\bs{x},\txi(s)) P(ds), \; \forall P \in \M, 
        \end{split}\\
        & \bs{W},\bs{Y} \succcurlyeq \bs{0}.
    \end{align*}
\end{theorem}
\begin{proof}
Using the conic duality results from Theorem \ref{thm: rev.conic_duality}, we write the dual of \linebreak $\sup_{P \in \Cs{P}^{\text{MM}}} \ \ee{P}{h(\bs{x},\txi)}$ as 
    \begin{align*}
        \inf_{\bs{W},\bs{Y}} \ & \sum_{i=1}^{m}  \bs{W}_{i} \bullet  \bs{U}_{i}  -\sum_{i=1}^{m} \bs{Y}_{i} \bullet \bs{L}_{i}  \\
        \st \quad & \sum_{i=1}^{m}  \bs{W}_{i} \bullet  \bs{F}_{i}  - \sum_{i=1}^{m}  \bs{Y}_{i} \bullet \bs{F}_{i} \succcurlyeq_{\Fs{M}_{+}^{\prime}\measurespace}  h(\bs{x},\cdot), \\
        & \bs{W},\bs{Y} \succcurlyeq \bs{0},
    \end{align*}
where $\Fs{M}_{+}^{\prime}\measurespace$ is the dual cone of $\Fs{M}_{+}\measurespace$: $$\Fs{M}_{+}^{\prime}\measurespace=\sset*{Z \in \Cs{S}\measurespace}{ \int_{\Xi} Z(s) P(ds) \ge 0, \; \forall P \in \M}.$$ Thus,
we can write the first constraint above as 
\begin{equation*}
    \begin{split}
        & \sum_{i=1}^{m}  \bs{W}_{i} \bullet \int_{\Xi} \bs{F}_{i}(s) P(ds)  - \sum_{i=1}^{m}  \bs{Y}_{i} \bullet \int_{\Xi} \bs{F}_{i}(s) P(ds)  \\
        &  \quad {} \ge \int_{\Xi} h(\bs{x},\txi(s)) P(ds), \; \forall P \in \M.
    \end{split}    
\end{equation*}
 
The Slater-type condition ensures that the strong duality holds  \cite{shapiro2001duality}. 
\end{proof}
Suppose that   every finite subset of $\Xi$ is $\Cs{F}$-measurable, i.e., for every $s \in \Xi$, the corresponding Dirac measure $\delta(s)$ (of mass one at point $s$) belongs to $\M$. Then, the first constraint in Theorem \ref{thm: rev.dual_MMI} can be written as follows \cite{shapiro2001duality}: $$\sum_{i=1}^{m}    \bs{W}_{i}^{*}  \bullet \bs{F}_{i}(s) - \sum_{i=1}^{m}  \bs{Y}_{i}^{*}  \bullet \bs{F}_{i}(s)  \ge  h(\bs{x}, \txi(s)), \quad \forall s \in \Xi.$$

Motivated by the difficulty in verifying the Slater-type conditions to  guarantee  strong duality for $\sup_{P \in \Cs{P}^{\text{MMI}}} \ \ee{P}{h(\bs{x},\txi)}$ and its dual, \citet{xu2018matrix} investigate the duality conditions
from the perspective of  lower semicontinuity of the optimal value function inner maximization problem, with  a perturbed ambiguty set. While these conditions are restrictive in general, they show that they are satisfied in the case of compact $\Xi$ or bounded $\bs{F}_{i}$. % and derive easy-to-verify sufficient conditions. 
\citet{xu2018matrix} present two discretization schemes to solve the resulting \dro\ model: (1) a cutting-plane-based exchange method that discretizes the ambiguity set  $\Cs{P}^{\text{MMI}}$ and (2) a cutting-plane-based dual method that discretizes the semi-infinite constraint of the dual problem. For both methods, they show the convergence of the optimal values and optimal solutions as sample size increases. They illustrate their results for  the portfolio optimization and multiproduct newsvendor problems.

\subsubsection{Cross-Moment or Nested Moment}
\label{sec: rev.Wiesemann}

In an attempt to unify modeling and solving \dro\ models, \citet{wiesemann2014} propose a  framework for modeling the ambiguity set of probability distributions as follows:
\begin{equation}
	\label{eq: rev.WKS}
	\Cs{P}^{\text{WKS}}:=\sset*{\Ts{P}\in\pprobset{d}{r}}{ 
		\begin{aligned}  
			& \ee{\Ts{P}}{\bs{A}\txi +\bs{B} \tbs{u}}=\bs{b},\\
			& P\{(\txi,\tbs{u}) \in \Cs{C}_{i} \} \in [\ul{p}_{i}, \ol{p}_{i}], \; i \in \Cs{I}
		\end{aligned}  
	},
\end{equation} 
%where $\pprobset{d}{r}$ denotes the set of all probability measures on the measurable space $\proprobspace{d}{r}$, with $\Fs{B}(\cdot)$ denotes the Borel $\sigma$-field on the respective space, 
where $\Ts{P}$ represents a joint probability distribution of $\txi$ and some auxiliary random vector $\tbs{u} \in \Bs{R}^{r}$. Moreover, $\bs{A} \in \Bs{R}^{s \times d}$, $\bs{B} \in \Bs{R}^{s \times r}$, $\bs{b} \in \Bs{R}^{s}$, and $\Cs{I}=\{1, \ldots, I\}$, while the confidence sets $\Cs{C}_{i}$ are defined as 
\begin{equation}
\label{eq: rev.WKS_Cone}
\Cs{C}_{i}:=\sset*{(\bs{\xi},\bs{u}) \in \Bs{R}^{d} \times \Bs{R}^{r} }{\bs{C}_{i} \bs{\bs{\xi}} + \bs{D}_{i} \bs{u} \preccurlyeq_{\Cs{K}_{i}} \bs{c}_{i}},
\end{equation}
with  $\bs{C}_{i} \in \Bs{R}^{L_{i} \times d}$, $\bs{D}_{i} \in \Bs{R}^{L_{i} \times r}$, $\bs{c} \in \Bs{R}^{L_{i}}$, and $\Cs{K}_{i}$ being a proper cone. 
By setting $\ul{p}_{I}=\ol{p}_{I}=1$, they ensure that $\Cs{C}_{I}$ contains the support of  the joint random vector $(\txi,\tbs{u})$. 
This set contains all distributions with prescribed conic-representable
confidence sets and with mean values residing on
an affine manifold. 
An important aspect of \eqref{eq: rev.WKS} is that the  inclusion of an auxiliary random
vector $\tbs{u}$ gives the flexibility to model a rich variety of structural
information about the marginal distribution of $\txi$ in a unified
manner. 
%In fact, this canonical form is restrictive enough to give rise to tractable optimization problems, but which at the same time is expressive enough to encompass a large variety of  ambiguity sets from the  literature as special cases. 
Using this  framework, \citet{wiesemann2014} show that many ambiguity sets studied in the literature can be represented by a projection of the  ambiguity set   \eqref{eq: rev.WKS} on the space of $\txi$. In other words, these ambiguity sets are special cases of the ambiguity set $\Cs{P}^{\text{WKS}}$. %/ they show that these ambiguity sets are formed via the union of all marginal distributions of $\txi$ under all $\Ts{P} \in \Cs{P}^{\text{WKS}}$. 
This development is based on the following lifting result. 
\begin{theorem}{(\citet[Theorem~5]{wiesemann2014})}
    \label{thm: rev.lifting_WKS}
    Let $\bs{f} \in \Bs{R}^{N}$ and $\bs{l}: \Bs{R}^{d} \mapsto \Bs{R}^{N}$ be a function with a conic-representable $\Cs{K}$-epigraph, and consider the following ambiguity set: 
    \begin{equation*}
	\Cs{P}^{\prime}:=\sset*{\Ts{P}\in\P}{ 
		\begin{aligned}  
			& \ee{\Ts{P}}{\bs{l}(\txi)} \preccurlyeq_{\Cs{K}}\bs{f},\\
			& \Ts{P}\{\txi  \in \Cs{C}_{i} \} \in [\ul{p}_{i}, \ol{p}_{i}], \; i \in \Cs{I}
		\end{aligned}  
	},
    \end{equation*} 
    as well as the lifted ambiguity set 
    \begin{equation*}
    	\Cs{P}:=\sset*{\Ts{P}\in \pprobset{d}{N} }{ 
    		\begin{aligned}  
    		    & \ee{\Ts{P}}{\tbs{u}}=\bs{f},\\
    			& P\{\bs{l}(\txi) \preccurlyeq_{\Cs{K}} \tbs{u} \} = 1,\\
    			& P\{\txi  \in \Cs{C}_{i} \} \in [\ul{p}_{i}, \ol{p}_{i}], \; i \in \Cs{I}
    		\end{aligned}  
    	},
    \end{equation*} 
    which involves the auxiliary random vector $\tbs{u} \in \Bs{R}^{N}$. We have that (i) $\Cs{P}^{\prime}$ is the union of  all marginal distributions of $\txi$ under all $\Ts{P} \in \Cs{P}$ and (ii) $\Cs{P}$ can be formulated as an instance of the ambiguity set $\Cs{P}^{\text{WKS}}$ in \eqref{eq: rev.WKS}. 
\end{theorem}

Using Theorem \ref{thm: rev.lifting_WKS}, \citet{wiesemann2014} show how an ambiguity set of the form $\Cs{P}^{\text{WKS}}$, defined in \eqref{eq: rev.WKS}, with conic-representable expectation constraints and a collection of conic-representable confidence sets, can represent ambiguity sets formed via (1) $\phi$-divergences, (2) mean, (3) mean and upper bound on the covariance matrix (i.e., a special case of the ambiguity set \eqref{eq: rev.DelageYe}), (4) coefficient of variation (i.e., the inverse of signal-to-noise ratio from information theory), (5) absolute mean spread, and (6) higher-order moment information.  Moreover, they illustrate that \eqref{eq: rev.WKS} can capture information from robust statistics, such as (7) marginal median, (8) marginal median-absolute deviation, and (9) known upper bound on the expected Huber loss function. 
It is worth noting that \eqref{eq: rev.WKS} does not cover ambiguity
sets that impose infinitely many moment restrictions that
would be required to describe symmetry, independence, or
unimodality characteristics of the distributions \citep{chen2018infinite}.

\citet{wiesemann2014} determine conditions under which distributionally robust  expectation constraints, formed via the proposed ambiguity
set \eqref{eq: rev.WKS}, can be solved in polynomial time %give rise to a conic optimization problem. In
as follows: (i) the cost function $g_{j}$, $j=1, \ldots, m$, is convex and piecewise affine in  $\bs{x}$
and $\bs{\txi}$ (i.e., $g_{j}(\bs{x},\txi):=\max_{k \in \{1,\ldots, K\}} g_{jk}(\bs{x},\txi)$ with $g_{jk}(\bs{x},\txi):=s_{jk}(\txi)\bs{x}+t_{jk}(\txi)$ such that $s_{jk}(\txi)$ and $t_{jk}(\txi)$ are affine in $\txi$) and (ii) the confidence sets $\Cs{C}_{i}$'s satisfy a strict nesting
condition. Below, we present a duality result under above assumptions and additional regularity conditions. 
\begin{theorem}{(\citet[Theorem~1]{wiesemann2014})}
    \label{thm: rev.dual_WKS}
    Consider  a fixed $\bs{x} \in \Cs{X}$. Then, under suitable regularity conditions, $\sup_{\Ts{P} \in \Cs{P}^{\text{WKS}}} \ \ee{\Ts{P}}{g_{j}(\bs{x},\txi)} \le 0$, $j=1, \ldots, m$, is satisfied if and only if there exists $\bs{\beta} \in \Bs{R}^{K}$, $\bs{\kappa}, \bs{\lambda} \in \Bs{R}_{+}^{I}$, and $\bs{\alpha}_{ik} \in \dual{K}_{i}$, $i \in \Cs{I}$ and $k \in \{1, \ldots, K\}$, that satisfy the following systems:
    \begin{align*}
       & \bs{b}^{\top}\bs{\beta} + \sum_{i \in \Cs{I}} (\ol{p}_{i} \kappa_{i}- \ul{p}_{i} \lambda_{i}) \le 0, \\
       & \bs{c}_{i}^{\top}\bs{\alpha}_{ik}+ \bs{s}_{k}^{\top}\bs{x} + \bs{t}_{k} \le \sum_{i^{\prime} \in \{i\} \cup \Cs{A}(i)} (\kappa_{i^{\prime}}-\lambda_{i^{\prime}}), \quad \forall i \in \Cs{I}, \ k \in \{1, \ldots, K\},\\
       & \bs{C}_{i}^{\top} \bs{\alpha}_{ik} + \bs{A}^{\top}\bs{\beta}= \bs{S}^{\top}_{k} \bs{x} + \bs{t}_{k}, \quad \forall i \in \Cs{I}, \ k \in \{1, \ldots, K\},\\
       &\bs{D}_{i}^{\top} \bs{\alpha}_{ik} + \bs{B}^{\top}\bs{\beta}=0, \quad \forall i \in \Cs{I}, \ k \in \{1, \ldots, K\},
    \end{align*}
    where $\Cs{A}(i)$ denote the set of all $i^{\prime} \in \Cs{I}$ such that $\Cs{C}_{i^{\prime}}$ is strictly contained in the interior of  $\Cs{C}_{i}$. 
\end{theorem}

The tractability of the resulting system in Theorem \ref{thm: rev.dual_WKS} depends on how the confidence sets $\Cs{C}_{i}$ are described, and hence, they give rise to linear, conic-quadratic, or semidefinite programs for the corresponding confidence sets $\Cs{C}_{i}$.  \citet{wiesemann2014} also provide tight tractable conservative approximations for problems that violate the nesting condition by proposing an outer approximation of \eqref{eq: rev.WKS}.
They  discuss several mild modifications of the conditions on $\bs{g}$.

%\citet{wiesemann2014}  present the first conicquadratic representation of ambiguity sets incorporating higher-order moment information. 

There are several papers that use the ambiguity set \eqref{eq: rev.WKS} and consider its generalization or special cases. 
\citet{chen2018infinite}  introduce an ambiguity set of probability distributions that is characterized by  conic-representable expectation constraints and a conic-represetable support set, similar to the one studied in \citet{wiesemann2014}. However, unlike \citet{wiesemann2014}, an infinite number of expectation constraints can be incorporated into the ambiguity set to describe  stochastic dominance, entropic dominance, and dispersion, among other. A main result in this work is that  for 
any ambiguity set, there exists an  infinitely constrained ambiguity set, such that worst-case expected $h(\bs{x}, \txi)$ over both sets are equal,  provided that the objective
function $h(\bs{x}, \txi)$ is  tractable and conic-representable in $\txi$ for any $\bs{x} \in \Cs{X}$. 
Reformulation of the resulting \dro\ model formed via this infinitely constrained ambiguity set yields a conic optimization problem. 
To solve the model, \citet{chen2018infinite} propose a procedure that consists of solving a sequence of  relaxed \dro\ problems---each of which considers a finitely constrained ambiguity set, and  results in a conic optimization reformulation---and converges to the optimal value of the original \dro\ model. When incorporating covariance and fourth-order moment information into the ambiguity set, they
show that the relaxed \dro\ is a SOCP. This is different from \citet{delage2010} which  shows that a \dro\ problem formed via a fixed mean and an upper bound on covariance  is  reformulated as a SDP.

\citet{postek2018} derive exact reformulation of the worst-case expected constraints when function   $g(\bs{x}, \cdot)$ is convex in $\txi$, and the ambiguity set of distributions consists of all distributions of componentwise independent $\txi$ with known support, mean, and mean-asboulute deviation information. They also obtain exact reformulation of the resulting model when     $g(\bs{x}, \cdot)$ is concave in $\txi$ and there is additional information on the probability that a component is greater than or  equal to its mean. These reformulations involve a number of terms that are exponential in the dimension of $\txi$. They
show how upper bounds can be constructed that alleviate the independence restriction,
and require only a linear number of terms, by exploiting models in which random variables are linearly aggregated and function $g(\bs{x}, \cdot)$ is convex. Under the assumption of independent random variables, they use the above results for the worst-case expected constraints to derive safe approximations to the corresponding individual chance constrained problems.  

To reduce the conservatism of the robust optimization due to its constraint-wise approach and the  assumption that all constraints are hard for all scenarios in the uncertainty set, \citet{roos2018reducing} propose an approach that bounds worst-case expected total violation of constraints  from above and condense all constraints into a single constraint. They form the ambiguity set with all distributions of $\txi$ with known support, mean, and mean-asboulute deviation information. When the right-hand side is uncertain, they use the results in \citet{postek2018} to show that the proposed formulation is tractable. When the left-hand side is uncertain, they use the aggregation approach introduced in \citet{postek2018} to derive tractable reformulations. 
We also refer to \citet{sun2018} for a two-stage quadratic stochastic optimization problem and 
\citet{demiguel2009portfolio} for a portfolio optimization problem. 

%\paragraph{Linear Decision Rule}

\citet{bertsimas2018adaptiveDRO} develop a modular and tractable framework for solving an adaptive distributionally robust two-stage linear optimization problem with recourse of the form
\begin{equation*}
    \min_{\bs{x}}\sset*{\bs{c}^{\top}\bs{x} +\sup_{P \in \Cs{P}} \ee{P}{h(\bs{x}, \txi)}}{ \bs{x} \in \Cs{X}},
\end{equation*}
where 
\begin{equation*}
    h(\bs{x},\bs{\xi})=\min_{\bs{y}}\sset*{\bs{q}^{\top}\bs{y}(\bs{\xi})}{\bs{W}\bs{y}(\bs{\xi}) \ge \bs{r}(\bs{\xi}) - \bs{T}(\bs{\xi}) \bs{x}, \; \bs{y}(\bs{\xi}) \in \Bs{R}^{q}},
\end{equation*}
and the function $\bs{r}(\bs{\xi})$ and $\bs{T}(\bs{\xi})$ are affinely dependent on $\bs{\xi}$.
%a modular framework to obtain exact and approximate solutions to a class of distributionally robust linear optimization
%problems with recourse.  
Both the ambiguity set of probability distributions $\Cs{P}$ and the support set are assumed to be second-order conic-representable. Such an ambiguity set is a special case of   the conic-representbale ambiguity set \eqref{eq: rev.WKS}.  
They show that the studied \dro\ model can be formulated as a classical RO problem with a second-order conic-representable uncertainty set.
%is specified by linear and convex conic quadratic representable expectation constraints and the support set is also linear and conic quadratic representable. 
To obtain a tractable formulation, they  replace the recourse  decision functions $\bs{y}(\bs{\xi})$ with generalized linear decision rules that have affine dependency on the uncertain parameters $\bs{\xi}$ and some auxiliary random variables\footnote{Restricting the recourse decision function $\bs{y}(\bs{\xi})$ to the class of functions that are affinely-dependent on $\bs{\xi}$, referred to as {\it linear decision rules}, is an approach to derive computationally tractable problems to approximate stochastic programming and robust optimization models \cite{chen2007robust,chen2008linear,ben2004adjustable}. Whether or not the linear decision rules are optimal depends on the problem \cite{shapiro2005complexity}.}. By adopting the approach of  \citet{wiesemann2014} 
to lift the  ambiguity set to an extended one by introducing additional auxiliary random variables, they improve the quality of solutions and show that one can transform the adaptive \dro\ problem to a classical RO problem with a second-order conic-representable uncertainty
set. %they can obtain good and sometimes tight approximations to a two-stage optimization problem.
\citet{bertsimas2018adaptiveDRO} discuss extension to the conic-representbale ambiguity set \eqref{eq: rev.WKS} and multistage problems. 
They also apply their results to medical appointment scheduling and single-item multiperiod newsvendor problems. 

Following the approach in \citet{bertsimas2018adaptiveDRO}, \citet{zhen2018} reformulate  an adaptive distributionally robust two-stage linear optimization problem with recourse into an adaptive robust two-stage optimization problem with recourse.  Then, using \linebreak Fourier-Motzkin elimination, they reformulate this problem into an equivalent problem with a reduced number of adjustable variables at the expense of an increased number of constraints. Although from a theoretical perspective,
every adaptive robust two-stage optimization problem with recourse  admits an equivalent static
reformulation, they propose to eliminate some of the adjustable variables,  and for the   remaining
adjustable variables, they impose linear decision rules to obtain an approximated solution.
They show that for problems with  simplex uncertainty sets, linear decision rules are optimal, and for problems  with box uncertainty sets,
 there exists convex two-piecewise affine functions
that are optimal for the adjustable variables. By studying the  medical appointment scheduling considered in \citet{bertsimas2018adaptiveDRO}, they show that their approach improves the solutions obtained in \citet{bertsimas2018adaptiveDRO}.

\paragraph{Statistical Learning}
\citet{gong2018} study a distributionally robust multiple linear regression model with the least absolute value cost function. They form the ambiguity set of distributions using expectation constraints over a conic-representable support set as in \eqref{eq: rev.WKS}. They reformulate the resulting model as a conic optimization problem, based on the results in \citet{wiesemann2014}.

\paragraph{Multistage Setting.} 
A Markov decision process with unknown distribution for the transition probabilities and rewards for each state is studied in \citet{xu2012MDP,xu2010MDP}. It is assumed that the parameters are statewise independent and each state belongs to only one stage. Moreover, the parameters of each state are constrained to a 
sequence of nested sets, such that the parameters  belong to the largest set with probability one, and there is a lower bound on the probability that they should belong to other sets, in a increasing manner. 
\citet{yu2016dMDP} extends the work in \citet{xu2012MDP,xu2010MDP} by forming the ambiguity set of distributions as in \eqref{eq: rev.WKS}.

%\subsubsection{\citet{hanasusanto2015chance}}
%\alertHR{This can go to the shape preserving models.}

%\subsubsection{Second-Order Conic  Representable}

%They also introduce the {\it entropic dominance} ambiguity set, which is  an  infinitely constrained ambiguity set that incorporates an upper bound on the uncertainty's moment-generating function. %; so that it yields an improved characterization of stochastic independence over existing approaches based solely on covariance information.

\subsubsection{Marginals (Fr\'{e}chet)}

All the moment-based ambiguity sets discussed so far, study the ambiguity of the joint probability distribution of the random vector $\txi$. Papers reviewed in this section assume that additional information on the marginal distributions is available. We refer to the class of joint distributions with fixed marginal distributions
as the {\it Fr\'{e}chet} class of distributions \cite{doan2015robustness}. 

\paragraph{Discrete problems}
\citet{chen2018} study a problem of the form  \eqref{eq: DRO_Obj}, where the cost function $h(\bs{x},\txi)$ denotes the optimal value of a linear or discrete optimization problem with random linear objective coefficients. They assume the ambiguity set of distribution is formed by all distributions with known marginals. Using techniques from optimal transport theory, they identify a set of sufficient
conditions for the polynomial time solvability of this class of problems. This generalizes the  tractability results under marginal information from 0-1 polytopes, studied in \citet{bertsimas2004probabilistic},  to a class of integral polytopes. 
They discuss their results on  four  polynomial time
solvable instances, arising in the appointment scheduling problem, max flow problem with random arc capacities, ranking problem with random utilities, and project scheduling problems with irregular random starting time costs.

\paragraph{Risk and Chance Constraints}

\citet{dhara2017} provide bounds on the worst-case CVaR over an ambiguity set of discrete distributions, where the ambiguity set % formed via known univariate marginals and also, by considering   a neighborhood of the bivariate marginals in terms of the Kullback-Leibler divergence measure. That is, they form the ambiguity set of 
contains all joint distributions whose univariate marginals are fixed and their bivariate marginals are within a minimum Kullback-Leibler distance from the nominal bivariate marginals. They  develop a convex reformulation for the resulting \dro. 
\citet{doan2015robustness} study a \dro\ model of the form  \eqref{eq: DRO_Obj} with a convex piecewise linear objective function in $\txi$ and affine in $\bs{x}$. They form the ambiguity set of joint  distributions via a Fr\'{e}chet class of discrete distributions with  multivariate marginals, where the components of the random vector are partitioned such that they have overlaps. They show that the resulting \dro\ model for a portfolio optimization problem  is efficiently solvable with linear programming. In particular, they develop a tight linear programming reformulation to find a bound on the worst-case CVaR over such an ambiguity set, provided that the structure of the marginals satisfy a regularity condition.  %They apply their results to a portfolio optimization problem that minimizes the worst-case CVaR. 

\citet{natarajan2014} study a distributionally robust approach to minimize the worst-case CVaR of regret in combinatorial optimization problems with uncertainty in the objective function coefficients, defined as follows:
\begin{equation*}
    \min_{\bs{x} \in \Cs{X}} \  \mathrm{WCVaR}_{\alpha}^{P}\left[h(\bs{x}, \txi)\right],
\end{equation*}
where $h(\bs{x}, \txi)=- \txi^{\top} \bs{x}+ \max_{\bs{y} \in \{0,1\}^{q_{1}}} \txi^{\top} \bs{y} $ and 
$$\mathrm{WCVaR}_{\alpha}^{P}\left[h(\bs{x}, \txi)\right]=\sup_{P \in \Cs{P}} \cccvar{P}{\alpha}{h(\bs{x}, \txi)}.$$ 
It is assumed that the ambiguity set is formed with the knowledge of marginal distributions, where the ambiguity for each marginal distribution is formed via \eqref{eq: moment-rob-set}. They reformulate the resulting problem as a polynomial sized mixed-integer LP when (i) the support is known, (ii) the support and mean are known, and (iii) the support, mean, and mean absolute deviation are
known; and as  a mixed-integer SOCP  when the support, mean, and standard deviation are known. They show the maximum weight subset selection problem is polynomially solvable under (i) and (ii). 
They illustrate their results on subset selection    and the shortest path problems. 

\citet{zhang2015bin} study a distributionally robust approach to a stochastic bin-packing problem subject to chance constraints on the total item sizes in the bins. They form the ambiguity set  by all discrete distributions with known marginal means and variances for each item size. By showing that there exists a worst-case distribution that is at most a three-point distribution, they obtain a closed-form expression for the chance constraint and they reformulate the problem as a mixed-binary program. They present 
a branch-and-price algorithm to solve the problem, and apply their results to a surgery scheduling problem for operating rooms.

%They reformulate the resulting problem as a polynomial sized mixed-integer linear program. %They provide their results on the maximum weight subset selection problem  and the shortest path problem. 

\paragraph{Statistical Learning}

\citet{farnia2016} study a \dro\ approach in the context of supervised learning problems to infer a function (i.e., decision rule) that predicts a response variable given a set of covariates. Motivated by the game-theoretic interpretation of \citet{grunwald2004game} and the principle of maximum entropy, they seek a   decision rule that predicts the response based on a distribution that maximizes a generalized entropy function over a set of probability distributions. However, because the covariate information is available,  they apply the principle of maximum entropy to the conditional distribution of the response given the covariates, see, also \citet{globerson2004} for the case of Shannon entropy.  \citet{farnia2016} form the ambiguity set of distributions by  matching the
marginal of covariates to the empirical marginal of covariates while keeping the cross-moments between the response variables and covariates close enough (with respect to some norm) to that of  the joint empirical distribution. They show that the \dro\ approach adopts a regularization interpretation for the maximum likelihood problem under the empirical distribution. As a result, \citet{farnia2016} recover the regularized maximum likelihood problem for generalized linear models for the following loss functions: linear regression under quadratic loss function, logistic regression under logarithmic loss function, and SVM under the 0-1 loss function. 

\citet{eban2014} study a \dro\ approach to a classification problem to minimize the worst-case hinge loss of  missclassification, where the ambiguity set of the joint probability distributions  of the discrete covariates and response should contain all distributions that agree with nominal pair-wise marginals. They show that the proposed classifier provides a 2-approximation upper bound on the worst-case expected  loss using a zero-one hinge loss.
\citet{razaviyayn2015} study a \dro\ approach to the binary classification problem, with an ambiguity set  similar to that of \citet{eban2014}, to minimize the worst-case missclassification probability. By changing the order of $\inf$ and $\sup$, and smoothing the objective function, they obtain a probability distribution, based on which they propose a randomized classifier. They show that this randomized classifier  enjoys a 2-approximation upper bound on the worst-case missclassification probability of the optimal solution to the studied \dro.

\subsubsection{Mixture Distribution}

In this section, we study \dro\ models, where the ambiguity set is formed via {\it mixture distribution}. A mixture distribution is defined as a convex combination of pdfs, known as the {\it mixture components}. The weights associated with the mixture components are called {\it mixture probabilities} \cite{kapsos2014}. For example, a mixture model can be defined as the set of all mixtures of normal distributions with mean $\mu$ and standard deviation $\sigma$ with parameter $\bs{a}=(\mu, \sigma)$ in some compact set $\Cs{A} \subset \Bs{R}^{2}$. In a more generic framework, the distribution $P$ can be any mixture of probability distributions $Q_{\bs{a}} \in \Fs{M}\measurespace$, for some family of distributions $\{Q_{\bs{a}}\}_{\bs{a} \in \Cs{A}} \in \Fs{M}\measurespace$, that depends on the  parameter vector $\bs{a} \in \Cs{A}$ as follows:
\begin{equation}
    \label{eq: rev.mixture}
    P(B)=\int_{\Cs{A}} Q_{\bs{a}}(B)  \ M (d \bs{a}), \quad B \in \Cs{F},
\end{equation}
where $M$ is any probability distribution on $\Cs{A}$ \cite{lasserre2018representation}. 
Hence, modeling the ambiguity in the mixture probabilities may give rise to a \dro\ model over the {\it resultant or barycenter} $P$ of $M$ \cite{popescu2005semidefinite}. 

\paragraph{Risk and Chance Constraints}

\citet{lasserre2018representation} study a distributionally robust (individual and joint) chance-constrained program with  a polynomial objective function, over a mixture ambiguity set and a semi-algebraic deterministic set. They approximate  the ambiguous chance constraint with a polynomial whose vector coefficients is an optimal solution of a SDP. They show that the induced feasibility set by a nested sequence of such   polynomial optimization approximation problems converges to that of  the  ambiguous  chance constraints as the degree of approximate polynomials increases.

\citet{kapsos2014} introduce a probability Omega ratio  for portfolio optimization (i.e., a probability weighted ratio of gains versus losses for some threshold return target). They study a distributionally robust counterpart of this ratio, where each distribution of the ratio can be represented through   a mixture of some known prespecified distributions with unknown mixture probabilities. In particular, they study a mixture model for a nominal discrete distribution, where the mixture probabilities are modeled via the box uncertainty and ellipsoidal uncertainty models. In the former case, they reformulate the problem as a linear program, and in the latter case, they reformulate the problem as a SOCP.

\citet{hanasusanto2015NV} study a distributionally robust newsvendor model with a mean-risk objective, as a convex combination of the worst-case CVaR and the worst-case expectation. The worst case is taken over all demand distributions
within a {\it multimodal} ambiguity set, i.e., a mixture of a finite number of modes, where the conditional information on the ellipsoid support, mean, and covariance of each mode is known.  The ambiguity in each mode is modeled via \eqref{eq: rev.DelageYe}. They cast the resulting model as  an exact SDP, and obtain a conservative semidefinite approximation by using quadratic decision rules to approximate the recourse decisions. \citet{hanasusanto2015NV} further robustify their model against ambiguity in estimating the mean-covariance information, caused from ambiguity about the mixture weights. They assume that the mixture weights are close  to a  nominal  probability vector in the sense of $\chi^2$-distance. For this case, they also obtain exact SDP reformulation as well as a conservative SDP approximation. %We also refer to \citet{chen2018chance} for studying a \dro\ approach to chance constraints. 

%%%%%%%%%%%%%%%%%%%%%%%%%%%%%%%%%%%%%
\subsection{Shape-Preserving Models}
\label{sec: rev.shape}

A few papers propose to model the distributional ambiguity in a way that all distributions in the ambiguity set share similar structural properties. 
We refer to such models  as {\it shape-preserving} models to form the ambiguity set of probability distributions.  
%Moment-based ambiguity sets are known to result in overly conservative decisions because a worst-case distribution is attained by a discrete distribution, see, e.g., \citet{scarf1958} and \cite{shapiro2012minimax}. Such a distribution may  not be realistic in practical systems. To overcome this issue, 

\citet{popescu2005semidefinite} propose to incorporate structural distributional information, such as symmetry, unimodality, and convexity, into a moment-based ambiguity set.
The proposed ambiguity set is of the following generic form:
\begin{equation}
\label{eq: rev.shape_set}
\Cs{P}^{SP}:=\sset*{P \in\M}{ \int_{\Xi} \bs{f} d P = \bs{a} } \cap \{P  \ \text{satisfies structural properties}\}.
\end{equation}
\citet{popescu2005semidefinite} obtains upper and lower bounds on a generalized moment of a random vector (e.g., tail probabilities), given the moments and structural constraints in a convex subset of the proposed  ambiguity set \eqref{eq: rev.shape_set}. 
\citet{popescu2005semidefinite} uses conic duality to evaluate such  lower and upper bounds via   SDPs. The key to the development in 
\citet{popescu2005semidefinite} is to focus on  ambiguity sets that posses a {\it Choquet representation}, where every distribution in the ambiguity set can be written as a mixture (i.e., an infinite convex combination) of measures in a generating set and in the virtue of \eqref{eq: rev.mixture}. For univariate distributions, it is assumed that the generating set is defined by a Markov kernel.  %Such an assumption allows to reduce the infinite number of constraints in the dual formulation of $\sup_{P \in \Cs{P}^{SP}} \ \ee{P}{h(\bs{x},\txi)}$. 
It is  shown that if the optimal value of the problem is attained, there exists a worst-case probability measure that is a convex combination of $m+1$ (recall $m$ is the dimension of $\bs{f}$) (extremal) probability measures from the generating set. 
\citet{popescu2005semidefinite} uses the above result to obtain generalized Chebyshev's inequalities bounds for distributions of a univariate random variable that are (1) symmetric, (2) unimodal
with a given mode, (3) unimodal with bounds on the mode, (4) unimodal and symmetric, or (5) convex/concave monotone  densities with bounds on the slope of densities. 
\citet{popescu2005semidefinite} further derives generalized Chebyshev's inequality for symmetric and unimodal
distributions of multivariate random variables. 
A related notion to unimodality is $\alpha$-unmiodality, which is defined as follows:
\begin{definition}{\citet{dharmadhikari1988unimodality}}
\label{def: rev.alpha_model}
For $\alpha>0$, a distribution $\Ts{P} \in \P$ is called $\alpha$-unimodal with mode $a$ if $\frac{\Ts{P}\{t (A-a)\}}{t^{\alpha}}$ is nonincreasing in $t>0$ for all $A \in \Cs{B}(\Bs{R}^{d})$.
\end{definition}

 \citet{vanparys2016SDP}  further extend the work of \citet{popescu2005semidefinite} %is motivated by the same caveat of the moment-based ambiguity sets, and they 
to obtain worst-case probability bounds over $\alpha$-unimodal multivariate distributions with the same mode and within the class of distributions in $\Cs{P}^{\text{DY}}$, defined in \eqref{eq: rev.DelageYe}, and on a  polytopic support. They show that when the support of the random vector is an open polyhedron, this generalized Gauss bound can be obtained  via a SDP. 
Similar to  \citet{popescu2005semidefinite}, \citet{vanparys2016SDP} derive semidefinite representations for worst-case probability bounds using Choquet representation of the ambiguity set. They  demonstrate that classical generalized Chebyshev and Guass bounds\footnote{The random variable differs from its mean by more than $k$
standard deviations.}  can be obtained as special cases of their result. They also show how to obtain a SDP reformulation to obtain the worst-case bound over $\alpha$-multimodal multivariate distributions, defined via a mixture distribution.

By relying on information from classical statistics as well as robust statistics, \citet{hanasusanto2015chance} propose a unifying canonical ambiguity set that contains many ambiguity sets studied in the literature as special cases, including  Gauss and median-absolute deviation ambiguity sets. %, Huber, and Wasserstein .  
Such a canonical framework  is characterized through intersecting the cross-moment ambiguity set, proposed in \citet{wiesemann2014}, and a  structural ambiguity set  on the marginal distributions, representing information  such as symmetry and $\alpha$-unimodality. %, unimodality, $\alpha$-unimodality, multimodality, or independence patterns. 
As in \cite{popescu2005semidefinite}, the key to the development in \citet{hanasusanto2015chance} is to focus on structural ambiguity sets that posses a  Choquet representation. %, where every distribution in the set can be written as a mixture (i.e., an infinite convex combination) of extremal distributions of the set. 
%\citet{hu2014shape}
They study  distributionally robust uncertainty quantification (i.e., a probabilistic objective function) and  chance-constrained programs over the proposed ambiguity sets, where the  safe region is characterized by a bi-affine expression in $\txi$ and  $\bs{x}$%\footnote{Recall the argument following \eqref{eq: SO_Obj} and \eqref{eq: SO_Cons}, where we gave a characterization of $A(\bs{x})$ as $\bs{a}(\bs{x})^{\top}\txi \le \bs{b}(\bs{x})$. A safe region characterized by a bi-affine expression in $\txi$ and  $\bs{x}$ means that  both $\bs{a}(\bs{x})$ and $\bs{b}(\bs{x})$ are affine in $\bs{x}$.}
.
They study the ambiguity sets over which the resulting problems are reformulated as conic programming formulations. A summary of these results can be found in \citet[Table 2]{hanasusanto2015chance}.
A by-product of their study is to recover some results from probability theory. 
For instance, by studying the worst-case probability of an event over the Chebyshev ambiguity set with a known mean and upper bound on the covariance matrix, they recover the generalized Chebyshev inequality, discovered in  \citet{popescu2005semidefinite,vandenberghe2007}. Similarly, they recover  the generalized Gauss inequality, discovered in \citet{vanparys2016SDP}, by considering the Gauss ambiguity set. 
Furthermore, they  propose computable conservative approximations for the chance-constrained problem. Recognizing  that the  uncertainty quantification problem is tractable over a broad range of ambiguity sets, their key idea  for the proposed approximation scheme is to decompose the chance-constrained problem into an uncertainty quantification problem that evaluates the worst-case probability of the chance constraint  for a fixed decision $\bs{x}$, followed by a decision improvement procedure. 
%To obtain tractable conic quadratic formulation a for a joint constraint problem, they restrict the problem to a Markov ambiguity set and a fixed technology matrix for the chance constraint.   

\citet{li2017} study distributionally robust chance- and CVaR-constrained stochastic programs, where  the ambiguity set contains all $\alpha$-unimodal distributions with the same first two order moments,  and the safe region is bi-affine in both $\txi$ and $\bs{x}$. They show that these two ambiguous risk constraints can be cast as an infinite set of SOC constraints. They  propose a separation approach   to find the violated  SOC  constraints in an algorithmic fashion. They also derive    conservative and relaxation approximations of the two SOC constraints by a finite number of constraints. These approximations for the CVaR-constrained problem are based on  the results in \citet{vanparys2017structured}. 
%semidefinite approximations of the  resulting model.  %They apply their results to a power network. 

%\citet{vanparys2017structured}
\citet{hu2015} study a data-driven newsvendor problem to decide on the optimal order quantity and price. They assume that demand depends on the pricing, however, there is ambiguity about the price-demand function. To hedge against the misspecification of the demand function, they introduce a novel approach to  this problem, called  {\it functionally robust} approach, where   the demand-price function is only known to be decreasing convex or concave. The proposed modeling approach in \citet{hu2018}  also  provides a systematic view on the risk-reward trade-off of coordinating pricing and order quantity decisions based on the size of the ambiguity set. To solve the resulting minimax model, \citet{hu2018}  reduce the problem into a univariate problem that seeks the optimal pricing and develop a two-sided cutting surface algorithm that generates function cuts to shrink the set of  admissible functions. 

To overcome the difficulty in evaluating extremal performance due to the lack of data, \citet{lam2017tail} study the computation of worst-case bounds under the geometric premise of the tail convexity. They show that 
the worst-case convex tail behavior is in a sense either extremely light-tailed or extremely
heavy-tailed. 

%%%%%%%%%%%%%%%%%%%%%%%%%%%%%%%%%%%%%

%%%%%%%%%%%%%%%%%%%%%%%%%%%%%%%%%%%%%
\subsection{Kernel-Based Models}
\label{sec: rev.kernel}
%%%%%%%%%%%%%%%%%%%%%%%%%%%%%%%%%%%%%

In Sections \ref{sec: rev.distance}--\ref{sec: rev.shape}, we discussed different sets to model the distributional ambiguity. In all the papers we reviewed in those sections, the form of  ambiguity set is endogenously chosen by decision makers. However, when facing high-dimensional uncertain parameters, it may not be practical to fix the form of  ambiguity set a priori, being even more complicated with the calibration of different parameters describing the set (see Section \ref{sec: rev.calibration}).  An alternative practice is to learn the form of the ambiguity set by using  unsupervised learning algorithms  on the historical data. 
Consider a given set of data $\{(\bs{u}^{i},\bs{\xi}^{i})\}_{i=1}^{N}$, where $\bs{u}^{i} \in \Bs{R}^{m}$ is a vector of covariates associated with the uncertain parameter of interest $\bs{\xi}^{i} \in \Bs{R}^{d}$. % principle, these unsupervised learning algorithms rely on 
Let $K: \Bs{R}^{d} \times \Bs{R}^{d} \mapsto \Bs{R}$ be a {\it kernel} function. %, defined based on a mapping  from $\Bs{R}^{d}$ to a higher-dimensional space.

\citet{bertsimas2018predictive} propose a decision framework that incorporates the covariates $\bs{u}$  in addition to $\bs{\xi}$ into the optimization problem in the form of a conditional-stochastic optimization problem, where the decision-maker is seeking a {\it predictive prescription} $\bs{x}(\bs{u})$ that minimizes the conditional expectation of $h(\bs{x},\txi)$ in anticipation of the future, given the observation $\bs{u}$. However, the conditional distribution of $\txi$ given $\bs{u}$ is not known and should be learned from data. Given $\{(\bs{u}^{i}, \bs{\xi}^{i})\}_{i=1}^{N}$, they suggest to find a data-driven predictive prescription that minimizes $\sum_{i=1}^{k} w_{k}^{i}(\bs{u}) h(\bs{x}, \bs{\xi}^{i})$ over $\Cs{X}$. Functions $w_{k}^{i}(\bs{u})$ are weights  learned locally  from the data, in a way that predictions are made based on the mean or mode of the past observations that are in some way similar to the one at hand. \citet{bertsimas2018predictive}  obtain these weight functions  by methods that are motivated by  $k$-nearest-neighbors regression, Nadaraya-Watson kernel regression, local linear regression (in particular, LOESS), classification and regression trees (in particular, CART), and random forests. 
For instance, the estimate of  $\ee{P}{h(\bs{x}, \txi)\Big|\bs{u}}$ using the  Nadaraya-Watson kernel regression is obtained as 
$$
\sum_{i=1}^{N} \frac{K_{b}(\bs{u}-\bs{u}^{i})}{\sum_{i=1}^{N} K_{b}(\bs{u}-\bs{u}^{i})} h(\bs{x}, \bs{\xi}^{i}),
$$
where $K_{b}(\cdot):=\frac{K(\frac{\cdot}{b})}{b}$ is a kernel function with bandwidth $b$. Common kernel smoothing functions are 
\begin{itemize}
    \item Naive: $K(a)= \mathbbm{1}_{[\|a\| \le 1]}$,
    \item Epanechnikov: $K(a)=(1- \|a\|^{2}) \mathbbm{1}_{[\|a\| \le 1]}$,
    \item Tri-cubic: $K(a)=(1- \|a\|^{3})^{3} \mathbbm{1}_{[\|a\| \le 1]}$,
    \item Guassian or radial basis function: $K(a)=\frac{1}{\sqrt{2\pi}} \exp(- \frac{\|a\|^{2}}{2} )$.
\end{itemize}

The general framework of the proposed data-driven model in \citet{bertsimas2018predictive} resembles SAA. They  show that under mild conditions, the problem is polynomially solvable and the resulting predictive prescription  is asymptotically optimal   and consistent. However, it is worth noting that \citet{bertsimas2018predictive} illustrate that direct usage of SAA on $\{\bs{\xi}^{i}\}_{i=1}^{N}$ and ignoring $\{\bs{u}^{i}\}_{i=1}^{N}$ can result in suboptimal decisions which are  neither asymptotically optimal nor consistent. 

A similar modeling framework as the conditional stochastic optimization problem studied in \citet{bertsimas2018predictive} is investigated  in other papers, see, e.g., \citet{hannah2010,deng2018LEO,ban2019,ho2019}, to incorporate machine learning into decision making. \citet{deng2018LEO} use regression models such as $k$-nearest-neighbors regression to learn the conditional distribution  of $\txi$ given $\bs{u}$. They  study the statistical optimality of the resulting solution and its generalization error, and they provide hypothesis-based tests for model validation and selection.  
In \citet{hannah2010,ban2019,ho2019}, the weights are obtained by the  Nadaraya-Watson  kernel  regression method. 
For a newsvendor problem, \citet{ban2019} show that the  SAA decision does not converge to the true optimal decision. This motivates them to derive generalization bounds for the out-of-sample performance of the cost and the finite-sample bias from the true optimal decision. \citet{ban2019} apply their study to the staffing levels of nurses for a hospital emergency room. 

\citet{tulabandhula2013ML} incorporate machine learning for the decision making. But, different from \citet{bertsimas2018predictive},  they study a framework that simultaneously seeks a best statistical model and a corresponding decision policy. In their framework, in addition to $\{(\bs{u}^{i},\bs{\xi}^{i})\}_{i=1}^{N}$, a new set of unlabeled data is available that in conjunction with the statistical model  affects the cost. The minimum of such a cost function over the set of possible decisions is cast by a regularization term in the objective function  of the learning algorithm. \citet{tulabandhula2013ML} show that under some conditions this problem is equivalent to a robust optimization model, where the uncertainty set of the statistical model contains all models that are within $\epsilon$-optimality from the predictive model describing $\{(\bs{u}^{i},\bs{\xi}^{i})\}_{i=1}^{N}$. They illustrate the form of the uncertainty set for different loss functions used in the predictive statistical model, including least squares, 0-1, logistic, exponential, ramp, and hing losses. \citet{tulabandhula2014combining} study the application of the framework studied in \citet{tulabandhula2013ML} to a travelling repairman problem, where a repair crew  is seeking for an optimal route to repair the nodes on a graph while the failure probabilities are unknown. 

Similar  to \citet{tulabandhula2013ML}, \citet{tulabandhula2014ML} use a new set of unlabeled data in addition to $\{(\bs{u}^{i},\bs{\xi}^{i})\}_{i=1}^{N}$ in order to combine machine learning and decision making. However, unlike \citet{bertsimas2018predictive}, \citet{deng2018LEO}, \citet{tulabandhula2013ML}, and \citet{tulabandhula2014combining}, \citet{tulabandhula2014ML} study a robust optimization framework. Their idea to  form the uncertainty set of $\txi$ is to consider a class of ``good" predictive models with low training error on the data set $\{(\bs{u}^{i},\bs{\xi}^{i})\}_{i=1}^{N}$. Recognizing that the uncertainty can be decomposed into the predictive model uncertainty and residual uncertainty, they form the uncertainty by the Minkowski sum of two sets: (1) predictions of the new data set with the class of ``good" predictive models, and (2) residuals of the new data set with the class of ``good" predictive models. To form the class of ``good" predictive models, one can use loss functions such as least squares and hing loss.

Similar to \citet{bertsimas2018predictive}, \citet{bertsimas2017} consider the problem of finding an optimal solution to a data-driven stochastic optimization problem, where the uncertain parameter is affected by a large number of covariates. They study a distributionally robust approach to this problem formed via Kullback-Leibler divergence. By borrowing ideas from the statistical bootstrap, they propose two  prescriptive methods
based on the Nadaraya-Watson and nearest-neighbors learning formulation, first introduced by \citet{bertsimas2018predictive}, which safeguards against overfitting and lead to an improved out-of-sample performance. Both resulting prescriptive methods reduce to tractable convex optimization problems.

Kernel density estimation (KDE) \cite{devroye1985} in combination with {\it principal component analysis} (PCA) is also  used in the RO literature to construct the uncertainty set \citep{ning2018kernel}. PCA captures the correlation between uncertain parameters and transfoms data into their corresponding uncorrelated principal components. KDE, then, captures the distributional information of the transformed, uncorrelated uncertain parameters along the principal components,   by using kernel smoothing methods. 
\citet{ning2018kernel} propose to use a Gaussian kernel $K$ defined between the latent uncertainty along the  principal component $k$,  $w_{k}$, and the projected data along the principal component $k$, $t_{k}$\footnote{It is known that for any positive definite symmetric kernel $K$, there is a mapping $\Phi$ from the covariates space  to a higher-dimensional space $\Bs{H}$ such that   $K(\xi_{k},t_{k})$ is equal to the inner product  between $\Phi(\xi_{k})$ and $\Phi(t_{k})$, see, e.g., \citet[Theorem~5.2]{mohri2018foundations}. Such a space $\Bs{H}$ is called {\it reproducing kernel Hilbert space}. A kernel is said to be positive definite symmetric if the induced kernel matrix is symmetric positive semidefinite.}.
By incorporating forward and backward deviations to allow for asymmetry \cite{chen2007robust}, \citet{ning2018kernel} propose the following polytopic uncertainty set that resembles the intersection of a box, with the so-called {\it budget}, and polyhedral uncertainty sets:
\begin{equation*}
	\Cs{U}=\sset*{\bs{u}}{ 
		\begin{aligned}  
			& \bs{u}=\bs{\mu}_{0}+ \bs{V} \bs{w}, \; \bs{w}= \bs{\ul{w}} \odot \bs{z}^{-} +  \bs{\ol{w}} \odot \bs{z}^{+}, \\
			& \bs{0} \le \bs{z}^{-}, \bs{z}^{+} \le \bs{1}, \; \bs{z}^{-}+\bs{z}^{+} \le \bs{1}, \; \bs{1}^{\top} (\bs{z}^{-}+\bs{z}^{+}) \le \Gamma, \\
			& \ul{\bs{w}}= [F^{-1}_{1}(\alpha), \ldots, F^{-1}_{m}(\alpha)]^{\top}, \\
			& \ol{\bs{w}}= [F^{-1}_{1}(1-\alpha), \ldots, F^{-1}_{m}(1-\alpha)]^{\top}
		\end{aligned}  
	}.
\end{equation*} 
Let us define $\bs{U}=[\bs{u}^{1}, \ldots, \bs{u}^{N}]^{\top}$. Above $\bs{\mu}_{0}=\frac{1}{N} \sum_{i=1}^{N} \bs{u}^{i}$, and $\bs{V}$ is a square matrix consists of all  $m$ eigenvvectors (i.e., principal components) obtained from the eignevalue decomposition of the sample covariance matrix $\bs{S}=\frac{1}{N-1} (\bs{U}-\bs{1}\bs{\mu}_{0}^{\top})^{\top}(\bs{U}-\bs{1}\bs{\mu}_{0}^{\top}) $. Moreover, $\bs{z}^{-}$ is a backward deviation,  $\bs{z}^{+}$ is a forward deviation vector, and $\Gamma$ is the uncertainty budget. In addition, $F^{-1}_{k}:=\min\{w_{k}|F_{k}(w_{k}) \ge \alpha\}$, $k=1, \ldots, m$, where $F_{k}(w_{k})$ is the cdf of $w_{k}$, with the density function is obtained using KDE as follows: $f_{k}(w_{k})=\frac{1}{N} \sum_{i=1}^{n} K_{b}(w_{k}, t_{k}^{i})$. \citet{ning2018kernel} further extend their approach to the data-driven static and adaptive robust optimization. 

In the context of RO, {\it support vector clustering} (SVC) is proposed to form the uncertainty set, which seeks for a sphere with the smallest radius that encloses all data mapped in the covariate space \cite{shang2017}. In SVC,  to avoid overfitting,  the violations of the data outside the sphere is penalized by a regularization term as follows:
\begin{align*}
    \min_{\delta, \bs{s}, \bs{c}} \ & \delta^{2} + \frac{1}{N\gamma} \sum_{i=1}^{N}  s_{i} \\
    \st \quad & \|\Phi(\bs{u}^{i})-\bs{c}\|_{2}^{2} \le  \delta^{2} +s_{i}, \; i=1, \ldots, N,\\
    & \bs{s}\ge \bs{0}. 
\end{align*}
Dualizing  the problem of finding the smallest sphere  using dual multipliers $\bs{\pi}$ results in a  quadratic problem where the kernel function appears in the objective function. It is shown that commonly used kernel functions in SVC, such as  polynomial, radial basis function, sigmoid function kernel, lead to an intractable robust counterpart problem for the corresponding uncertainty set. Hence, \citet{shang2017} propose to use a piecewise linear kernel, referred to as a {\it  weighted generalized intersection kernel}, defined as follows: 
\begin{equation}
    \label{eq: rev.kernel}
    K(\bs{u},\bs{v})=\sum_{k=1}^{m} l_{k} - \|\bs{Q}(\bs{u}-\bs{v})\|_{1},
\end{equation}
where $\bs{Q}=\bs{S}^{-\frac{1}{2}}$ and $\bs{S}=\frac{1}{N-1} \sum_{i=1}^{N} \Big[ \bs{u}^{i} (\bs{u}^{i})^{\top} - \big(\sum_{i=1}^{N}\bs{u}^{i}\big) \big(\sum_{i=1}^{N}\bs{u}^{i}\big) ^{\top} \Big]$, and $l_{k}$, $k=1, \ldots, m$, is chosen such that $l_{k} > \max_{i=1}^{N} \bs{Q}_{\cdot k}^{\top}\bs{u}^{i} -\min_{i=1}^{N} \bs{Q}_{\cdot k}^{\top}\bs{u}^{i}$.  
Such a kernel not only incorporates covariance information, but also gives rise to the following  results. %Given the above construction, we state the following results from \cite{shang2017}. 
\begin{theorem}{(\citet[Propositions~1, Propositions~3--4]{shang2017})}
    \label{thm: shang_SVC}
    Suppose that the kernel function is constructed as in \eqref{eq: rev.kernel}. Then,
    \begin{enumerate}[label=(\roman*)]
        \item The kernal matrix induced by the kernel $K$ is positive definite. %; hence, the problem of finding the smallest sphere is a convex optimization problem. 
        \item The constructed  uncertainty set 
        \begin{equation*}
        	\Cs{U}=\sset*{\bs{u}}{ 
        		\begin{aligned}  
        			& \exists \bs{v}_{i}, \; i \in \Cs{S} \ \st \\
        			& \sum_{i \in \Cs{S}} \pi_{i} \bs{v}_{i}^{\top} \bs{1} \le \epsilon, \\
        			& -\bs{v}_{i} \le \bs{Q}(\bs{u}-\bs{u}^{i}) \le \bs{v}_{i}, \; i \in \Cs{S}
        		\end{aligned}  
        	},
    \end{equation*} 
        where $\Cs{S}:=\sset*{i}{\pi_{i} > 0}$,  $\epsilon=\sum_{i \in \Cs{S}} \pi_{i} \|\bs{Q}(\bs{u}^{j}-\bs{u}^{i})\|_{1}$, $j \in \Cs{B}$, and  $\Cs{B}:=\sset*{i}{0 < \pi_{i} < \frac{1}{N\gamma}}$, 
        is a polytope; hence, the robust counterpart \linebreak $\max_{\bs{u} \in \Cs{U}} \ \bs{u}^{\top} \bs{x} \le b$ has the same complexity as  the deterministic problem. 
        \item The regularization parameter $\gamma$ gives an upper bound on the fraction of the outliers; hence, a feasible solution $\bs{x}$ in the robust counterpart $\max_{\bs{u} \in \Cs{U}} \ \bs{u}^{\top} \bs{x} \le b$ is also feasible to a SAA-based chance-constrained problem $P\{\tbs{u}^{\top} \bs{x} \le b\} \ge 1-\gamma$.
        \item As the number of data points increases, the fraction of outliers converges to the regularization parameter $\gamma$ with probability one. 
        \item The regularization parameter $\gamma$ gives a lower  bound on the fraction of the support vectors. 
    \end{enumerate}
\end{theorem}

\citet{shang2018predictive} further propose to calibrate the radius of the uncertainty set  and provide a probabilistic guarantee of the proposed uncertainty set. 
\citet{shang2018SVC}  use PCA in combination with SVC to construct the uncertainty set. By employing PCA, the data space is decomposed into the principal subspace and residual subspace. Then, they utilize  the uncertainty set formed in \citet{shang2017} to explain the variation in the principal subspace, and  utilize a polyhedral set to explain noise in the residual subspace. The proposed uncertainty set is then the intersection of the above two sets. 
\citet{shang2018dDROscheduling} adopt the ambiguity set proposed in \citet{wiesemann2014}, and propose to use PCA to calibrate the moment functions. In fact, a moment function in their model is a piecewise linear function, which  is defined as a  first-order deviation of the uncertain parameter along a certain projection direction, truncated at certain points. They propose to use PCA to come up with the projection directions, and  choose the truncation points symmetrically around the sample mean along the direction.

Applications of the proposed method in \citet{ning2018kernel} are studied in  production scheduling \citep{ning2018kernel} and in process network planning \citep{ning2018kernel,ning2018hedging,ning2018PCA}. 
The proposed method in \citet{shang2017} is used in different application domains to construct the uncertainty set, see, e.g., control of irrigation system \citep{shang2018robust} and chemical process network planning \citep{shang2017}.   
Applications of the proposed method in \citet{shang2018dDROscheduling} are studied in  production scheduling \citep{shang2018dDROscheduling,shang2018process} and in process network planning \citep{shang2018dDROscheduling,shang2018DRO}.

%%%%%%%%%%%%%%%%%%%%%%%%%%%%%%%%%%%%%
\subsection{General Ambiguity Sets}	
\label{sec: rev.general}
%%%%%%%%%%%%%%%%%%%%%%%%%%%%%%%%%%%%%

In Sections \ref{sec: rev.distance}--\ref{sec: rev.kernel}, we reviewed papers with specific distributional and structural properties for the random parameters, captured  via discrepancy-based, moment-based,  shape-preserving, and kernel-based ambiguity sets. In this section, we review papers that either do not consider any specific form for the ambiguity set or  provide some  general results for a broad class of ambiguty sets. 

A unified scenario-wise format for ambiguity sets to contain both the moment-based and discrepancy-based distributional information about the ambiguous distribution is proposed in \citet{chen2018adaptive}. It is shown that the ambiguity sets formed via generalized moments, mixture distribution, Wasserstein
metric, $\phi$-divergence, $k$-means clustering, among other, all can be  represented under this unified ambiguity set.  The key feature of
this scenario-wise ambiguity set is the introduction of a discrete random variable, which represents
a finite number of scenarios that would affect the distributional ambiguity of the underlying nominal random
variable. This  ambiguity set can be characterized by a finite number of (conditional) expectation constraints based on
generalized moments \citet{wiesemann2014}. For practical purposes, they restrict the ambiguity set to be second-order conic  representable. Based on the scenario-wise ambiguity set, they introduce an adaptive robust optimization format that unifies the  classical SP and (distributionally) RO
models with recourse. They also introduce a  scenario-wise affine recourse approximation to provide tractable solutions to
the adaptive robust optimization model. 
Besides \citet{chen2018adaptive}, there are some proposals for unified models in the context of discrepancy-based, moment-based,  and shape-preserving models. As mentioned before, a broad class of moment-based ambiguity
sets with conic-representable expectation constraints and a collection of nested conic-representable
confidence sets is proposed in \citet{wiesemann2014}, and a broad class of shape-preserving ambiguity sets is proposed in \citet{hanasusanto2015chance}.

\citet{luo2018} study \dro\  problem where the ambiguity sets of probability distributions can depend on the decision variables. They consider a wide range of moment- and discrepancy-based ambiguity sets formed, such as (1)  measure and moment inequalities (see Section \ref{sec: rev. measure_marginal_moments}), (2) bounds on moment constraints (see Section \ref{sec: rev.Chebyshev}), (3)  $1$-Wasserstein metric utilizing $\ell_1$-norm, (4) $\phi$-divergences, and (5) Kolmogorov-Smirnov test. They present equivalent reformulations for these problems by relying on  duality results. 

\citet{pflug2007} study a \dro\ problem, where the ambiguity exists in both the objective function and constraints as in \eqref{eq: DRO}. To solve the model, they propose an exchange method to  successively   %convex programming solution method, which uses a
generate a finite inner approximation of the ambiguity
set of distributions. They show that when the ambiguity set is compact and convex, and the risk measure is jointly continuous in both $\bs{x}$ and $\Ts{P}$, then the proposed algorithm is finitely convergent.

\citet{bansal2018conic}   introduce  two-stage stochastic integer programs in which the second-stage problem have $p$-order conic constraints as well as integer variables. They present sufficient conditions under which the addition of parametric (non)linear
cutting planes along with the linear relaxation of the integrality constraints provides a convex programming equivalent for the second-stage problem. They show that this result is also valid for the distributionally robust counterpart of this problem. This paper generalizes the results on two-stage mixed-binary linear programs  studied in \citet{bansal2018}. 

\citet{bansal2018solving} introduce two-stage distributionally robust disjunctive programs with
disjunctive constraints in both stages and a general ambiguity set for the probability distributions. 
To solve the resulting model, they develop decomposition algorithms, which utilize Balas' linear programming equivalent for deterministic
disjunctive programs or his sequential convexification approach within the L-shaped method. They  demonstrate that the proposed algorithms are finitely convergent if a distribution separation subproblem can be solved in a finite number of iterations, as in sets formed via $\Cs{P}^{\text{MM}}$, defined in \eqref{eq: moment-rob-set}, $1$-Wasserstein metric utilizing an arbitrary norm, and the total variation distance.
These algorithms generalize the distributionally robust integer L-shaped algorithm of   \citet{bansal2018} for two-stage mixed binary linear programs.

\citet{wang2019OR} study a distributionally robust chance-constrained bin-packing problem with a finite number of scenarios, where the safe region of the chance constraint is bi-affine in $\bs{x}$ and $\txi$, with a random technology matrix. They present a binary bilinear reformulation of the problem, where the feasible region is modeled as the intersection of multiple binary bilinear knapsack constraints, a cardinality constraint, and a general (probability) knapcksack constraint. They propose lifted cover valid inequalities for the binary bilinear knapsack substructure induced by a given bin and scenario, and they further obtain lifted  cover inequalities that are valid for the substrcture induced by each bin.
They obtain valid probability cuts and incorporate them with the lifted cover  inequalities in a branch-and-cut framework to solve the model. They show that the proposed algorithm is finitely convergent if a distribution separation subproblem can be solved in a finite number of iterations. \citet{wang2019OR} apply their results to an operating room scheduling problem. 

\citet{guo2017convergence} study the impacts of the variation of the ambiguity set of probability distributions on the optimal value and  optimal solution of the stochastic programs with distrubutionally robust chance constraints. To establish the results, they present conditions under which a sequence of approximated ambiguity sets converges to the true ambiguity set, for some discrepancy measure, including Kolmogorov  and the total variation distance. They apply their convergence results to the ambiguity sets formed via \eqref{eq: moment-rob-set} and Kullback-Leibler divergence.

\citet{delage2018mip} study the value of using a randomized policy, as compared to a deterministic policy,  for mixed-integer \dro\ problems. They show that the value of randomization for such \dro\ models with a convex cost function $h$ and a convex risk measure
is bounded by the difference between the optimal values of the nominal \dro\ 
problem and that of its convex relaxation. They show that when the risk measure is an expectation and the cost function is affine in the decision vector, this bound
is tight. They also develop
a column generation algorithm  for solving a two-stage mixed-integer linear \dro\ problem, formed via  \eqref{eq: moment-rob-set} and
$1$-Wasserstein metric utilizing an arbitrary norm. They test their results on assignment problem, and on uncapacitated and capacitated facility location problems.

% \citet{gupta2018} 

\citet{long2014} study a distributionally robust binary stochastic program to minimize the entropic VaR, also known as Bernstein approximation
for the chance constraint. They propose an approximation algorithm to solve the problem via solving a sequence of problems. They showcase their results for ambiguity set formed as in \eqref{eq: moment-rob-set} for a stochastic shortest path problem.

\citet{shapiro2013worst} study a  multistage stochastic program, where the data process can be naturally separated into two components: one can be modeled as a random process, with a known
probability distribution, and the other  can be treated as a random process, with a known support and no distributional information. They propose a variant of the stochastic dual dynamic programming (SDDP) method to solve this problem.

% % % % % % % % % % % % % % % % % % % % % % % % % % % % % % % % % % % % % % % % % % % %	
\section{Calibration of the Ambiguity Set of Probability Distributions}
\label{sec: rev.calibration}
% % % % % % % % % % % % % % % % % % % % % % % % % % % % % % % % % % % % % % % % % % % %	

%%%%%%%%%%%%%%%%%%%%%%%%%%%%%%%%%%%%%
\subsection{Choice of the Nominal Parameters}
%%%%%%%%%%%%%%%%%%%%%%%%%%%%%%%%%%%%%
All discrepancy-based ambiguity sets, studied in Section \ref{sec: rev.distance}, and some of the moment-based ambiguity sets, studied in Section \ref{sec: rev.moment}, rely on some nominal input parameters, for instance, the nominal distribution $P_{0}$  in the ambiguity set $\Cs{P}^{\text{W}}(P_{0}, \epsilon)$, defined in \eqref{eq: rev.opt.transport.set},  and parameters $\bs{\mu}_{0}$ and $\bs{\Sigma}_{0}$ in the  ambiguity set $\Cs{P}^{\text{DY}}$, defined in \eqref{eq: rev.DelageYe}.
In this section, we discuss how these parameters are chosen in a data-driven setting.

The nominal distribution $P_{0}$ in the  discrepancy-based ambiguity sets is usually obtained by the maximal likelihood estimator of the true unknown distribution. In the discrete case, $P_{0}$ is typically chosen as the empirical distribution on data. 
In the  case that the true unknown distribution is continuous, \citet{jiang2018} and \citet{zhao2015} propose to obtain $P_{0}$ with nonparametric kernel density estimation methods, see, e.g., \citet{devroye1985}. 

\citet{delage2010} propose to estimate $\bs{\mu}_{0}$ and $\bs{\Sigma}_{0}$ by their empirical estimates (see Section \ref{sec: rev.robustness_param} for more details on how this choice of nominal parameters, in conjuction with other assumptions, ensure that the constructed  ambiguity set $\Cs{P}^{\text{DY}}$ contains the true unknown probability distribution with a high probability). 

%%%%%%%%%%%%%%%%%%%%%%%%%%%%%%%%%%%%%
\subsection{Choice of Robustness Parameters}
\label{sec: rev.robustness_param}

%%%%%%%%%%%%%%%%%%%%%%%%%%%%%%%%%%%%%

In  Section \ref{sec: rev.choice.ambiguity}, we reviewed different approaches to form the ambiguity set of distributions. All  discrepancy-based ambiguity sets, studied in Section \ref{sec: rev.distance}, and some of the moment-based ambiguity sets, studied in Section \ref{sec: rev.moment}, rely on parameters that control the size of the ambiguity set. 
For instance, parameter $\epsilon$ in the ambiguity set $\Cs{P}^{\text{W}}(P_{0}; \epsilon)$, defined in \eqref{eq: rev.opt.transport.set},  and parameters $\varrho_{1}$ and $\varrho_{2}$ in the  ambiguity set $\Cs{P}^{\text{DY}}$, defined in \eqref{eq: rev.DelageYe}, control the size of their corresponding ambiguity sets. 
A judicial choice of these parameters reduce the level of  conservatism of the resulting \dro. A natural question  is then how to choose appropriate values for these parameters. 

In this section, we review different approaches to choose the level-of-robustness parameters. To have a structured review, we  make a distinction between data-driven \dro\ and non-data-driven \dro. 

\subsubsection{Data-Driven \dro s}
%Naturally, the more data points are available from the unknown true distribution, the less ambiguous the underlying distribution is. This simple fact is a general guideline to choose the level of robustness in data-driven \dro s. 
Data-driven \dro s usually propose a robustness parameter that is inversely proportional to the number of available data points. %As  the number of data points increases, the size of the ambiguity set goes to zero. Hence, the ambiguity set shrinks to a singleton set, containing the true unknown distribution. 
This construction is motivated from the asymptotic convergence of the optimal value of \dro\ to that of the corresponding model under the true unknown distribution, with an increasing number of data points, see, e.g., \cite{pflug2007,delage2010,bertsimas2018RO}.  

An underlying assumption in data-driven methods is that data points are independently and identically distributed (i.i.d.) from the unknown distribution. 
Given this assumption, data-driven approaches for discrepancy-based ambiguity sets  propose to choose the level of robustness %in two systematic ways.
%One common approach is to choose the level of robustness 
by analyzing the discrepancy---with respect to some metric---between the empirical distribution and the true unknown distribution\footnote{Some probability metrics, such as Wasserstein metric, metrize the weak convergence \citep{gibbs2002}. That is, the convergence between two probability distributions, with respect to some metric, implies the convergence in probability.}, asymptotically, see, e.g.,  \citet{ben2013,shafieezadeh2015},   or with a finite sample, see, e.g., \citet{pflug2007}.   %(1) the asymptotic convergence  of the empirical distribution to the true unknown distribution, with respect to some metric, see, e.g., \citet{shafieezadeh2015}, and/or (2) finite-sample probabilistic guarantee  on the distance between the empirical distribution and the true unknown distribution, with respect to some metric, see, e.g., \citet{pflug2007}. 
A direct consequence of such  analysis is that it establishes a finite-sample probabilistic guarantee  on the discrepancy between the empirical distribution and the true unknown distribution. Hence, it gives rise  to a probabilistic guarantee on the inclusion of the  unknown distribution in the constructed set, with respect to the empirical distribution. By construction, such an ambiguity set can be interpreted as a confidence set on the true unknown distribution. Moreover, such a construction implies a finite-sample guarantee  on  the out-of-sample performance, so that the current optimal value provides an upper bound on the out-of-sample performance of the current solution with a high probability. 
A similar idea is used in moment-based ambiguity sets, see, e.g., \citet{goldfarb2003} and \citet{delage2010}. 
In a recent work, \citet{gotoh2017} propose to choose the level of robustness by trading off between the mean and variance of the out-of-sample objective function value. We refer the readers to that paper for a review of calibration approaches in \dro.

%pro  on the large-sample analysis of the corresponding distance, so that the unknown distribution belongs to the constructed set with a high probability; see, e.g., . The second approach, on the other hand, propose to choose the level of robustness based on the finite-sample analysis of the out-of-sample performance,  see, e.g.,  %; thus, they  benefit a finite-sample probabilistic guarantee.

Below, we review the data-driven approaches to choose the level of robustness in more details. 
In this section, we suppose that a set  $\{\bs{\xi}^{i}\}_{i=1}^{N}$ of i.i.d\ data, distributed according to $\trueP$, is available, where $\Ts{P}_{N}$ denotes the empirical probability distribution of data. 

\paragraph{Optimal Transport Discrepancy}
When the ambiguity set contains all discrete distributions around the empirical distribution in the sense of the Wasserstein metric,  \citet{pflug2007} and \citet{pflug2012} propose to choose the level of robustness based on a probabilistic statement on the Wasserstein metric between the empirical and true distributions, due to \citet{dudley1969}, as $\epsilon=\frac{C N^{-\frac{1}{d}}}{\alpha}$. %, and in conjunction with the Markov's inequality.
This choice of $\epsilon$ guarantees that $\Ts{P}\{ \Fs{d}^{\text{W}}_{c} (\Ts{P}, \Ts{P}_{N}) \ge \epsilon\} \le \alpha$. 
In addition to the confidence level $1-\alpha$ and  the number of available data points $N$, the proposed level of robustness in \citep{pflug2007,pflug2012}  depends on the dimension of $\txi$, $d$, and a constant $C$.
For such a Wasserstein-based ambiguity set, one can also choose the size of the set  by utilizing the probabilistic statement  on the discrepancy between  empirical distribution and the true unknown distribution, established in \citet{fournier2015}. 
Nevertheless, because all the utilized probabilistic statements rely on the exogenous constant $C$, the size of the ambiguity set calculated from the theoretical analysis may be very conservative; hence,  such proposals are not practical.

By acknowledging the issue  raised above, some researchers propose to choose the level of robustness without relying on exogenous constants. 
For cases that the ambiguity set contains all discrete  distributions, supported on a compact space and around the empirical distribution,  \citet{ji2017} derive a closed-form expression for computing the size of the Wasserstein-based ambiguity set. 
\begin{theorem}{(\citet[Theorem~2]{ji2017})}
    \label{thm: rev.Was.epsilon}
    Suppose that the random vector $\txi$ is supported on a finite Polish space $(\Omega, d)$, where $\Omega \subseteq \Bs{R}^{d}$ and $d$ is the $\ell_{1}$-norm. Choose $c(\cdot, \cdot)=d(\cdot, \cdot)$ in the definition of the optimal transport discrepancy \eqref{eq: rev.opt_transport}. Assume that  
    $$\log \int_{\Omega} e^{\lambda d(\bs{\xi},\bs{\xi}_{0})} \trueP(d \bs{\xi}) < \infty, \quad \forall \lambda >0,$$
    for some $\bs{\xi}_{0}$. 
    Let $\theta:=\sup\{d(\bs{\xi}_{1}, \bs{\xi}_{2}): \bs{\xi}_{1}, \bs{\xi}_{2} \in \Omega\}$ be the diameter of $\Omega$. 
    Then, $$\Ts{P}_{N}\{\Fs{d}^{\text{W}}_{d} (\trueP, \Ts{P}_{N}) \le \epsilon \} \ge 1- \exp\Big\{- N \Big( \frac{\sqrt{4 \epsilon (4 \theta +3) + (4 \theta +3 )^{2} }}{4 \theta + 3 }-1\Big)^{2}\Big\}.$$
    Moreover, if 
    $$\epsilon \ge \Big( \theta + \frac{3}{4}\Big) \Big(-\frac{1}{N} \log \alpha + 2 \sqrt{-\frac{1}{N} \log \alpha}\Big),$$
    then 
    $$\Ts{P}_{N}\{\Fs{d}^{\text{W}}_{d} (\trueP, \Ts{P}_{N}) \le \epsilon \} \ge 1-\alpha.$$
\end{theorem}
Unlike the result in \citet{pflug2007}, the proposed level of robustness in  \citet{ji2017}, stated in Theorem \ref{thm: rev.Was.epsilon}, depends  only  on the confidence level $\alpha$, the number of available data points, and the diameter of the compact support $\Omega$.  \citet{ji2017} obtain this result by bounding  the Wasserstein distance between two probability distributions from  above, using the properties of the weighted total variation \citep{bolley2005}, and the weighted Csiszar-Kullback-Pinsker inequality \citep{villani2008}, and consequently applying Sanov's large deviation theorem \citep{dembo1998} to reach a probabilistic statement on the Wasserstein distance between  two distributions.
As stated in Theorem \ref{thm: rev.Was.epsilon}, such a result guarantees that the constructed set  contains the unknown probability distribution  with a high probability. Moreover, it implies a probabilistic guarantee on the true optimal value. 

Another criticism of methods such as those proposed in \citet{pflug2007} and \citet{pflug2012} is that they merely rely on the discrepancy between two probability distributions, and the optimization framework plays no role in the prescription. By making connection between the regularizer parameter and the size of the ambiguity for Wassersetin-based sets, \citet{blanchet2016robust}  aim to optimally  choose the regularization parameter. A key component of their analysis is a {\it robust Wasserstein profile} (RWP) function. At a given solution $\bs{x}$, this function calculates the minimum Wasserstein distance from the nominal distribution to the set of optimal probability distributions for the inner problem at $\bs{x}$. For any confidence level $\alpha$, they show that the size of the ambiguity set should be chosen as $(1-\alpha)$-quantile of RWP at the optimal solution to the minimization problem under the true unknown distribution. Using this selection of  $\epsilon$, the optimal solution to the true problem belongs to the set of optimal solutions to the \dro\ problem, with $(1-\alpha)$ confidence for all $\Ts{P} \in \Cs{P}^{\text{W}}(\Ts{P}_{N},\epsilon)$. As such a result is based on the true optimal solution, they study the asymptotic behavior of the RWP function and discuss how to use it to optimally  choose the regularization parameter without cross validation.  
The  work in \citet{blanchet2016robust} is extended in \citet{blanchet2017groupwise,blanchet2016SOS}. \citet{blanchet2017groupwise} utilize  the RWP function to introduce a data-driven (statistical) criterion for the optimal choice of the regularization parameter and study its asymptotic behavior. 
For  a \dro\ approach to linear regression, \citet{chen2018regression}  give guidance on the selection of the regularization parameter from the standpoint of a confidence region. 

\paragraph{Goodness-of-Fit Test}
\citet{bertsimas2018RO} propose to form the ambiguity set of distributions using the  confidence set of the unknown distribution via goodness-of-fit tests. With such an approach, one chooses the level of robustness as the threshold value of the corresponding test, depending on the confidence level $\alpha$, data, and the null hypothesis.

\paragraph[Phi-Divergences]{\texorpdfstring{$\phi$-Divergences}{Phi-Divergences}}

By noting that the class of $\phi$-divergences can be used in statistical hypothesis tests, a similar approach to the one in \citet{bertsimas2018RO} can be used to choose the level of robustness for $\phi$-divergence-based ambiguity sets. For the case that the distributional ambiguity in discrete distributions  is modeled via $\phi$-divergences, some papers propose to choose the level of robustness by relying on the asymptotic behavior of the discrepancy between  the empirical distribution and true unknown distribution, see, e.g., \citet{ben2013,bayraksan2015,yanikoglu2012}. 

Suppose that $\Xi$ is finite sample space of size $m$ and the $\phi$-divergence function in \eqref{eq: rev.phi_set} is twice continuously differentiable in a neighborhood of $1$, with  $\phi^{\prime\prime}(1)>0$. Then, it is shown in \citet{pardo2005} that under the true distribution, the statistics $\frac{2N}{\phi^{\prime\prime}(1)}\Cs{D}_{\phi}(\trueP,\nomP)$ converges in distribution to a $\chi^{2}_{m-1}$-distribution, with $m-1$ degrees of freedom.
Thus, at a given confidence level $\alpha$, one can set the level of robustness to $\frac{\phi^{\prime\prime}(1)}{2N}\chi^{2}_{m-1, 1-\alpha}$, where $\chi^{2}_{m-1, 1-\alpha}$ is the $(1-\alpha)$-quantile of  $\chi^{2}_{m-1}$, to obtain an (approximate) confidence
set on the true unknown distribution.
\citet{ben2013} show that such a choice of the level of robustness gives a  one-sided confidence
interval with (asymptotically) inexact coverage on the true optimal value of $\inf_{\bs{x} \in \Cs{X}} \ \ee{\trueP}{h(\bs{x},\txi)} $. 
For corrections for small sample sizes, we refer readers to \citet{pardo2005}.

By generalizing the empirical likelihood framework \citep{owen2001empirical} on a separable metric space (not necessarily finite), \citet{duchi2016} propose to choose the level of robustness $\epsilon$ such that a confidence interval $[l_{N}, u_{N}]$ on the true optimal value of $\inf_{\bs{x} \in \Cs{X}} \ \ee{\trueP}{h(\bs{x},\txi)} $ has an asymptotically exact coverage $1-\alpha$, i.e., \linebreak $\lim_{N \rightarrow \infty} \Ts{P}_{N}\{\inf_{\bs{x} \in \Cs{X}} \ee{\trueP}{h(\bs{x},\txi)} \in [l_{N},u_{N}] \}=1-\alpha$, where 
$$u_{N}:= \inf_{\bs{x} \in \Cs{X}} \ \sup_{\Ts{P} \in \Cs{P}^{\phi}(\Ts{P}_{N}; \epsilon) } \ \ee{\Ts{P}}{h(\bs{x},\txi)}, $$
$$l_{N}:= \inf_{\bs{x} \in \Cs{X}} \ \inf_{\Ts{P} \in \Cs{P}^{\phi}(\Ts{P}_{N}; \epsilon) } \ \ee{\Ts{P}}{h(\bs{x},\txi)}, $$
and 
$$\Cs{P}^{\phi}(\Ts{P}_{N}; \epsilon):= \sset*{\Ts{P} \in \P}{ \Cs{D}_{\phi}(\Ts{P} \| \Ts{P}_{N}) \le \epsilon}.$$ 
%The general idea of the work in \citet{duchi2016}  .  

\begin{theorem}{(\citet[Theorem~4]{duchi2016})}
    \label{thm: rev.phi.epsilon}
    Suppose that the $\phi$ function is three time continuously differentiable in a neighborhood of $1$, and normalized with $\phi(1)=\phi^{\prime}(1)=0$\footnote{As in the definition of $\phi$-divergence, the assumptions $\phi(1)=\phi^{\prime}(1)=0$ are without loss of generality because the function $\psi(t)= \phi(t)- \phi^{\prime}(1) (t-1)$ yields  identical discrepancy measure to $\phi$ \cite{pardo2005}} and  $\phi^{\prime\prime}(1)=2$. 
    Furthermore, suppose that $\Cs{X}$ is compact, there exists a measurable function $M: \Omega \mapsto \Bs{R}_{+}$ such that for all $\bs{\xi} \in \Omega$, $h(\cdot, \bs{\xi})$ is $M(\bs{\xi})$-Lipschitz with respect to some norm $\|\cdot\|$ on $\Cs{X}$,  $\ee{\trueP}{M(\txi)^{2}}< \infty$, and $\ee{\trueP}{|h(\bs{x}_{0}, \txi)|}<\infty$ for some  $\bs{x}_{0} \in \Cs{X}$. Additionally, suppose that $h(\cdot, \bs{\xi})$ is proper and lower semicontinuous for $\bs{\xi}$, $\trueP$-almost surely. %If $\{\bs{\xi}^{i}\}_{i=1}^{N}$ is a set of i.i.d\ data distributed according to $\trueP$ and 
    If $\inf_{\bs{x} \in \Cs{X}} \ \ee{\trueP}{h(\bs{x},\txi)} $ has a unique solution, then
    $$\lim_{n \rightarrow \infty} \Ts{P}_{N}\{\inf_{\bs{x} \in \Cs{X}} \ee{\trueP}{h(\bs{x},\txi)} \le u_{N} \}= 1- \frac{1}{2} P(\chi^{2}_{1} \ge N \epsilon)$$
    and
    $$\lim_{n \rightarrow \infty} \Ts{P}_{N}\{\inf_{\bs{x} \in \Cs{X}} \ee{\trueP}{h(\bs{x},\txi)} \ge l_{N} \}= 1- \frac{1}{2} P(\chi^{2}_{1} \ge N \epsilon).$$
\end{theorem}
According to Theorem \ref{thm: rev.phi.epsilon}, if  $\inf_{\bs{x} \in \Cs{X}} \ \ee{\trueP}{h(\bs{x},\txi)} $ has a unique solution, the desired asymptotic guarantee is achieved  with the choice $\epsilon=\frac{\chi^{2}_{1, 1-\alpha}}{N}$. 
\citet{duchi2016} also give rates at which $u_{N} - l_{N} \rightarrow 0$. 
Moreover, the upper confidence interval $(-\infty, u_N]$ is a one-sided confidence interval with an asymptotic exact coverage when $\epsilon=\chi^{2}_{1, 1-2\alpha}$.  

On another note, it can be seen from Table \ref{T: rev.phi} that the $\phi$-divergence function corresponding to the variation distance is not twice differentiable at 1. Hence, one cannot use the above result. However, by utilizing the first inequality in Lemma \ref{lem: rev.TV}, i.e., the relationship between the variation distance and the Hellinger distance, \citet{jiang2018} propose to set the level of robustness to $\sqrt{\frac{1}{N}\chi^{2}_{m-1, 1-\alpha}}$ in order to obtain  an (approximate) confidence
set on the true unknown discrete distribution. The proposed choice of the level of robustness ensures that the unknown discrete distribution belongs to the ambiguity set with a high probability.  For the case that $\txi$ follows a  continuous distribution, the proposed level of robustness in \cite{jiang2018} depends on some constants that appear in the probabilistic statement of the discrepancy between the empirical distributions and the true distribution.  

\paragraph[Lp-Norm]{\texorpdfstring{$\ell_{p}$-Norm}{Lp-Norm}}

For the case that  $\ell_{\infty}$-norm  is used to model the distributional ambiguity, \citet{jiang2018} propose to choose the level of robustness based on a probabilistic statement on the discrepancy between the empirical distributions and the true distribution as $\epsilon=\frac{z_{1-\frac{\alpha}{2}}}{\sqrt{N}}\max_{i=1}^{m} \ \sqrt{\nomp^{i}(1-\nomp^{i})}$, where $z_{1-\frac{\alpha}{2}}$ represents the $(1-\frac{\alpha}{2})$-quantile of the standard normal distribution, and $\bs{p}_{0}:=[\nomp^{1}, \ldots, \nomp^{m}]$ denotes the empirical distribution of data. 
The proposed choice of the level of robustness ensures that the unknown discrete distribution belongs to the ambiguity set with a high probability.  Similar to the $\ell_{1}$-norm (i.e., the variation distance) case, when $\txi$ follows a  continuous distribution, the proposed level of robustness depends on some constants that appear in the probabilistic statement of the discrepancy between the empirical distributions and the true distribution.

\paragraph[Zeta-Structure]{\texorpdfstring{$\zeta$-Structure}{Zeta-Structure}}

By exploiting the relationship between different metrics in the $\zeta$-structure family, see, e.g., Lemma \ref{lem: rev.zeta}, \citet{zhao2015} provide guidelines on how to choose the level of robustness for the ambiguity sets of the unknown discrete distribution formed via bounded Lipschitz, Kantorovich, and Fortet-Mourier metrics as follows. 
\begin{theorem}
    \label{thm: rev.zeta.epsilon}
    Suppose that the random vector $\txi$ is supported on a bounded finite  space $\Omega$ and $\theta$ denotes the diameter of $\Omega$, as defined in Theorem \ref{thm: rev.Was.epsilon}. 
    \begin{enumerate}[label=(\roman*)]
        \item if $\epsilon \ge \theta \sqrt{-2 \frac{\log \alpha}{N} }$, then 
    $\Ts{P}_{N}\{\Fs{d}^{\text{K}} (\trueP, \Ts{P}_{N}) \le \epsilon \} \ge 1-\alpha$ and \linebreak $\Ts{P}_{N}\{\Fs{d}^{\text{BL}} (\trueP, \Ts{P}_{N}) \le \epsilon \} \ge 1-\alpha$. 
    
        \item if $\epsilon \ge \theta \max\{1, \theta^{q-1}\} \sqrt{-2 \frac{\log \alpha}{N} }$, then 
    $\Ts{P}_{N}\{\Fs{d}^{\text{FM}} (\trueP, \Ts{P}_{N}) \le \epsilon \} \ge 1-\alpha$. 
    \end{enumerate}
\end{theorem}

\begin{proof}
    The proof is immediate from the relationship between $\zeta$-structure metrics, stated in Lemma \ref{lem: rev.zeta}, and the fact that \linebreak $\Ts{P}_{N}\{\Fs{d}^{\text{K}} (\trueP, \Ts{P}_{N}) \le \epsilon \} \ge 1- \exp\{-\frac{\epsilon^{2}N}{2 \theta^{2}}\}$ due to \citet[Proposition~3]{zhao2015}. 
\end{proof}
As it can be seen from Theorem \ref{thm: rev.zeta.epsilon}, the proposed levels of robustness for the case that the unknown distribution is discrete depend on the diameter of $\Omega$, the number of data points $N$, and the confidence level $1-\alpha$.
However, the results in \citet{zhao2015} for the continuous case suffer from similar practical issues as in  \citep{pflug2007,pflug2012,jiang2018}. 

\paragraph{Chebyshev}

A data-driven approach to construct a Chebyshev ambiguity set is proposed in \citet{goldfarb2003}. Recall the linear model for  the asset returns $\txi$ in \citet{goldfarb2003}: $\txi=\bs{\mu} + \bs{A} \tbs{f} + \tbs{\epsilon}$, where $\bs{\mu}$ is the vector of mean returns, $\tbs{f} \sim N(\bs{0}, \bs{\Sigma})$ is the vector of random returns that derives the market, $\bs{A}$ is the factor loading matrix, and $\tbs{\epsilon} \sim N(\bs{0}, \bs{B})$ is the vector of residual returns with a diagonal matrix $\bs{D}$. 
%\citet{goldfarb2003} study three different cases: (i) $\bs{\mu}$  lies in the interval uncertainty set $\Cs{U}_{\bs{\mu}}$, defined in \eqref{eq: mu}, (ii) $\bs{A}$ belongs to the ellipsoid uncertainty set $\Cs{U}_{\bs{A}}$, defined in \eqref{eq: A}, and   (iii) diagonal elements of $\bs{\bs{B}}$ lie in an interval uncertainty set $\Cs{U}_{\bs{B}}$, defined in \eqref{Eq: B}. 
Under the assumption that the covariance matrix $\bs{\Sigma}$ is known, recall that \citet{goldfarb2003} study three different models to form the uncertainty in $\bs{B}$, $\bs{A}$, and $\bs{\mu}$ as follows:
\begin{align*}
    & \Cs{U}_{\bs{B}}=\sset*{\bs{B}}{\bs{B}=\text{diag}(\bs{b}), \; b_{i} \in [\ul{b}_{i}, \ol{b}_{i}], \; i=1, \ldots, d},\\
    & \Cs{U}_{\bs{A}}=\sset*{\bs{A}}{\bs{A}=\bs{A}_{0}+ \bs{C}, \; \|\bs{c}_{i}\|_{g} \le \rho_{i}, \; i=1, \ldots, d},\\
    & \Cs{U}_{\bs{\mu}}=\sset*{\bs{\mu}}{\bs{\mu}=\bs{\mu}_{0}+ \bs{\zeta}, \; |\zeta_{i}| \le \gamma_{i}, \; i=1, \ldots, d},
\end{align*}
where $\bs{c}_{i}$ denotes the $i$-th column of $\bs{C}$, and $\|\bs{c}_{i}\|_{g}=\sqrt{\bs{c}_{i}^{\top} \bs{G} \bs{c}_{i}^{\top}}$ denotes the elliptic norm of $\bs{c}_{i}$ with respect to a symmetric positive definite matrix $\bs{G}$. 
Calibrating the uncertainty sets $\Cs{U}_{\bs{B}}$, $\Cs{U}_{\bs{A}}$, and $\Cs{U}_{\bs{\mu}}$ involves choosing parameters $\ul{d}_{i}$, $\ol{d}_{i}$, $\rho_{i}$, $\gamma_{i}$, $i=1, \ldots, d$, vector $\bs{\mu}_{0}$, and matrices $\bs{A}_{0}$ and $\bs{G}$. 
Assuming that a set of data points is available on $\txi$ and $\tbs{f}$, by relying on the multivariate linear regression, \citet{goldfarb2003} obtain  least square estimates $(\bs{\mu}_{0},\bs{A}_{0})$ of $(\bs{\mu},\bs{A})$, respectively, and construct a multdimensional  confidence region of  $(\bs{\mu},\bs{A})$ around  $(\bs{\mu}_{0},\bs{A}_{0})$. Now, projecting this confidence region along vector $\bs{A}$ and matrix $\bs{\mu}$ gives the corresponding uncertainty sets $\Cs{U}_{\bs{A}}$ and $\Cs{U}_{\bs{\mu}}$, respectively. To form the uncertainty set $\Cs{U}_{\bs{B}}$, they propose to use a bootstrap confidence interval around the regression error of the residual. 
%the model parameters in order to provide a probabilistic guarantee on the  out-of-sample performance of the current solution.

\paragraph{Delage and Ye}

Data-driven methods to construct the ambiguity set $\Cs{P}^{\text{DY}}$ is proposed in \citet{delage2010}. 

\begin{theorem}{(\citet[Corollary~4]{delage2010})}
    \label{thm: rev.DelageYe.calibrration}
    Suppose that the random vector $\txi$ is supported on a bounded  space $\Omega$. Consider the following parameters:
    \begin{align*}
        & \hat{\bs{\mu}}_{0}=\frac{1}{N} \sum_{i=1}^{N} \bs{\xi}^{i},\\
        & \hat{\bs{\Sigma}}_{0}=\frac{1}{N-1}  \sum_{i=1}^{N} (\bs{\xi}^{i}- \hat{\bs{\mu}}_{0})(\bs{\xi}^{i}- \hat{\bs{\mu}}_{0})^{\top},\\
        & \hat{\theta}= \sup_{i=1}^{N} \|\hat{\bs{\Sigma}}^{-\frac{1}{2}} (\bs{\xi}^{i}- \hat{\bs{\mu}}_{0})\|_{2},
    \end{align*}
    where $\hat{\bs{\mu}}_{0}$, $\hat{\bs{\Sigma}}_{0}$, and $\hat{\theta}$ are  estimates of the mean, covariance, and diameter of the support of $\txi$, respectively. 
    Moreover, for a confidence level $1-\alpha$, let us define
    \begin{align*}
        & \bar{\theta}=\Big( 1- (\hat{\theta}^{2}+2) \frac{2 + \sqrt{2 \log (\frac{4}{\bar{\alpha}})}}{\sqrt{N}}\Big)^{-\frac{1}{2}} \hat{\theta},\\
        & \bar{\gamma}_{1}= \frac{\bar{\theta}^{2}}{\sqrt{N}} \Big( \sqrt{1- \frac{d}{\bar{\theta}^{4}}} + \sqrt{\log{(\frac{4}{\bar{\alpha}})}} \Big)\\
        & \bar{\gamma}_{2}=\frac{\bar{\theta}^{2}}{N} \Big( 2 + \sqrt{2 \log{(\frac{2}{\bar{\alpha}})}}\Big),\\
        & \bar{\varrho}_{1}= \frac{\bar{\gamma}_{2}}{1- \bar{\gamma}_{1}- \bar{\gamma}_{2}},\\
        & \bar{\varrho}_{2}= \frac{1+ \bar{\gamma}_{2}}{1- \bar{\gamma}_{1}- \bar{\gamma}_{2}},
    \end{align*}
    where  $\bar{\alpha}=1- \sqrt{1-\alpha}$. 
    Let $\Cs{P}^{\text{DY}}(\Omega, \hat{\bs{\mu}}_{0}, \hat{\bs{\Sigma}}_{0}, \bar{\varrho}_{1}, \bar{\varrho}_{2})$ be the ambiguity set formed via \eqref{eq: rev.DelageYe}, using parameters $\hat{\bs{\mu}}_{0}$, $\hat{\bs{\Sigma}}_{0}$, $\bar{\varrho}_{1}$, and $\bar{\varrho}_{2}$. 
    Then, we have $$\Ts{P}_{N}\{\trueP \in \Cs{P}^{\text{DY}}(\Omega, \hat{\bs{\mu}}_{0}, \hat{\bs{\Sigma}}_{0}, \bar{\varrho}_{1}, \bar{\varrho}_{2})\} \ge 1- \alpha.$$
\end{theorem}

\subsubsection{Non-Data-Driven \dro s}	

As mentioned before, data-driven \dro s \linebreak typically assume that a set of i.i.d.\ sampled data is available from the unknown true distribution. 
In many situations, however, there is no guarantee that the future uncertainty is drawn from the same distribution.  
Recognizing this fact, some research is devoted to choosing the level of robustness in situations where the i.i.d.\ assumption is violated and data-driven methods to calibrate the level of robustness may be unsuitable.  

\citet{rahimian2019NV} use the notions of maximal effective  subsets and prices of optimism/pessimism and nominal/worst-case regrets to calibrate the level of robustness in discrepancy-based \dro\ models. Price of optimism pessimism is defined as the loss by being too optimistic (i.e., using SO model with the nominal distribution)---and hence, implementing the corresponding solution---while \dro\ accurately represents the ambiguity in the  distribution. Similarly, the price of pessimism is defined  as the loss by being too pessimistic (i.e., using RO model with no distributional information except for the support of uncertainty). 
Nominal/worst-case regret is defined as the loss of being unnecessarily ambiguous/not being ambiguous enough---and hence, implementing the corresponding solution---while \dro\ is ill-calibrated.
\citet{rahimian2019NV} suggest to balance the price of optimism and pessimism if the decision-maker is indifferent regarding the error from using too optimistic or pessimistic solutions. They refer to the smallest level of robustness for which such a balance happens as {\it indifferent-to-solution} level of robustness. On the other hand, \citet{rahimian2019NV} propose to balance the nominal and worst-case regrets if the decision-maker wants to be indifferent regarding the error from using an ill-calibrated \dro\ model in either the optimistic or the pessimistic scenarios. They refer to the smallest level of robustness for which such a balance happens as {\it indifferent-to-distribution} level of robustness.

% % % % % % % % % % % % % % % % % % % % % % % % % % % % % % % % % % % % % % % % % % % %	
\section{Cost Function of the Inner Problem}
\label{sec: rev.cost_inner}
% % % % % % % % % % % % % % % % % % % % % % % % % % % % % % % % % % % % % % % % % % % %	

Recall formulation \eqref{eq: DRO} and the functional $\Cs{R}_{P}: \Cs{Z} \mapsto \Bs{R}$. This functional accounts for quantifying the uncertainty in the outcomes of a fixed decision $\bs{x} \in \Cs{X}$ and  for a given fixed probability measure $P \in \Fs{M}\measurespace$. 
As pointed out before in Section \ref{sec: rev.generic_model} for \eqref{eq: SO_Obj} and \eqref{eq: SO_Cons}, one choice for this  functional is the expectation operator. Other functionals,  such as {\it regret function}, {\it risk measure}, and {\it utility function} have also been used in the \dro\ literature.
These functionals are closely related concepts and we refer to \citet{bental2007OCE} and \cite{rockafellar2015} for a comprehensive treatment and how one can induce one from the other. 
In this section, we review some notable works, where regret function, risk measure, and utility function are used to capture the uncertainty in the outcomes of the decision. 

%%%%%%%%%%%%%%%%%%%%%%%%%%%%%%%%%%%%%
\subsection{Regret Function}
\label{sec: rev.regret}
%%%%%%%%%%%%%%%%%%%%%%%%%%%%%%%%%%%%%

Given a decision $\bs{x} \in \Cs{X}$ and a probability measure $P \in \Fs{M}\measurespace$, a regret functional $\Cs{V}_{P}$ may quantify the expected displeasure or disappointment of the current decision with respect to a possible mix of future outcomes as follows: 
$$\rregret{P}{h(\bs{x},\txi)}:= \ee{P}{h(\bs{x}, \txi)- \min_{\bs{x} \in \Cs{X}} \  h(\bs{x}, \txi)}. $$
In other words, $\rregret{P}{h(\bs{x},\txi)}$  calculates the expected additional loss that could have been avoided. This definition of regret function is used in \citet{natarajan2014} and \citet{hu2011budget} in the context of combinatorial optimization and multicriteria decision-making, respectively. 
Another way for formulating a regret function may be as 
$$\rregret{P}{h(\bs{x},\txi)}:= \ee{P}{h(\bs{x}, \txi)}- \min_{\bs{x} \in \Cs{X}} \  \ee{P}{h(\bs{x}, \txi)}. $$
This type of regret function is used in \citet{perakis2008regret} in the context of the newsvendor problem. \citet{perakis2008regret} obtain closed form solutions to distributionally robust single-item newsvendor problems that minimize the worst-case expected regret of acting optimally, where only (1) support, (2) mean,  (3) mean and median, and (4) mean and variance information is available.  This  information can be captured with the ambiguity set $\Cs{P}^{\text{MM}}$, defined in \eqref{eq: moment-rob-set}. 
\citet{perakis2008regret} also study the ambiguity sets that preserve the shape of the distribution, including information on  (1) mean and symmetry, (2) support and unimodality with a given mode,  (3) median and unimodality with a given mode, and (4) mean, symmtery, and unimodality with a given mode.

%%%%%%%%%%%%%%%%%%%%%%%%%%%%%%%%%%%%%
\subsection{Risk Measure}
\label{sec: rev.risk}
%%%%%%%%%%%%%%%%%%%%%%%%%%%%%%%%%%%%%
As introduced in Section \ref{sec: rev.dro_coherent}, a functional that quantifies the uncertainty in the outcomes of a decision is a risk measure \citet{artzner1999,acerbi2002,kusuoka2001,shapiro2013kusuoka}. A  risk measure $\rho_{P}$ usually satisfies some {\it averseness} property, i.e., $\rro{P}{\cdot}>\ee{P}{\cdot}$ and imposes a preference order on random variables, i.e., if $Z, Z^{\prime} \in \Cs{Z}$ and $Z \ge Z^{\prime}$, then $\rro{P}{Z} \ge \rro{P}{Z^{\prime}}$.  Explicit incorporation of  a risk measure into a \dro\ model has also received attention  in the literature. 
We refer to \citet{pflug2012,pichler2013,wozabal2014,pichler2017}
for spectral and distortion risk measures, 
\citet{calafiore2007} for variance,  \citet{calafiore2007} for mean absolute-deviation, \citet{hanasusanto2016,wiesemann2014} for optimized certainty equivalent,  \citet{hanasusanto2015NV} for CVaR, and \citet{postek2016} for a variety of risk measures. 
%%%%%%%%%%%%%%%%%%%%%%%%%%%%%%%%%%%%%
\subsection{Utility Function}
\label{sec: rev.utility}
%%%%%%%%%%%%%%%%%%%%%%%%%%%%%%%%%%%%%

%\citet{natarajan2010PO}: portfolio optimization with a wide range of ambiguity, mostly Delage and Ye, but also Chebyshev with uncertainty

An alternative to using risk measures to compare random variables is to evaluate their expected utility \citet{gilboa1989}.  
As before, let us consider a probability space $\Pspace{P}$. A random variable $Z \in \Cs{Z}$ is preferred over a  random variable $Z^{\prime} \in \Cs{Z}$ if $\ee{P}{u(Z_{1})} \ge \ee{P}{u(Z_{2})}$ for a given  univariate utility function $u$\footnote{For definitions in a multivariate case, we refer to \citet{hu2012SAA,hu2014weighted}.}. A bounded utility function $u$ can be normalized to take values between $0$ and $1$, and hence, it can be interpreted as a cdf of a random variable $\zeta$, i.e., $u(t)=P\{\zeta \le t\}$ for $t \in \Bs{R}$. Under this interpretation, $Z$ is preferred over  $Z^{\prime}$ if $P\{Z \ge \zeta\} \ge P\{Z^{\prime} \ge \zeta\} $ because 
$$\ee{P}{u(Z)}=\ee{P}{P\{\zeta \le Z| Z\}}=\ee{P}{\ee{P}{\one{\zeta \le Z}|Z}}=\ee{P}{\one{\zeta \le Z}}=P\{\zeta \le Z\}.$$
However, as in decision theory, it is difficult to have a complete knowledge of a decision maker's preference (i.e., utility function), it is also difficult to have a complete knowledge of the cdf of $\zeta$. The notion of {\it stochastic dominance} handles this issue by comparing the expected utility of random variables, for a given family $\Cs{U}$ of utility functions, or equivalently, compare the probability of exceeding the target random variable $\zeta$ for a given family of cdf. Consequently,  to address the problem of ambiguity in decision maker's utility or equivalently, cdf of  the random variable $\zeta$, one can study   
\begin{equation}
\label{eq: Utility_Obj}
\min_{\bs{x} \in \Cs{X} } \ \max_{\zeta \in \Cs{U}} \ 	P\{h(\bs{x},\txi) \ge \zeta\},
\end{equation}
and
\begin{equation}
\label{eq: Utility_Cons}
\min_{\bs{x} \in \Cs{X} } \ \max_{\zeta \in \Cs{U}} \ 	\sset*{h(\bs{x})}{\max_{\zeta \in \Cs{U} } \ P\{\bs{g}(\bs{x},\txi) \ge \zeta\} \le \bs{0}},
\end{equation}
where $\Cs{U}$ denotes a given family of normalized and nondecreasing utility functions, or equivalently, a given family of  cdf. 
Note that  problems \eqref{eq: Utility_Obj} and \eqref{eq: Utility_Cons} have the form of problems \eqref{eq: DRO_Obj} and \eqref{eq: DRO_Cons}, respectively.
\citet{hu2015utility} study problem of the form  \eqref{eq: Utility_Obj}, where $\Cs{U}$ is further restricted to include concave utility functions or equivalently, cdf, and satisfy functional bounds on the utility and marginal utility functions (cdf  and pdf of $\zeta$)  as in \eqref{eq: moment-rob-set}. They provide a linear programming formulation of a particular case where the bounds on the
utility function are piecewise linear increasing concave functions, and the bounds on all other functions are step functions.
For the general continuous case, they study an approximation problem by discretisizing the continuous functions, and analyze the convergence properties of the approximated problem. They apply  their results to  a portfolio optimization problem. Unlike \citet{hu2015}, in \citet{hu2018}, no shape restrictions on the utility function is assumed  and only functional bounds on the utility function are enforced.
\citet{hu2018} show that an SAA approach to the Lagrangian dual of the resulting problem can be used while solving a  mixed-integer LP. They  study the convergence properties of this SAA problem, and illustrate  their results using examples in portfolio optimization and a streaming bandwidth allocation problem. 
\citet{bertsimas2010minmax} study a \dro\ model of the form  \eqref{eq: DRO_Obj}, where a convex nondecreasing disutility function is used  to quantify the uncertainty in decision. A utility function is closely related to risk measures \cite{hu2015utility}. For instance, for a given probability measure, the expected utility might have the form of a combination of expectation and expected excess beyond a target, or an optimized certainty equivalent risk measure. 
As shown in \citet{bental2007OCE}, under appropriate choices of utility functions,  an optimized certainty equivalent risk measure can be reduced to the mean-variance and the mean-CVaR formulations. 
 \citet{wiesemann2014}  study a \dro\ model formed via \eqref{eq: rev.WKS}, where the  decision maker is risk-averse via a nondecreasing convex piecewise affine disutility function. In particular, they investigate shortfall risk and optimized certainty equivalent risk measures.

Unlike the above discussion, many decision-making problems involve comparing random vectors. One can  generalize the notion of utility-based comparison to random vectors by using multivariate utility functions \cite{armbruster2015}. %However, such an approach is known to have computational and modeling limitations. 
Another approach to compare random vectors is based on the idea of the weighted scalarization of random vectors. For the case that the weights are deterministic and take value in an arbitrary set, we refer to \citet{dentcheva2009} for unrestricted sets, \citet{homemdemello2009,hu2011budget,hu2012MO} for polyhedral sets, and \citet{hu2012SAA} for convex sets. For instance,  \citet{hu2011budget} study a  weighted sum approach to a multiobjective budget allocation problem under uncertain performance indicators of projects. They assume that the weights take value in the convex hull of the weights suggested by experts and study a minmax approach to the expected  weighted sum problem, where the expectation is taken with respect to the uncertainty in the performance indicators and  the worst-case is taken with respect to the weights.  Note that the problem studied in \citet{hu2011budget} is in the framework of RO as the weights are deterministic. 

The idea of using stochastic weights,  governed by a probability measure that determines the relative importance of each vector of weights, is also introduced in \citet{hu2012MO} and \citet{hu2014weighted}. For instance,  \citet{hu2012MO} study a \dro\ approach to stochastically weighted multiobjective deterministic and stochastic optimization problems, where the weights are perturbed along different rays from a reference weight vector. They study the reformulations of the deterministic problem for the cases where the weights take values in (1) a polyhedral set, including those induced by a simplex, $\ell_{1}$-norm, and $\ell_{\infty}$-norm, and (2) a conic-representable set,   including those induced by a single cone (e.g.,  $\ell_{p}$-norm, ellipsoids), intersection of multiple cones, and union of multiple cones.  They further study the stochastic optimization problem. For the case that the weights and random parameters are independent, and the ambiguity in the probability distribution of weights is modeled via \eqref{eq: rev.DelageYe}, they obtain a reformulation of the problem using the result in \citet{delage2010}. For the case that the weights and random parameters are dependent, they also obtain reformulations of the resulting problem by utilizing the result from the deterministic case. 
They illustrate the ideas set forth in the paper using examples from disaster planning  and  agriculture revenue management problems.  

% % % % % % % % % % % % % % % % % % % % % % % % % % % % % % % % % % % % % % % % % % % %	
%\section{Price of Robustness}
% % % % % % % % % % % % % % % % % % % % % % % % % % % % % % % % % % % % % % % % % % % %	

% % % % % % % % % % % % % % % % % % % % % % % % % % % % % % % % % % % % % % % % % % % %	
\section{Modeling Toolboxes}
\label{sec: rev.toolboxes}
% % % % % % % % % % % % % % % % % % % % % % % % % % % % % % % % % % % % % % % % % % % %

\citet{goh2011} develop a MATLAB-based algebraic modeling toolbox, named ROME, for a class of \dro\ problems with conic-representable sets for the support and mean, known covariance matrix, and  upper bounds on the directional
deviations  studied in \citet{goh2010tractable}. \citet{goh2011} elucidate the practicability of this toolbox in the context of  (1) a service-constrained
inventory management problem, (2) a project-crashing problem, and (3) a portfolio optimization problem.
A C++-based algebraic modeling package, named ROC,  is developed in \citet{bertsimas2018adaptiveDRO}, to demonstrate the practicability and scalability of the studied  adaptive \dro\ model. Some features of ROC include  declaration of uncertain parameters and linear decision
rules, transcriptions of ambiguity sets, and reformulation of \dro\ using the results obtained in   \citet{bertsimas2018adaptiveDRO}. A brief introduction to ROC and some illustrative examples to declare the objects of a model, such as  variables, constraints, ambiguity set, among others,  are given in an early version of  \citet{bertsimas2014practicable}.
XProg (\url{http://xprog.weebly.com}), is a MATLAB-based algebraic modeling
package that also implements the proposed model in  \citet{bertsimas2018adaptiveDRO}. 
\citet{chen2018adaptive}  develop an algebraic modeling package, AROMA, to  illustrate the modeling power of their proposed ambiguity set.

\newpage
{\small  %\bibliographystyle{plainnat} 
\bibliographystyle{bibbst} 
	\setlength{\baselineskip}{0pt}
	\bibliography{bibfile_abv}  
}

\end{document}